\documentclass[12pt, centertags,reqno]{amsart}
\usepackage{graphicx}
\usepackage{amsmath, amstext, amsthm, a4, amssymb, amscd}
\usepackage{amssymb, graphics, epsfig, color}
\usepackage{epstopdf}\usepackage{enumerate}
\usepackage[mathscr]{eucal}
\usepackage[all]{xy}
\numberwithin{equation}{section}
\usepackage[a4paper,width=15.8cm,top=2.5cm,bottom=4.5cm]{geometry}

\newtheorem{thm}{Theorem}[section]
\newtheorem{lemma}[thm]{Lemma}

\newtheorem{prop}[thm]{Proposition}

\theoremstyle{definition}
\newtheorem{rem}[thm]{Remark}
\theoremstyle{definition}
\newtheorem{defn}[thm]{Definition}
\theoremstyle{remark}

\newcommand{\field}[1]{\mathbb{#1}}
 \newcommand{\Z}{\field{Z}} \newcommand{\R}{\field{R}} \newcommand{\C}{\field{C}} \newcommand{\N}{\field{N}}
  \newcommand{\K}{\field{K}}
  \newcommand{\abs}[1]{\lvert#1\rvert}
 
\newcommand{\be}{\begin{align}\begin{aligned}}
\newcommand{\cliffu}{c({\scriptstyle\frac{\partial}{\partial u}})}
\newcommand{\hatcliffu}{\hat{c}({\scriptstyle\frac{\partial}{\partial u}})}
\newcommand{\hh}{\mathscr{H}^\bullet}
\newcommand{\hhe}{\mathscr{H}}
\newcommand{\End}{{\rm End}}
\newcommand{\hhl}{{\mathscr{H}^\bullet_{L^2}}}
\newcommand{\Sp}{{\rm Sp}}
\newcommand{\LL}{\mathscr{L}}
\newcommand{\cntrtu}{i_{\scriptstyle\frac{\partial}{\partial u}}}
\newcommand{\hhar}{\mathscr{H}^\bullet_{\rm abs/rel}}
\newcommand{\n}{\nabla}
\newcommand{\smooth}{\mathscr{C}^\infty}
\newcommand{\derpar}[2]{\frac{\partial #1}{\partial #2}}
\newcommand{\e}{\varepsilon}
\newcommand{\continu}{\mathscr{C}^0}
\newcommand{\im}{{\rm Im}}
\newcommand{\bd}{{\rm bd}}
\newcommand{\hhel}{\mathscr{H}_{L^2}}

\DeclareMathOperator{\Ker}{Ker}  
 \DeclareMathOperator{\rk}{rk} \DeclareMathOperator{\Id}{Id} 
\DeclareMathOperator{\tr}{Tr}

\begin{document}

\title{Scattering matrix and analytic torsion}
\date{\today}

\author{Martin PUCHOL}
\address{D{\'e}partement de Math{\'e}matiques,
B{\^a}timent IMO, Bureau 2G3,
Facult{\'e} des Sciences d'Orsay, Universit{\'e} Paris-Sud,
F-91045 Orsay Cedex}
\email{martin.puchol@math.cnrs.fr}

\author{Yeping ZHANG}
\address{Department of Mathematics,
Faculty of Science,
Kyoto University,
Kitashirakawa Oiwake-cho, Sakyo-ku,
Kyoto 606-8502, Japan}
\email{yp.zhang@math.kyoto-u.ac.jp}

\author{Jialin Zhu}
\address{Shanghai Center for Mathematical Sciences, Fudan University, Shanghai 200433, P.R. China}
\email{jialinzhu@fudan.edu.cn}

\begin{abstract}
We consider a compact manifold with a piece isometric to a (finite length) cylinder.
By making the length of the cylinder tend to infinity,
we obtain an asymptotic gluing formula for the zeta determinant of the Hodge-Laplacian and
an asymptotic expansion of the $L^2$-torsion of the corresponding Mayer--Vietoris exact sequence.
As an application,
we give a purely analytic proof of the gluing formula for analytic torsion.


\noindent Classification: global analysis, analysis on manifolds.
\end{abstract}

\maketitle
\tableofcontents
\setcounter{section}{-1}

\section{Introduction}\label{s.0}
We consider
a flat vector bundle $F$
equipped with a Hermitian metric $h^F$
over a compact Riemannnian manifold $Z$.
The associated Ray-Singer analytic torsion \cite{RaySing71}
is a weighted product
of the zeta determinants of the Hodge-Laplacian
on the vector space of $p$-forms with values in $F$,
$\Omega^p(Z,F)$ ($p=0,\cdots,\dim Z$).
The Ray-Singer metric \cite{BZ92} on $\det H^\bullet(Z,F)$
is the product of the Ray-Singer analytic torsion
and the $L^2$-metric (induced by the Hodge theory).
The Ray-Singer metric has a topological counterpart, the Reidemeister metric \cite{Reid35}.
Ray and Singer \cite{RaySing71} conjectured that the two metrics coincide
for unitary flat vector bundle.
Cheeger \cite{Chg} and M{\"u}ller \cite{Mu78} independently proved the conjecture.
This result is now called Cheeger-M{\"u}ller theorem.
Bismut, Zhang and M{\"u}ller simultaneously considered the extension of this theorem.
M{\"u}ller \cite{Mu93} extended the theorem to unimodular flat bundles,
i.e., $\det F$ is unitary.
Bismut and Zhang \cite{BZ92} extended the theorem to arbitrary flat vector bundles.
There are also various extensions to equivariant cases \cite{LoRo91,Luck93,BisZh94}.

Now we suppose that there is a hypersurface $Y \subseteq Z$
cutting $Z$ into two pieces,
which we denote by $Z_1$ and $Z_2$.
Ray and Singer \cite{RaySing71} proposed that
there should be a gluing formula
expressing the analytic torsion of $Z$
in terms of the analytic torsions of $Z_1$ and $Z_2$.
They proposed such a formula as an intermediate step towards the Ray-Singer conjecture.
Dramatically, the conjecture was proved by other means,
and the gluing formula was proved as a consequence of the conjecture.
L{\"u}ck \cite{Luck93} proved the gluing formula for unitarily flat vector bundles.
The proof is based on the Cheeger-M{\"u}ller theorem and \cite{LoRo91}.
Br{\"u}ning and Ma \cite{BruMa12} proved the gluing formula for arbitrary flat vector bundles.
The proof is based on
an equivariant extension of the Cheeger-M{\"u}ller theorem
by Bismut and Zhang \cite{BisZh94}.
There are also purely analytic proofs,
i.e., without using Cheeger-M{\"u}ller theorem or its extension.
Vishik \cite{Vishik95} found an analytic proof for unitarily flat vector bundles.
There are also related works by Hassell \cite{Has} and Lesch \cite{Lesch13}.

Both the Ray-Singer analytic torsion and the Reidemeister torsion
were generalized to the relative context,
i.e., a flat vector bundle over a smooth fibration.
Bismut and Lott \cite{BL} constructed analytic torsion forms (Bismut-Lott torsion).
Igusa \cite{Igusa02} and Dwyer, Weiss,  Williams \cite{DwWeiWill03}
independently constructed higher Reidemeister torsions,
which turned out to be equivalent \cite{BDKW11}
up to a universal cohomology class.
Bismut and Goette  \cite{BGo} proved
a higher version of Cheeger-M{\"u}ller/Bismut-Zhang theorem
under the condition that there exists a fiberwise Morse function.
However, the higher version of Cheeger-M{\"u}ller/Bismut-Zhang theorem in general remains unknown.
Igusa \cite{Igu08} axiomatized the higher torsion invariants.
Two torsion invariants satisfying the axiomatization are equivalent
up to a universal cohomology class.
Igusa also showed that higher Reidemeister torsion satisfies his axiomatization.
It was conjectured that
there should exist a gluing formula for Bismut-Lott torsion
\footnote{conference on the higher torsion invariants, G{\"o}ttingen, September 2003}.
The gluing formula conjectured together with the work of Ma \cite{Ma02}
could show that Bismut-Lott torsion satisfies Igusa's axiomatization
(cf. \cite{Goette09} for a survey).
Zhu \cite{Zhu15} proved the gluing formula for Bismut-Lott torsion
under the condition that there exists a fiberwise Morse function.
The proof is based on the work of Bismut and Goette \cite{BGo},
which we mentioned earlier.
Zhu \cite{Zhu16} also proved the gluing formula for Bismut-Lott torsion
under the condition that $H^\bullet(Y,F)=0$.
This proof is analytical.

In this paper,
we give an analytic proof of the gluing formula
for Ray-Singer torsion.
Our approach would eventually lead to a proof of the gluing formula
for Bismut-Lott torsion in full generality \cite{PZZ-II}.
Our proof is based on the adiabatic limit.
We assume that the Riemannian metric is product on a tubular neighborhood of $Y\subseteq Z$,
which we identify with a cylinder $(-1,1)\times Y$.
In the adiabatic limit,
we deform the Riemannian metric
such that the length of the cylinder tends to infinity.

Note that the study of the adiabatic limit of $\eta$-invariant started first
in Bismut-Freed \cite{BF1,BF2}, Bismut-Cheeger \cite{BC}
by blowing up the metric on the base manifold of a fibration.
The adiabatic limit used in our paper first appeared
in the work of Douglas-Wojciechowski \cite{DouWoj91} on  $\eta$-invariant.
Cappell-Lee-Miller \cite{CLM96} studied the behavior of
the eigenvalues of a Dirac operator in the adiabatic limit.
They classified the eigenvalues according to their decaying rate:
large eigenvalues, polynomially decaying eigenvalues and exponentially decaying eigenvalues.
In M{\"u}ller's work \cite{Muller94} on $\eta$-invariants associated with manifolds with cylindrical ends, scattering matrix was implied.
Based on M{\"u}ller's work,
Park-Wojciechowski \cite{PaWo06} expressed the asymptotics of the polynomially decaying eigenvalues
in terms of the scattering matrix.

Now we explain our results in more detail.

Let $Z$ be a compact manifold.
Let $Y\subseteq Z$ be a hypersurface.
We suppose that $Y$ cuts $Z$ into two pieces,
which we denote by $Z_1$ and $Z_2$
(see \textsection \ref{ch1-2}, \textsection \ref{ch3-1}).
Then
\begin{equation}
\label{eq-ch0-glue}
Z=Z_1\cup_Y Z_2.
\end{equation}
Let $g^{TZ}$ be a Riemannian metric on $Z$.
Let $g^{TY}$ be the induced Riemannian metric on $Y$.
We identify a tubular neighborhood of $Y$ with $(-1,1)\times Y$
such that $(-1,0]\times Y\subset Z_1$ and $[0,1)\times Y\subset Z_2$.
Let $(u,y)\in (-1,1)\times Y$ be the coordinates.
We assume that $g^{TZ}$ is product on $(-1,1)\times Y$, i.e.,
\begin{equation}
\label{eq-ch0-riemann-metric}
g^{TZ}\big|_{(-1,1)\times Y} = du^2 + g^{TY}.
\end{equation}

Let $(F,\n^F)$ be a flat complex vector bundle over $Z$ with flat connection $\n^F$,
i.e., $\big(\n^F\big)^2=0$.
Let $\pi_Y : \,(-1,1)\times Y \rightarrow Y$ be the canonical projection.
As $\n^F$ is flat,
we can identify $(F,\n^F)\big|_{(-1,1)\times Y}$ with
the pullback $\pi_Y^*\big(F\big|_Y,\n^F\big|_Y\big)$
via the parallel transport with respect to $\n^F$
along the curves $\gamma_y :t\in\,(-1,1)\mapsto (t,y)\in (-1,1)\times Y$.

Let $h^F$ be a Hermitian metric on $F$.
We assume that $h^F$ is product on $(-1,1)\times Y$, i.e.,
\begin{equation}
\label{eq-ch0-herm-metric}
(F,\n^F,h^F)\big|_{(-1,1)\times Y} = \pi^*_Y \left(F\big|_Y, \n^F\big|_Y, h^F\big|_Y \right).
\end{equation}

For $R > 2$, set $Z_{1,R} = Z_1 \cup_Y ([0,R]\times Y)$,
where we identify $\partial Z_1$ with $\{0\}\times Y$.
We extend
the Riemannian metric $g^{TZ}$,
the flat vector bundle $(F,\n^F)$ and
the Hermitian metric $h^F$
to $Z_{1,R}$
such that \eqref{eq-ch0-riemann-metric}-\eqref{eq-ch0-herm-metric} hold on $(-1,R)\times Y$.
Set $Z_{2,R}=([-R,0]\times Y) \cup_Y Z_2$.
Set
\begin{equation}
\label{eq-ch0-glue-R}
Z_R = Z_{1,R}\cup_Y Z_{2,R},
\end{equation}
where we identify $\partial Z_{1,R} = \{R\}\times Y$ with $\partial Z_{2,R} = \{-R\}\times Y$
(see Figure \ref{lab-fig-intro}).
We extend $g^{TZ}$, $(F,\n^F)$ and $h^F$ to $Z_{2,R}$ and $Z_R$ in the same way.
Note that we take $R>2$ to avoid confusions
between $Z_1$, $Z_2$ and $Z_R$ for $R = 1,2$.
In this paper,
we will always consider $Z_R$ with $R$ large enough.


\begin{figure}[h]
\setlength{\unitlength}{0.75cm}
\centering
\begin{picture}(18,6.5)
\qbezier(1,2)(3,1)(5,2) 
\qbezier(1,4)(3,5)(5,4) 
\qbezier(1,2)(-0.5,3)(1,4) 
\qbezier(2,3.2)(3,2.7)(4,3.2) 
\qbezier(2.5,3)(3,3.2)(3.5,3) 
\qbezier(5,2)(5.5,2.3)(6,2.4) 
\qbezier(5,4)(5.5,3.7)(6,3.6) 
\qbezier(13,2)(15,1)(17,2) 
\qbezier(13,4)(15,5)(17,4) 
\qbezier(17,2)(18.5,3)(17,4) 
\qbezier(14,3.2)(15,2.7)(16,3.2) 
\qbezier(14.5,3)(15,3.2)(15.5,3) 
\qbezier(13,2)(12.5,2.3)(12,2.4) 
\qbezier(13,4)(12.5,3.7)(12,3.6) 
\qbezier(6,2.4)(9,2.4)(12,2.4) 
\qbezier(6,3.6)(9,3.6)(12,3.6) 
\qbezier(6,2.4)(5.7,3)(6,3.6) 
\qbezier[30](6,2.4)(6.3,3)(6,3.6) 
\qbezier(9,2.4)(8.7,3)(9,3.6) 
\qbezier[30](9,2.4)(9.3,3)(9,3.6) 
\qbezier(12,2.4)(11.7,3)(12,3.6) 
\qbezier[30](12,2.4)(12.3,3)(12,3.6) 
\put(6.9,3){\vector(-1,0){1}} 
\put(7.9,3){\vector(1,0){1}} 
\put(9.9,3){\vector(-1,0){1}} 
\put(10.8,3){\vector(1,0){1}} 
\put(9,5.3){\vector(0,-1){1.5}} 
\put(7.15,2.8){$R$} 
\put(10.15,2.8){$R$} 
\put(2.8,2){$Z_1$} 
\put(14.8,2){$Z_2$} 
\put(1,4.8){$\overbrace{\hspace{55mm}}^{}$} \put(4.1,5.4){$Z_{1,R}$}  
\put(9.7,4.8){$\overbrace{\hspace{55mm}}^{}$} \put(12.9,5.4){$Z_{2,R}$}  
\put(1.5,1.2){$\underbrace{\hspace{113mm}}_{}$} \put(8.8,0.3){$Z_R$}  
\put(6.8,5.5){$\partial Z_{1,R} = Y = \partial Z_{2,R}$}
\end{picture}
\caption{}
\label{lab-fig-intro}
\end{figure}

We denote by $\Omega^\bullet(Z_R,F)$ the vector space of differential forms on $Z_R$ twisted by $F$.
Let $d^F : \Omega^\bullet(Z_R,F) \rightarrow \Omega^{\bullet+1}(Z_R,F)$
be the de Rham operator induced by $\n^F$.
Let $d^{F,*}$ be the formal adjoint of $d^F$.
The Hodge-de Rham operator is defined as
\begin{equation}
\label{eq-ch0-def-DR}
D^F_{Z_R} = d^F + d^{F,*} .
\end{equation}
Let $D^F_{Z_{1,R}}$ and $D^F_{Z_{2,R}}$
be the Hodge-de Rham operators
on $\Omega^\bullet_\mathrm{rel}(Z_{1,R},F)$
(with the relative boundary condition)
and $\Omega^\bullet_\mathrm{abs}(Z_{2,R},F)$
(with the absolute boundary condition, see \eqref{eq-ch1-def-bd-D}).

Let $N$ be the number operator on $\Omega^\bullet(Z_R,F)$, i.e.,
$N\omega = p\omega$ for $\omega\in\Omega^p(Z_R,F)$.
Let $P: \Omega^\bullet(Z_R,F) \rightarrow \Ker\big(D^{F,2}_{Z_R}\big)$
be the orthogonal projection (with respect to the $L^2$-metric).
The zeta-type function associated with $D^{F,2}_{Z_R}$ is defined as follows,
for $s\in\{\C\;:\; \mathrm{Re}(s)>\frac{1}{2}\dim Z\}$,
\begin{equation}
\label{eq-ch0-def-zeta}
\theta_R (s) = -\tr\left[(-1)^N N \big(D^{F,2}_{Z_R}\big)^{-s} \left(1-P\right) \right] .
\end{equation}
The function $\theta_R(s)$ admits a meromorphic continuation to the whole complex plane,
which is holomorphic at $0\in\C$
(see Definition \ref{def-ch1-zeta} for an alternative definition).
Let $\theta_{1,R}$ and $\theta_{2,R}$ be the zeta-type functions associated with $D^{F,2}_{Z_{1,R}}$ and $D^{F,2}_{Z_{2,R}}$.

Let $D^F_Y$ be the Hodge-de Rham operator on $\Omega^\bullet(Y,F)$.
Let $\mathrm{Sp}\big(D^F_Y\big)$ be the spectrum of $D^F_Y$.
Set
\begin{equation}
\delta_Y = \min\big\{\abs{\lambda}\;:\;0\neq\lambda\in \mathrm{Sp}\big(D^F_Y\big)\big\}.
\end{equation}
Set $\hh(Y,F) = \Ker\big(D^F_Y\big) \subseteq\Omega^\bullet(Y,F)$.
Let $C_1(\cdot)$ (resp. $C_2(\cdot)$) be the scattering matrix
associated with $D^{F}_{Z_{1,\infty}}$ (resp. $D^{F}_{Z_{2,\infty}}$)
in the sense of \cite{Muller94},
which are analytic functions on $(-\delta_Y,\delta_Y)$
with values in $\mathrm{End}\big(\hh(Y,F) \oplus \hh(Y,F)du\big)$
(see \textsection \ref{ch2-2}).
Set
\begin{equation}
\label{eq-ch0-def-C12}
C_{12} = \left(C_2^{-1}C_1 \right)(0), \hspace{5mm}
C_{12}^p = C_{12}\big|_{\mathscr{H}^p(Y,F) \oplus \mathscr{H}^{p-1}(Y,F)du},
\hspace{2.5mm} \text{ for } p=0,\cdots,\dim Z.
\end{equation}
Let $H^\bullet_\mathrm{rel}(Z_1,F)$ (resp. $H^\bullet_\mathrm{abs}(Z_2,F)$)
be the relative (resp. absolute) cohomology
of $Z_1$ (resp. $Z_2$) with coefficients in $F$.
In the whole paper,
we denote $n=\dim Z$.
Set
\begin{align}
\label{eq-ch0-def-chi}
\begin{split}
& {\chi}' = \sum_{p=0}^n (-1)^p p
\Big\{ \dim H^p(Z,F)- \dim H^p_{\rm rel}(Z_1,F) - \dim H^p_{\rm abs}(Z_2,F) \Big\},\\
&{\chi}'(C_{12}) = \sum_{p=0}^n (-1)^p p\, \dim \Ker \left( C_{12}^p - \Id \right),\\
&\chi(Y) = \sum_{p=0}^{n-1} (-1)^p \dim H^p(Y,\C).
\end{split}
\end{align}

For a square matrix $A$, we denote by $\det^*(A)$ the product of its non-zero eigenvalues.
Let $\rk(F)$ be the rank of $F$.

\begin{thm}
\label{thm-ch0-1}
For any $\e>0$, as $R\to +\infty$,
\begin{align}
\begin{split}
& \theta_R{}'(0) - \theta_{1,R}{}'(0) -\theta_{2,R}{}'(0) \\
& =  2 \chi' \log R + \Big(\chi(Y)\rk(F) +\chi'(C_{12})\Big)\log2 \\
& \hspace{5mm} + \sum_{p=0}^{\dim Z} \frac{p}{2}(-1)^p \log {\det}^*\Big(\frac{2- C_{12}^p - (C_{12}^p)^{-1}}{4}\Big)
+ \mathscr{O}(R^{-1+\e}).
\end{split}
\end{align}
\end{thm}

For a finite dimensional complex vector space $E$,
we set $\det E = \Lambda^{\mathrm{max}} E := \Lambda^{\dim E} E$.
For a complex line $\lambda$,
we denote by $\lambda^{-1}$ its dual line.
For a $\Z$-graded vector space $E^\bullet = \bigoplus_{k=0}^m E^k$,
we define
\begin{equation}
\label{eq-ch0-def-det}
\det E^\bullet = \bigotimes_{k=0}^m \left( \det E^k \right)^{(-1)^k} .
\end{equation}
For an exact sequence of finite dimensional complex vector spaces
\begin{equation}
\label{eq-ch0-exact-sequence}
(V^\bullet,\partial) :
0 \rightarrow
V^0 \overset{\partial} \rightarrow
V^1 \overset{\partial}\rightarrow
\cdots \overset{\partial} \rightarrow
V^m \rightarrow 0,
\end{equation}
there is a canonical section $\varrho \in \det V^\bullet$ defined as follows:
let $m_j = \dim \im \left( \partial \big|_{V^j} \right)$,
we choose $(s_{j,k})_{1\leqslant k \leqslant m_j}$ in $V^j$,
which form a basis of $V^j/\partial V^{j-1}$,
set $\wedge_k s_{j,k} := s_{j,1}\wedge\cdots\wedge s_{j,m_j}$,
\begin{equation}
\label{eq-ch0-def-rho}
\varrho =
\bigotimes_{j=0}^m \big( (\wedge_k \partial s_{j-1,k}) \wedge (\wedge_k s_{j,k}) \big)^{(-1)^j}
\in \det V^\bullet.
\end{equation}
Let $h^{V^\bullet} = \bigoplus_{k=0}^m h^{V^k}$ be a Hermitian metric on $V^\bullet$.
Let $\mathscr{T}(V^\bullet,\partial,h^{V^\bullet})$ be the analytic torsion
associated with $(V^\bullet,\partial,h^{V^\bullet})$ (see \cite[Def. 1.4]{BGS1}).
Let $\lVert\,\cdot\,\rVert_{\det V^\bullet}$ be the norm on $\det V^\bullet$
induced by $h^{V^\bullet}$.
By \cite[Prop. 1.5]{BGS1},
we have
\begin{equation}
\label{eq-ch0-torsion}
\mathscr{T}(V^\bullet,\partial,h^{V^\bullet}) =
\big\lVert\varrho\big\rVert_{\det V^\bullet}.
\end{equation}

We consider the following Mayer-Vietoris exact sequence
(see \textsection \ref{ch1-2})
\begin{equation}
\label{eq-ch0-mv}
\cdots \rightarrow
H^p_\mathrm{rel}(Z_{1,R},F) \rightarrow
H^p(Z_R,F) \rightarrow
H^p_\mathrm{abs}(Z_{2,R},F) \rightarrow
\cdots.
\end{equation}
We denote by $\mathscr{T}_R$ the analytic torsion associated with
the exact sequence \eqref{eq-ch0-mv} equipped with $L^2$-metrics
(induced by Hodge theory).

\begin{thm}
\label{thm-ch0-2}
As $R\rightarrow+\infty$,
\begin{align}
\label{eq-thm-ch0-2}
\begin{split}
\mathscr{T}_R =
2^{\chi'(C_{12})/2} R^{\chi'} \prod_{p=0}^n {\det}^\ast \Big(\frac{2-C^p_{12}-(C^p_{12})^{-1}}{4}\Big)^{\frac{p}{4}(-1)^p}
+ \mathscr{O}(R^{{\chi}'-1}).
\end{split}
\end{align}
\end{thm}

Let $\big\lVert\,\cdot\,\big\rVert^{L^2}_{\det H^\bullet(Z_R,F)}$
be the $L^2$-metric on $\det H^\bullet(Z_R,F)$.
The Ray-Singer metric on $\det H^\bullet(Z_R,F)$ is defined as
\begin{equation}
\label{eq-ch0-def-rs-metric}
\big\lVert\,\cdot\,\big\rVert^\mathrm{RS}_{\det H^\bullet(Z_R,F)} =
\big\lVert\,\cdot\,\big\rVert^{L^2}_{\det H^\bullet(Z_R,F)}\exp\Big(\frac{1}{2}{\theta_R}'(0)\Big) .
\end{equation}
We define the Ray-Singer metrics
$\big\lVert\,\cdot\,\big\rVert^\mathrm{RS}_{\det H^\bullet_\mathrm{rel}(Z_{1,R},F)}$
and $\big\lVert\,\cdot\,\big\rVert^\mathrm{RS}_{\det H^\bullet_\mathrm{abs}(Z_{2,R},F)}$
in the same way (see \textsection \ref{ch1-1}).

Set
\begin{equation}
\label{eq-ch0-def-lambda}
\lambda_R(F) =
\det H_\mathrm{rel}^\bullet(Z_{1,R},F) \otimes
\big( \det H^\bullet(Z_R,F) \big)^{-1} \otimes
\det H^\bullet_\mathrm{abs}(Z_{2,R},F).
\end{equation}
Let $\varrho_R\in \lambda_R(F)$ be the canonical section
(see \eqref{eq-ch0-def-rho}, \eqref{eq-ch0-mv}).
Let $\big\lVert\,\cdot\,\big\rVert^\mathrm{RS}_{\lambda_R(F)}$ be the metric on $\lambda_R(F)$
induced by $\big\lVert\,\cdot\,\big\rVert^\mathrm{RS}_{\det H^\bullet(Z_R,F)}$,
$\big\lVert\,\cdot\,\big\rVert^\mathrm{RS}_{\det H^\bullet_\mathrm{rel}(Z_{1,R},F)}$
and $\big\lVert\,\cdot\,\big\rVert^\mathrm{RS}_{\det H^\bullet_\mathrm{abs}(Z_{2,R},F)}$.

The gluing formula of the analytic torsion established in \cite[Theorem 0.3]{BruMa12}
is stated as follows.

\begin{thm}
\label{thm-ch0-3}
We have
\begin{equation}
\label{eq-thm-ch0-3}
\big\lVert \varrho_R \big\rVert^\mathrm{RS}_{\lambda_R(F)} =
2^{-\frac{1}{2}\chi(Y)\rk(F)} .
\end{equation}
\end{thm}

In this paper,
we give a different proof of Theorem \ref{thm-ch0-3}.
In fact, \eqref{eq-thm-ch0-3} is equivalent to the following identity
\begin{equation}
\label{eq-thm-ch0-3-equiv}
\frac{1}{2}{\theta_R}'(0)
- \frac{1}{2}\theta_{1,R}'(0)
- \frac{1}{2}\theta_{2,R}'(0)
- \log \mathscr{T}_R
= \frac{1}{2}\chi(Y)\rk(F)\log 2.
\end{equation}
By Theorem \ref{thm-ch0-1}, \ref{thm-ch0-2},
the left hand side of \eqref{eq-thm-ch0-3-equiv}
tends to $\frac{1}{2}\chi(Y)\rk(F)\log 2$ as $R\rightarrow +\infty$.
On the other hand,
by the anomaly formulas
(\cite[Theorem 0.1]{BZ92}, \cite[Theorem 0.1]{BruMa06})
the left hand side of \eqref{eq-thm-ch0-3-equiv}
is independent of $R$.
Then \eqref{eq-thm-ch0-3-equiv} follows.

This paper is organized as follows.
In \textsection \ref{ch1},
we recall the definition of Ray-Singer metrics
and introduce the Mayer-Vietoris exact sequence.
In \textsection \ref{ch2},
we study the spectrum of Hodge-Laplacians on manifolds with cylindrical ends.
In \textsection \ref{ch3},
we study asymptotics of the spectrum of the Hodge-Laplacian on $Z_R$ as $R\rightarrow+\infty$.
In \textsection \ref{ch4},
we prove Theorem \ref{thm-ch0-1}.
In \textsection \ref{ch5},
we prove Theorem \ref{thm-ch0-2}.
We also give a (more detailed) proof of Theorem \ref{thm-ch0-3}.

\hfill\\
\noindent\textbf{Notations.}
Hereby we summarize some frequently used notations in this paper.

A manifold (with or without boundary) is usually denoted by $X$, $Y$ or $Z$.
We always consider a manifold equipped with a flat complex vector bundle $(F,\n^F)$.
We denote by $\Omega^\bullet(X,F)$ (resp. $\Omega^\bullet_\mathrm{c}(X,F)$)
the vector space of differential forms (resp. differential forms with compact support)
on $X$ with values in $F$.
We denote by $d^F$ the de Rham operator on $\Omega^\bullet(X,F)$ induced by $\n^F$.
We will also use the notation $d^F_X$
when we want to emphasize the underlying manifold.
We denote
\begin{align}
\label{eq-notation-H}
\begin{split}
& H^\bullet_\mathrm{abs}(X,F) = H^\bullet(X,F)
\simeq H^\bullet\big(\Omega^\bullet(X,F),d^F\big),\\
& H^\bullet_\mathrm{rel}(X,F) = H^\bullet(X,\partial X,F)
\simeq H^\bullet\big(\Omega^\bullet_\mathrm{c}(X,F),d^F\big).
\end{split}
\end{align}
We write $H^\bullet_\bd(X,F)$ for short if the choice of abs/rel is clear.
If $X$ is equipped with a Riemannian metric,
we denote by $\big\lVert\cdot\big\rVert_X$
the $L^2$-metric (see \eqref{eq-ch1-def-scalar-prod})
on $\Omega^\bullet_\mathrm{c}(X,F)$
or its $L^2$-closure $L^2\big(\Omega^\bullet_\mathrm{c}(X,F)\big)$.
For a vector subspace $V \subseteq L^2\big(\Omega^\bullet_\mathrm{c}(X,F)\big)$,
we denote by $L^2(V)$ the closure of $V$ with respect to $\big\lVert\cdot\big\rVert_X$.

For a linear map $A$,
we denote by $\mathrm{Ker}(A)$ (resp. $\mathrm{Im}(A)$)
the kernel (resp. image) of $A$.
For a self-adjoint operator $A$,
we denote by $\Sp(A)$ its spectrum.
For a Hermitian matrix $A$, we note
\begin{equation}
{\det}^*(A) = \prod_{\lambda\in\Sp(A)\backslash\{0\}} \lambda.
\end{equation}

\hfill\\
\noindent\textbf{Acknowledgments.}
The authors are grateful to Prof. Xiaonan Ma for having raised the question which is solved in this paper.
This paper was written during the work of Y. Z. on his PhD thesis.
Y. Z. would like to thank his advisor Prof. Jean-Michel Bismut for his support.
This paper was partially written during postdoctoral work of J. Z. at Chern Institute of Mathematics.
J. Z. would like to thank his postdoctoral advisor Prof. Weiping Zhang for his support.
The authors would like to thank the referee for his/her detailed comments and suggestions.

The research has received funding from
the European Research Council (E.R.C.) under European Union's Seventh Framework Program
(FP7/2007-2013)/ ERC grant agreement No. 291060.
J. Z. was also supported by the NNSFC (Grant No. 11601089, No. 11571183).

The results contained in this paper were announced in \cite{PZZ-CRAS}.

\section{Ray-Singer metrics}
\label{ch1}

In this section,
we recall several fundamental notions and constructions.
This section is organized as follows.
In \textsection \ref{ch1-1},
we recall the definition of Ray-Singer metric.
In \textsection \ref{ch1-2},
we introduce the Mayer-Vietoris exact sequence.

\subsection{Hodge theory and Ray-Singer metrics}
\label{ch1-1}

Let $X$ be a $n$-dimensional compact manifold with boundary $Y$.
Let $F$ be a flat complex vector bundle over $X$ with flat connection $\nabla^{F}$.

Let $g^{TX}$ be a Riemannian metric on $X$.
Let $h^{F}$ be a Hermitian metric on $F$.
Let $\big\langle\cdot,\cdot\big\rangle_{\Lambda^\bullet(T^{*}X)\otimes F}$ be the scalar product
on $\Lambda^\bullet(T^{*}X)\otimes F$ induced by $g^{TX}$ and $h^F$.
Let $o(TX)$ be the orientation line bundle of $X$.
Let $dv_{X}\in\Omega^n(X,o(TX))$ be the Riemannian volume form on $(X,g^{TX})$.
The $L^2$-metric on $\Omega^\bullet(X,F)$ is defined as follows:
for $s,s'\in \Omega^\bullet(X,F)$,
\begin{equation}
\label{eq-ch1-def-scalar-prod}
\big\langle s,s' \big\rangle_X =
\int_X \big\langle s(x),s'(x)\big\rangle_{\Lambda^\bullet(T^{*}X)\otimes F} dv_{X}.
\end{equation}
Let $d^{F,*}$ be the formal adjoint of $d^{F}$ with respect to \eqref{eq-ch1-def-scalar-prod}.
The Hodge-de Rham operator is defined as
\begin{equation}
\label{eq-ch1-def-D}
D^{F}_{X} = d^{F}+d^{F,*}.
\end{equation}
Then the associated Hodge-Laplacian is defined by
\begin{equation}
\label{eq-ch1-def-D2}
D^{F,2}_{X} = d^{F}d^{F,*}+d^{F,*}d^{F}: \;
\Omega^p(X,F)\rightarrow \Omega^p(X,F).
\end{equation}

We identify the normal bundle $\mathfrak{n}$ of $Y \subseteq X$
with the orthogonal complement of $TY\big|_Y \subseteq TX|_{Y}$.
We denote by $e_{\mathfrak{n}}$ the inward pointing unit normal vector field on $Y$.
Let $e^{\mathfrak{n}}$ be the dual vector field.
We denote by $i_\cdot$ (resp. $\wedge \cdot$)
the interior (resp. exterior) multiplication.

Set
\begin{align}
\label{eq-ch1-def-bd-D}
\begin{split}
& \Omega^{p}_\mathrm{abs}(X,F) =
\big\{\sigma\in \Omega^{p}(X,F) \;:\; \left (i_{e_\mathfrak{n}}\sigma\right)|_{Y}=0 \big\},\\
& \Omega^{p}_\mathrm{rel}(X,F) =
\big\{\sigma\in \Omega^{p}(X,F) \;:\; \left (e^\mathfrak{n}\wedge\sigma\right)|_{Y}=0 \big\}.
\end{split}
\end{align}
The restriction of $D^{F}_X$
to $\Omega^{p}_\mathrm{abs/rel}(X,F)$
is essentially self-adjoint.
Set
\begin{align}
\label{eq-ch1-def-bd-D2}
\begin{split}
& \Omega^p_{\mathrm{abs},D^2}(X,F) =
\big\{\sigma\in \Omega^p(X,F) \;:\;
\left (i_{e_\mathfrak{n}}\sigma\right)|_Y =
\left(i_{e_\mathfrak{n}}(d^{F}\sigma)\right)|_Y = 0 \big\},\\
& \Omega^p_{\mathrm{rel},D^2}(X,F) =
\big\{\sigma\in \Omega^p(X,F) \;:\;
\left (e^\mathfrak{n}\wedge\sigma\right)|_Y =
\left(e^\mathfrak{n}\wedge(d^{F,*}\sigma)\right)|_Y = 0 \big\}.
\end{split}
\end{align}
We denote by $D^{F,2}_{X,\mathrm{abs}}$ and $D^{F,2}_{X,\mathrm{rel}}$
the restrictions of $D^{F,2}_{X}$
to $\Omega^p_{\mathrm{abs},D^{2}}(X,F)$ and $\Omega^p_{\mathrm{rel},D^{2}}(X,F)$.
For the sake of convenience,
a unified notation '$\mathrm{bd}$' will be adopted
to represent '$\mathrm{abs}$' or '$\mathrm{rel}$',
when it is not necessary to distinguish the boundary conditions.
The operator $D^{F,2}_{X,\mathrm{bd}}$ is essentially self-adjoint.
Set
\begin{equation}
\label{eq-ch1-def-mathscrH}
\mathscr{H}^p_\mathrm{bd}(X,F) =
\big\{ \sigma\in\Omega^p_{\mathrm{bd},D^{2}}(X,F) \;:\; D^{F,2}_{X}\sigma = 0 \big\}.
\end{equation}

Let $\K_{X}$ be a smooth triangulation of $X$.
Let $C_{\bullet}(\K_{X},F^{*})$ (resp. $C^{\bullet}(\K_{X},F)$)
be the corresponding chain (resp. cochain) groups with coefficents in $F^*$
(cf. \cite[$\S$ 1.1]{BruMa12}).
We define the de Rham map
$P_{\infty}: \Omega^\bullet(X,F)\rightarrow C^{\bullet}(\K_{X},F)$ as follows
(cf. \cite[(1.15)]{BruMa12}),
\begin{equation}
\label{eq-ch1-dR-map}
P_{\infty}(\sigma)(\mathfrak{a}) =
\int_{\mathfrak{a}} \sigma ,  \quad
\text{for} \, \sigma\in \Omega^\bullet(X,F),\, \mathfrak{a}\in C_{\bullet}(\K_{X},F^{*}).
\end{equation}

The following theorem comes from \cite[Theorem 1.1]{BruMa12}.

\begin{thm}[Hodge decomposition theorem]
\label{thm-ch1-Hodge} \hfill
\begin{enumerate}
\item We have
\begin{equation}
\label{eq-thm-ch1-Hodge-1}
\mathscr{H}^p_\mathrm{bd}(X,F) =
\Ker(d^F) \cap \Ker(d^{F,*}) \cap \Omega^p_\mathrm{bd}(X,F).
\end{equation}
\item The vector space $\mathscr{H}^p_\mathrm{bd}(X,F)$ is finite dimensional.
\item We have the following orthogonal decomposition
\begin{align}
\label{eq-thm-ch1-Hodge-2}
\begin{split}
\Omega^p_\mathrm{bd}(X,F) & =
\mathscr{H}^p_\mathrm{bd}(X,F) \oplus
d^F \Omega^{p-1}_\mathrm{bd}(X,F) \oplus
d^{F,*} \Omega^{p+1}_\mathrm{bd}(X,F),\\
L^2\big(\Omega^p(X,F)\big) & =
\mathscr{H}^p_\mathrm{bd}(X,F) \oplus
L^2\big(d^F \Omega^{p-1}_\mathrm{bd}(X,F)\big) \oplus
L^2\big(d^{F,*} \Omega^{p+1}_\mathrm{bd}(X,F)\big).
\end{split}
\end{align}
\item The inclusion
$\mathscr{H}^p_\mathrm{bd}(X,F) \rightarrow \Ker(d^{F})\cap \Omega^{p}_\mathrm{bd}(X,F)$
composed with $P_\infty$ maps into
$C^{p}(\K_{X},F)$ for '{\rm bd}' = '{\rm abs}'
(resp. $C^{p}(\K_{X},\K_{Y},F)$ for '{\rm bd}' = '{\rm rel}').
Moreover, the induced map
\begin{equation}
\label{eq-thm-ch1-Hodge-3}
\mathscr{H}^{p}_\mathrm{bd}(X,F)\rightarrow H^{p}_{\rm bd}(X,F).
\end{equation}
is an isomorphism.
\end{enumerate}
\end{thm}

The de Rham cohomology is defined as
\begin{equation}
\label{eq-ch1-def-dR-cohomology}
H^p(\Omega^{\bullet}_\mathrm{bd}(X,F),d^F) =
\frac{\Ker(d^{F})\cap\Omega^p_\mathrm{bd}(X,F)}
{d^{F}(\Omega^{p-1}_\mathrm{bd}(X,F)) \cap \Omega^p_\mathrm{bd}(X,F)}.
\end{equation}
By Theorem \ref{thm-ch1-Hodge},
the map $P_{\infty}$ induces an isomorphism
\begin{equation}
\label{eq-ch1-iso-cohomolgy}
H^p(\Omega^{\bullet}_\mathrm{bd}(X,F),d^{F})
\simeq H^{p}_\mathrm{bd}(X,F).
\end{equation}

Let
$P_{\mathscr{H}}: \Omega^\bullet(X,F) \rightarrow \mathscr{H}_\mathrm{bd}^\bullet(X,F)$
be the orthogonal projection
(with respect to the $L^2$-metric \eqref{eq-ch1-def-scalar-prod}).
Set $P^{\bot}_{\mathscr{H}} = 1-P_{\mathscr{H}}$.
Let $N$ be the number operator on $\Omega^\bullet(X,F)$ (see \eqref{eq-ch0-def-zeta}).
Let $\exp\big(-tD^{F,2}_{X,\mathrm{bd}}\big)$
be the heat semi-group associated with $D^{F,2}_{X,\mathrm{bd}}$.

\begin{defn}
\label{def-ch1-zeta}
For $s\in \C$ with $\mathrm{Re}(s)>\frac{n}{2}$,
set
\begin{equation}
\label{eq-def-ch1-zeta}
\theta^{F}_{X,\mathrm{bd}}(s) = - \frac{1}{\Gamma(s)}
\int_0^{+\infty} t^s \tr\big[(-1)^NN\exp\big(-tD^{F,2}_{X,\mathrm{bd}}\big)P^{\bot}_{\mathscr{H}}\big] \frac{dt}{t}.
\end{equation}
The function $\theta^{F}_{X,\mathrm{bd}}(u)$ admits a meromorphic continuation to $\C$,
which is holomorphic at $s=0$ (cf. \cite[Theorem 3.2] {BruMa12}).
\end{defn}

Definition \ref{def-ch1-zeta} is
equivalent to the definition of zeta-type function given in the introduction.
It is Definition \ref{def-ch1-zeta}
that will be used in our proof (in \textsection \ref{ch4}).

\begin{defn}
\label{def-ch1-rs-torsion}
The Ray-Singer analytic torsion is defined as
\begin{equation}
\label{eq-def-ch1-rs-torsion}
T_\mathrm{bd}(X,g^{TX},h^{F}) =
\exp\Big(\frac{1}{2}{\theta^{F}_{X,\mathrm{bd}}}'(0)\Big).
\end{equation}
\end{defn}

Let $\det H^{\bullet}_\mathrm{bd}(X,F)$ be the determinant line
(see \eqref{eq-ch0-def-det}) of $H^\bullet_\mathrm{bd}(X,F)$.
By the isomorphism in \eqref{eq-thm-ch1-Hodge-3},
$H^{\bullet}_{\rm bd}(X,F)$ inherits a $L^{2}$-metric $h^{H^{\bullet}_\mathrm{bd}(X,F)}$
from the $L^2$-metric on $\Omega^\bullet(X,F)$.
Let $\big\lVert\,\cdot\,\big\rVert^{L^{2}}_{\det H^{\bullet}_\mathrm{bd}(X,F)}$
be the induced metric on $\det H^{\bullet}_\mathrm{bd}(X,F)$.

\begin{defn}
\label{def-ch1-rs-metric}
The Ray-Singer metric on $\det H^{\bullet}_\mathrm{bd}(X,F)$ is defined as
\begin{equation}
\label{eq-def-ch1-rs-metric}
\big\lVert\,\cdot\,\big\rVert^\mathrm{RS}_{\det H^{\bullet}_\mathrm{bd}(X,F)}
= \big\lVert\,\cdot\,\big\rVert^{L^{2}}_{\det H^{\bullet}_\mathrm{bd}(X,F)}T_\mathrm{bd}(X,g^{TX},h^{F}).
\end{equation}
\end{defn}

\subsection{Mayer-Vietoris exact sequence}
\label{ch1-2}

In this subsection,
we will use the notations in the introduction.
Recall that $Z$ is a compact smooth Riemannian manifold
and $Y\subseteq Z$ is a hypersurface cutting $Z$ into two pieces $Z_1$ and $Z_2$.
We have the following Mayer-Vietoris exact sequence (cf. \cite[(0.16)]{BruMa12})
\begin{equation}
\label{eq-ch1-mv}
\cdots \rightarrow
H^p_\mathrm{rel}(Z_1,F) \overset{\alpha_p} \rightarrow
H^p(Z,F) \overset{\beta_p} \rightarrow
H^p_\mathrm{abs}(Z_2,F) \overset{\delta_p} \rightarrow
H^{p+1}_\mathrm{rel}(Z_1,F) \rightarrow
\cdots.
\end{equation}
By \eqref{eq-ch1-iso-cohomolgy} and \eqref{eq-ch1-mv},
we get the following exact sequence
\begin{equation}
\label{eq-ch1-mv-dR}
\cdots \rightarrow
H^p(\Omega^\bullet_\mathrm{rel}(Z_1,F),d^{F}) \overset{\alpha_p} \rightarrow
H^p(\Omega^\bullet(Z,F),d^{F}) \overset{\beta_p} \rightarrow
H^p(\Omega^{\bullet}_\mathrm{abs}(Z_2,F),d^{F}) \overset{\delta_p} \rightarrow
\cdots.
\end{equation}

The following proposition is well-known.

\begin{prop}
\label{prop-ch1-mv}
The maps $\alpha_p$, $\beta_p$ and $\delta_p$ in \eqref{eq-ch1-mv-dR}
are given as follows:
\begin{enumerate}[{\rm a)}]
\item For $[\sigma]\in H^p(\Omega^\bullet_\mathrm{rel}(Z_1,F),d^F)$,
there exists $\sigma'\in [\sigma]$ vanishing near $Y$.
Let $\sigma''\in \Omega^p(Z,F)$ be the extension of $\sigma'$ by zero.
We have $\alpha_p([\sigma]) = [\sigma'']$.
\item For $[\sigma]\in H^p(\Omega^\bullet(Z,F),d^F)$,
there exists $\sigma'\in [\sigma]$ such that
$\sigma'' := \sigma'\big|_{Z_2}\in\Omega^p_\mathrm{abs}(Z_2,F)$.
We have $\beta_p([\sigma]) = [\sigma'']$.
\item For $[\sigma]\in H^p(\Omega^\bullet_\mathrm{abs}(Z_2,F),d^F)$,
there exists $\sigma'\in\Omega^p(Z,F)$ such that
$\sigma'\big|_{Z_2}\in[\sigma]$.
Then $\sigma'' := d^F \sigma'\big|_{Z_1}\in\Omega^{p+1}_\mathrm{rel}(Z_1,F)$.
We have $\delta_p([\sigma]) = [\sigma'']$.
\end{enumerate}
\end{prop}

By Theorem \ref{thm-ch1-Hodge} and \eqref{eq-ch1-mv},
we have the following exact sequence
\begin{equation}
\label{eq-ch1-mv-harmonic}
\cdots \rightarrow
\mathscr{H}^p_\mathrm{rel}(Z_1,F) \overset{\alpha_p} \rightarrow
\mathscr{H}^p(Z,F) \overset{\beta_p}\rightarrow
\mathscr{H}^p_\mathrm{abs}(Z_2,F) \overset{\delta_p}\rightarrow\cdots.
\end{equation}
By Theorem \ref{thm-ch1-Hodge}, Proposition \ref{prop-ch1-mv} and \eqref{eq-ch1-iso-cohomolgy},
we get the following proposition.

\begin{prop}
\label{prop-ch1-mv-harmonic}
For
$\omega\in \mathscr{H}^p_\mathrm{rel}(Z_1,F)$,
$\mu\in \mathscr{H}^p(Z,F)$,
$\nu\in \mathscr{H}^p_\mathrm{abs}(Z_{2},F)$ and
$\omega'\in \mathscr{H}^{p+1}_\mathrm{rel}(Z_1,F)$,
we have
\begin{equation}
\label{eq-prop-ch1-mv-harmonic}
\big\langle \alpha_p(\omega),\mu \big\rangle_{Z} =
\big\langle \omega,\mu \big\rangle_{Z_{1}},\quad
\big\langle \beta_p(\mu),\nu \big\rangle_{Z_{2}} =
\big\langle \mu,\nu \big\rangle_{Z_{2}},\quad
\big\langle \delta_p(\nu),\omega' \big\rangle_{Z_1} =
\big\langle \nu,i_{e_{\mathfrak{n}}}\omega' \big\rangle_Y.
\end{equation}
\end{prop}
\begin{proof}
We only prove the first identity.
The second and third identities can be proved in a similar way.

By the second identity in \eqref{eq-notation-H} and \eqref{eq-ch1-iso-cohomolgy},
there exists $\omega'\in \Omega^p_\mathrm{c}(Z_1,F)$
(i.e., $\omega'$ vanishes near $\partial Z_1$)
such that
\begin{equation}
\label{eq1-pf-prop-ch1-mv-harmonic}
d^F\omega' = 0, \hspace{5mm}
[\omega'] = [\omega] \in H^p(\Omega^{\bullet}_\mathrm{rel}(Z_1,F),d^{F}).
\end{equation}
We may view $\omega'$ as an element in $\Omega^p(Z,F)$,
i.e., $\omega'$ vanishes on $Z_2$.
By Proposition \ref{prop-ch1-mv}
and \eqref{eq1-pf-prop-ch1-mv-harmonic},
we have
\begin{equation}
\label{eq2-pf-prop-ch1-mv-harmonic}
[\alpha_p(\omega)] = [\omega'] \in H^p(Z,F).
\end{equation}
By Theorem \ref{thm-ch1-Hodge},
\eqref{eq1-pf-prop-ch1-mv-harmonic}
and \eqref{eq2-pf-prop-ch1-mv-harmonic},
we have
\begin{equation}
\big\langle \alpha_p(\omega),\mu \big\rangle_Z =
\big\langle \omega',\mu \big\rangle_Z =
\big\langle \omega',\mu \big\rangle_{Z_1} =
\big\langle \omega,\mu \big\rangle_{Z_1}.
\end{equation}
This completes the proof of Proposition \ref{prop-ch1-mv-harmonic}.
\end{proof}

\section{Manifolds with cylinder ends and Scattering matrix}
\label{ch2}

In this section,
we study the Hodge-Laplacians on manifolds with cylindrical ends.
This section is organized as follows.
In \textsection \ref{ch2-1}, we calculate
the eigensections of the Hodge-Laplacian on a cylinder
without boundary condition.
In \textsection \ref{ch2-2}, we study
the absolutely continuous spectrum of the Hodge-Laplacian in question.
In \textsection \ref{ch2-3}, we study
the extended $L^2$-solutions,
i.e., generalized eigensections associated with $0$.

\subsection{Hodge-de Rham operator on a cylinder}
\label{ch2-1}

Let $(Y,g^{TY})$ be a compact Riemannian manifold.
Let $F$ be a flat vector bundle over $Y$.
Let $h^F$ be a Hermitian metric on $F$.

Let $du$ be a unit odd Grassmannian variable such that $(du)^{2}=0$,
which we view as an endomorphism of degree $1$,
and fix its norm by $|du|=1$
and assume that it anticommutes with $\Lambda^\bullet(T^*Y)$.
Set $\Omega^\bullet(Y,F[du]) = \Omega^\bullet(Y,F) \otimes \C[du]$,
then as a $\Z_2$-graded vector space,
\begin{equation}
\Omega^{p}(Y,F[du])=\Omega^{p}(Y,F)\oplus \Omega^{p-1}(Y,F)du.
\end{equation}

Let $D^F_Y$ be the Hodge-de Rham operator
(cf. \eqref{eq-ch1-def-D})
on $\Omega^\bullet(Y,F[du])$.
Let $\mathscr{H}^\bullet(Y,F)$ (resp. $\mathscr{H}^{\bullet}(Y,F[du])$) be
the kernel of $D^{F,2}_Y$ on $\Omega^\bullet(Y,F)$ (resp. $\Omega^\bullet(Y,F[du])$).
For $\mu\in \R$,
Let $\mathscr{E}_\mu(Y,F)$ (resp. $\mathscr{E}_\mu(Y,F[du])$) be
the eigenspace of $D^F_Y$ on $\Omega^\bullet(Y,F)$ (resp. $\Omega^\bullet(Y,F[du])$)
associated with $\mu$.
Then we have
\begin{align}
\begin{split}
& \mathscr{H}^\bullet(Y,F[du]) = \mathscr{H}^\bullet(Y,F) \oplus \mathscr{H}^\bullet(Y,F)du, \\
& \mathscr{E}_\mu(Y,F[du]) = \mathscr{E}_\mu(Y,F) \oplus \mathscr{E}_{-\mu}(Y,F)du.
\end{split}
\end{align}

Let $I\subseteq\R$ be an interval.
For ease of notation,
we denote $Y_I = I\times Y$.
Let $(u,y)\in Y_I$ be the coordinates.
Let $\pi_Y: Y_I\rightarrow Y$ be the canonical projection.
We equip $Y_I$ with the product metric \eqref{eq-ch0-riemann-metric}.
The pull-back of $F$ (resp. $h^F$) by $\pi_{Y}^{*}$
is still denoted by $F$ (resp. $h^F$).
We equip $F$ over $Y_I$ with the flat connection
\begin{equation}
\nabla^F = du\wedge\frac{\partial}{\partial u} + \nabla^F\big|_{Y}.
\end{equation}

We will use the following identification
\begin{equation}
\label{eq-ch2-iden-duform}
\Omega^\bullet(Y_I,F) =
\mathscr{C}^\infty(I,\Omega^\bullet(Y,F[du])).
\end{equation}
For $u\in I$ and $\omega\in \Omega^\bullet(Y_{I},F)$,
we denote by $\omega_u\in \Omega^{\bullet}(Y,F[du])$
the value of $\omega$, viewed as a function on $I$, at $u\in I$.
By \eqref{eq-ch1-def-scalar-prod} and \eqref{eq-ch2-iden-duform},
for $\omega,\omega'\in \Omega^{\bullet}(Y_{I},F)$,
we have
\begin{equation}
\big\langle \omega,\omega' \big\rangle_{Y_I} =
\int_I \big\langle\omega_u,\omega'_u \big\rangle_Y du.
\end{equation}

Let $\frac{\partial}{\partial u}$ be the dual of $du$.
We define the action on $\C[du]$ by
\begin{equation}
\cliffu = du\wedge - i_{\frac{\partial}{\partial u}},\quad
\hatcliffu = du\wedge + i_{\frac{\partial}{\partial u}}.
\end{equation}

Let $D^F_{Y_I}$ be the Hodge-de Rham operator on $\Omega^\bullet(Y_{I},F)$.
We have
\begin{equation}
\label{eq-ch2-D-YI}
D^F_{Y_I} =
\cliffu \frac{\partial}{\partial u} +
D^F_Y, \hspace{5mm}
\cliffu D^F_Y = - D^F_Y \cliffu.
\end{equation}
For $I=[a,b]$, we have the Green formula:
for $\omega,\omega'\in\Omega^\bullet(Y_I,F)$,
\begin{equation}
\label{eq-ch2-green}
\big\langle D^F_{Y_I}\omega,\omega' \big\rangle_{Y_I} -
\big\langle \omega,D^F_{Y_I}\omega' \big\rangle_{Y_I} =
\big\langle \cliffu\omega_b,\omega'_b \big\rangle_Y -
\big\langle \cliffu\omega_a,\omega'_a \big\rangle_Y.
\end{equation}

Set
\begin{equation}
\label{eq-ch2-def-deltaY}
\delta_Y = \mathrm{min}\big\{\abs{\mu} \;:\;
0\neq\mu\in \mathrm{Sp}\big(D^{F}_{Y}\big)\big\}.
\end{equation}
For $\omega\in\Omega^\bullet(Y_I,F)$ an eigensection of $D^F_{Y_I}$
associated with eigenvalue $\lambda\in(-\delta_Y,\delta_Y)$,
by \eqref{eq-ch2-D-YI},
we have
\begin{align}
\label{eq-ch2-eigen-expansion}
\begin{split}
\omega & =
e^{-iu\lambda}\big(\phi^-_0-i\cliffu\phi^-_0\big) +
e^{iu\lambda} \big(\phi^+_0+i\cliffu\phi^+_0\big) \\
& \hspace{5mm} + \sum_{0\neq\mu\in\mathrm{Sp}(D^F_Y)}
e^{-\sqrt{\mu^2-\lambda^2}u} \big( \phi^-_\mu -
\frac{\mu-\lambda}{\sqrt{\mu^2-\lambda^2}}\cliffu\phi^-_\mu \big) \\
& \hspace{5mm} + \sum_{0\neq\mu\in\mathrm{Sp}(D^F_Y)}
e^{\sqrt{\mu^2-\lambda^2}u} \big( \phi^+_\mu +
\frac{\mu-\lambda}{\sqrt{\mu^2-\lambda^2}}\cliffu\phi^+_\mu \big),
\end{split}
\end{align}
where $\phi^\pm_0\in\mathscr{H}^\bullet(Y,F)$,
$\phi^\pm_\mu\in\mathscr{E}_\mu(Y,F[du])$
(cf. \cite[\textsection 4]{Muller94}).
Note that
\begin{equation}
\label{eq-cliff-ann}
\big(1-i\cliffu\big)\big(1+i\cliffu\big) = 0,
\end{equation}
which explains that
$\phi^\pm_0\in\mathscr{H}^\bullet(Y,F)$
(instead of $\phi^\pm_0\in\mathscr{H}^\bullet(Y,F[du])$).
Set
\begin{align}
\label{eq-ch2-def-zm-nz-1}
\begin{split}
& \omega^{\mathrm{zm},\pm} =
e^{\pm iu\lambda}\big(\phi^\pm_0 \pm i\cliffu\phi^\pm_0\big),\\
& \omega^{\mu,\pm} =
e^{\pm\sqrt{\mu^2-\lambda^2}u}\big(\phi^\pm_\mu \pm
\frac{\mu-\lambda}{\sqrt{\mu^2-\lambda^2}}\cliffu\phi^\pm_\mu\big),\\
& \omega^\pm = \sum_{0\neq\mu\in\mathrm{Sp}(D^F_Y)}\omega^{\mu,\pm},
\quad \omega^\mu = \omega^{\mu,-} + \omega^{\mu,+},\\
& \omega^\mathrm{zm} = \omega^{\mathrm{zm},-} + \omega^{\mathrm{zm},+},
\quad \omega^\mathrm{nz} = \omega^- + \omega^+.
\end{split}
\end{align}
We remark that
\begin{equation}
\cliffu \omega^{\mathrm{zm},\pm} = \mp i \, \omega^{\mathrm{zm},\pm}.
\end{equation}
We call $\omega^\mathrm{zm}$ the zeromode of $\omega$.
By \eqref{eq-ch2-eigen-expansion}-\eqref{eq-ch2-def-zm-nz-1},
we have
\begin{equation}
\label{eq-ch2-decomp-zm-nz}
\omega
= \omega^\mathrm{zm} + \omega^\mathrm{nz}
= \omega^\mathrm{zm} + \omega^- + \omega^+.
\end{equation}
Moreover,
for $u\in I$ and $\mu,\nu\in\R$ with $|\mu|\neq|\nu|$,
we have
\begin{equation}
\label{eq-ch2-orth-zm-nz}
\big\langle \omega^\mathrm{zm}_u,\omega^\mu_u \big\rangle_Y = 0, \quad
\big\langle \omega^\mu_u,\omega^\nu_u \big\rangle_Y = 0.
\end{equation}
Furthermore,
for $a<u<v<b$,
we have
\begin{equation}
\label{eq-ch2-nz-decreasing}
\big\lVert \omega^-_v \big\rVert_Y \leqslant
e^{-(v-u)\sqrt{\delta^2_Y-\lambda^2}} \big\lVert \omega^-_u \big\rVert_Y,\quad
\big\lVert \omega^+_u \big\rVert_Y \leqslant
e^{-(v-u)\sqrt{\delta^2_Y-\lambda^2}} \big\lVert \omega^+_v \big\rVert_Y.
\end{equation}
By \eqref{eq-ch2-def-zm-nz-1},
$\big\lVert \omega^\mathrm{zm}_u \big\rVert_Y$ is independent of $u\in I$.
We denote
\begin{equation}
\label{eq-ch2-zm-norm}
\big\lVert \omega^\mathrm{zm} \big\rVert_Y =
\big\lVert \omega^\mathrm{zm}_u \big\rVert_Y.
\end{equation}

\begin{lemma}
\label{lem-ch2-zm-nz}
We assume that $I=[a,b]$.
For $\omega_1,\,\omega_2\in \Omega^\bullet(Y_I,F)$ eigensections of $D^F_{Y_I}$
associated with eigenvalues in $[-\delta_Y/2,\delta_Y/2]$,
we have
\begin{align}
\label{eq-lem-ch2-zm-nz}
\begin{split}
\Big| \big\langle \omega^\mathrm{nz}_1,\omega^\mathrm{nz}_2 \big\rangle_{Y_I} \Big|
& \leqslant
2 \delta_Y^{-1}
\Big( 1 - e^{-(b-a)\delta_Y/2} \Big)^{-2}
\big\lVert \omega_1 \big\rVert_{\partial Y_I}
\big\lVert \omega_2 \big\rVert_{\partial Y_I},\\
\Big| \big\langle \omega^\mathrm{zm}_1,\omega^\mathrm{zm}_2 \big\rangle_Y \Big|
& \leqslant
\frac{1}{2}
\big\lVert \omega_1 \big\rVert_{\partial Y_I}
\big\lVert \omega_2 \big\rVert_{\partial Y_I}.
\end{split}
\end{align}
\end{lemma}
\begin{proof}
They follow from the Cauchy-Schwarz inequality
and \eqref{eq-ch2-eigen-expansion}-\eqref{eq-ch2-zm-norm}.
\end{proof}

\subsection{Hodge-de Rham operator on manifolds with cylindrical ends}
\label{ch2-2}

Let $(X_\infty,g^{TX_\infty})$ be a non-compact Riemannian manifold
with cylindrical end $Y_{\R_+}$, i.e.,
$X_\infty$ contains a subset isometric to $Y_{\R_+}$ whose complement is compact.
Let $(F,\nabla^{F})$ be a flat vector bundle over $X_\infty$.
Let $h^F$ be a Hermitian metric on $X_\infty$
such that \eqref{eq-ch0-herm-metric} holds on $Y_{\R_+}$.

By \cite[Theorem 3.2]{Muller94},
the Hodge-de Rham operator $D^F_{X_\infty}$ on $\Omega^\bullet_\mathrm{c}(X_\infty,F)$
is essentially self-adjoint.
Its self-adjoint extension  is still denoted by $D^F_{X_\infty}$.
We have the following decomposition
\begin{equation}
\label{eq-ch2-pp-ac}
L^2\big(\Omega^\bullet_\mathrm{c}(X_\infty,F)\big) =
\mathscr{E}^\bullet_\mathrm{pp}(X_\infty,F) \oplus
\mathscr{E}^\bullet_\mathrm{sc}(X_\infty,F) \oplus
\mathscr{E}^\bullet_\mathrm{ac}(X_\infty,F) \,,
\end{equation}
where the vector subspaces on the right hand side correspond to
purely point spectrum, singularly continuous spectrum and absolutely continuous spectrum
of $D^F_{X_\infty}$ (cf. \cite[Chapter 7.2]{ReedSimon1}).
We denote by
$D^F_{X_{\infty},\mathrm{pp}}$,
$D^F_{X_{\infty},\mathrm{sc}}$ and
$D^F_{X_{\infty},\mathrm{ac}}$
the restrictions of $D^F_{X_{\infty}}$ to the corresponding vector subspaces.

Now we give a formal definition of generalized eigensection.
We refer to \cite[Chapter 5]{Bere68} for details.
For $\Delta\subseteq \R$ a Borel subset,
let $I_\Delta:\R\rightarrow\R$ be the function defined by
\begin{equation}
I_\Delta(\lambda) = \left\{
\begin{array}{rl}
1 & \text{for } \lambda \in \Delta, \\
0 & \text{for } \lambda \notin \Delta.
\end{array} \right.
\end{equation}
We denote by $A^\bullet(X_\infty,F)$
the vector space of currents on $X_\infty$ with values in $F$.
The operator
\begin{equation}
I_\Delta \big(D^F_{X_\infty}\big):
\Omega^\bullet_\mathrm{c}(X_\infty,F) \rightarrow
L^2\big(\Omega^\bullet_\mathrm{c}(X_\infty,F)\big) \subseteq A^\bullet(X_\infty,F)
\end{equation}
is well-defined.
There exists an operator valued distribution
\begin{equation}
\Psi: \R \rightarrow \mathrm{Hom}\big(\Omega^\bullet_\mathrm{c}(X_\infty,F),A^\bullet(X_\infty,F)\big)
\end{equation}
such that
\begin{equation}
I_\Delta \big(D^F_{X_\infty}\big)
= \int_\Delta \Psi(\lambda) d\lambda.
\end{equation}
We have the decomposition
\begin{equation}
\Psi = \Psi_\mathrm{pp} + \Psi_\mathrm{sc} + \Psi_\mathrm{ac},
\end{equation}
where $\Psi_\mathrm{pp}$ (resp. $\Psi_\mathrm{sc}$, $\Psi_\mathrm{ac}$)
is a purely point (resp. singularly continuous, absolutely continuous) measure.
We define
\begin{equation}
\mathscr{E}_\lambda =
\Psi_\mathrm{ac}(\lambda)\big(\Omega^\bullet_\mathrm{c}(X_\infty,F)\big)
\subseteq A^\bullet(X_\infty,F).
\end{equation}
An element in $\mathscr{E}_\lambda$ is called generalized eigensection.
It turns out that $\mathscr{E}_\lambda \subseteq \Omega^\bullet(X_{\infty},F)$.
Moreover, the following properties hold,
\begin{itemize}
\item[-] For $\omega_\lambda\in\mathscr{E}_\lambda$,
$D^F_{X_\infty}\omega_\lambda = \lambda \omega_\lambda$.
\item[-] We have  $\mathscr{E}_\lambda \cap L^2\big(\Omega^\bullet_\mathrm{c}(X_\infty,F)\big) = 0$.
In particular, a generalized eigensection is determined by its restriction to $Y_{\R_+}$.
\item[-] For $\omega\in \mathscr{E}^\bullet_\mathrm{ac}(X_{\infty},F)$,
there exists a unique family $\big(\omega_\lambda\in\mathscr{E}_\lambda\big)_{\lambda\in\R}$
such that $\omega = \int_\R \omega_\lambda d\lambda$.
\end{itemize}
In this paper, the word 'generalized eigensection'
is uniquely assigned to the absolutely continuous spectrum.
In other words,
\textbf{a non zero eigensection will not be considered as a generalized eigensection}.

Let
\begin{equation}
\label{eq-def-Pi}
\Pi : \Omega^\bullet(Y,F[du]) \rightarrow \mathscr{H}(Y,F) du \oplus \bigoplus_{\mu>0}
\big((1-du)\mathscr{E}_\mu(Y,F) \oplus (1+du)\mathscr{E}_{-\mu}(Y,F) \big)
\end{equation}
be the orthogonal projection.
Let $D^F_{Y_{\R_+}}$ be the Hodge-de Rham operator with
\begin{equation}
\label{eq-def-dom-YI}
\mathrm{Dom}\big(D^F_{Y_{\R_+}}\big) =
\big\{ \omega \in \Omega^\bullet(Y_{\R_+}, F) \;:\; \omega_0 \in \Ker(\Pi) \big\},
\end{equation}
where $\omega_0$ is defined in the paragraph containing \eqref{eq-ch2-iden-duform}.
Such boundary condition was considered by Atiyah, Patodi and Singer \cite[(2.3)]{APS}.
By the proof of \cite[Proposition 2.5]{APS},
$D^F_{Y_{\R_+}}$ only possesses absolutely continuous spectrum.

Let $J : L^2\big(\Omega^\bullet(Y_{\R_+}, F)\big) \hookrightarrow L^2\big(\Omega^\bullet(X_\infty,F)\big)$
be the push-forward map induced by the embedding $Y_{\R_+} \hookrightarrow X_{\infty}$.
Following \cite[Prop. 4.9]{Muller94},
we define the wave operators
\begin{equation}
W_\pm \big(D^F_{X_\infty}, D^F_{Y_{\R_+}}\big) =
\lim_{t\rightarrow \pm\infty} e^{itD^F_{X_\infty}} J e^{-itD^F_{Y_{\R_+}}} : \;
L^2\big(\Omega^\bullet(Y_{\R_+},F)\big) \rightarrow
L^2\big(\Omega^\bullet(X_\infty,F)\big).
\end{equation}

The following theorem comes from \cite[Theorems 4.1, 4.2]{Muller94}.

\begin{thm}
\label{thm-ch2-muller}
The operator $D^F_{X_{\infty}}$ has no singularly continuous spectrum.
For $t>0$,
the heat operator $\exp\big(-tD^{F,2}_{X_\infty,\mathrm{pp}}\big)$ is of trace class.
The operator $W_\pm\big(D^F_{X_\infty}, D^F_{Y_{\R_+}}\big)$ is unitary.
Its image is $\mathscr{E}^\bullet_\mathrm{ac}(X_\infty,F)$.
Moreover, the following diagram commutes,
\begin{equation}
\xymatrix{
L^2\big(\Omega^\bullet(Y_{\R_+},F)\big)
\ar[d]_{W_\pm \big(D^F_{X_\infty}, D^F_{Y_{\R_+}}\big)}
\ar[r]^{D^F_{Y_{\R_+}}} &
L^2\big(\Omega^\bullet(Y_{\R_+},F)\big)
\ar[d]^{W_\pm \big(D^F_{X_\infty}, D^F_{Y_{\R_+}}\big)} \\
\mathscr{E}^\bullet_\mathrm{ac}(X_\infty,F)
\ar[r]^{D^F_{X_\infty,\mathrm{ac}}} &
\mathscr{E}^\bullet_\mathrm{ac}(X_\infty,F) \;.
}
\end{equation}
\end{thm}

Set
\begin{equation}
C\big(D^F_{X_\infty}, D^F_{Y_{\R_+}}\big) =
W^*_+\big(D^F_{X_\infty}, D^F_{Y_{\R_+}}\big)\circ
W_-\big(D^F_{X_\infty}, D^F_{Y_{\R_+}}\big)
\in \mathrm{End}\Big(L^2\big(\Omega^\bullet(Y_{\R_+},F)\big)\Big).
\end{equation}
Then $C\big(D^F_{X_\infty}, D^F_{Y_{\R_+}}\big)$ commutes with $D^F_{Y_{\R_+}}$.

By \eqref{eq-ch2-eigen-expansion},
\eqref{eq-cliff-ann}
and \eqref{eq-def-dom-YI},
a generalized eigensection of $D^F_{Y_{\R_+}}$ associated with $\lambda\in(-\delta_Y,\delta_Y)$
takes the following form,
\begin{equation}
E_0(\phi,\lambda) =
e^{-i\lambda u} \big(\phi - i\cliffu\phi\big) +
e^{i\lambda u}  \big(\phi + i\cliffu\phi\big),
\quad \text{where}\;  \phi\in\hh(Y,F).
\end{equation}
Since $C\big(D^F_{X_\infty}, D^F_{Y_{\R_+}}\big)$ commutes with $D^F_{Y_{\R_+}}$,
there exists $C(\lambda)\in\End(\hh(Y,F))$ such that
we have the formal identity
\begin{equation}
\label{eq-ch2-def-C-lambda-1}
C\big(D^F_{X_\infty}, D^F_{Y_{\R_+}}\big)
E_0(\phi,\lambda) =
E_0(C(\lambda)\phi,\lambda).
\end{equation}
More precisely,
for
\begin{equation}
\omega = \int_{-\delta_Y}^{\delta_Y} E_0(\phi_\lambda,\lambda)d\lambda
\in L^2\big(\Omega^\bullet(Y_{\R_+},F)\big),
\end{equation}
we have
\begin{equation}
C\big(D^F_{X_\infty}, D^F_{Y_{\R_+}}\big) \omega =
\int_{-\delta_Y}^{\delta_Y} E_0(C(\lambda)\phi_\lambda,\lambda)d\lambda.
\end{equation}
We call $C(\lambda)$ scattering matrix.
We extend the action of $C(\lambda)$ to $\hh(Y,F[du])$ such that
\begin{equation}
\label{eq-ch2-def-C-lambda-2}
C(\lambda) \cliffu = -\cliffu C(\lambda) .
\end{equation}

The following propositions come from \cite[\textsection 4, paragraphs between Lemma 4.16 and Proposition 4.26]{Muller94}.

\begin{prop}
\label{prop-ch2-C-lambda}
The following properties hold,
\begin{itemize}
\item[-] $C(\lambda)$ depends analytically on $\lambda$;
\item[-] $C(\lambda)\in\End(\hh(Y,F[du]))$ is unitary;
\item[-] $C(\lambda)$ preserves $\mathscr{H}^p(Y,F)$ and $\mathscr{H}^p(Y,F)du$
for $p=0,\cdots,n-1$ ;
\item[-] $C(\lambda)C(-\lambda) = \mathrm{Id}$,
in particular, $C(0)^2 = \mathrm{Id}$ .
\end{itemize}
\end{prop}

\begin{prop}
\label{prop-ch2-g-eigen-Dinf}
A generalized eigensection of $D^F_{X_\infty,\mathrm{ac}}$
associated with $\lambda\in(-\delta_Y,\delta_Y)$
takes the following form on $Y_{\R_+}$,
\begin{equation}
\label{eq-prop-ch2-g-eigen-Dinf-1}
e^{-i\lambda u} \big(\phi-i\cliffu\phi\big) +
e^{i\lambda u} C(\lambda) \big(\phi-i\cliffu\phi\big) +
\theta(\phi,\lambda),
\end{equation}
where $\phi\in\hh(Y,F)$ and
$\theta(\phi,\lambda)\in L^2\big(\Omega^\bullet(Y_{\R_+},F)\big)$.
Furthermore, for $u\in\R_+$,
\begin{equation}
\label{eq-prop-ch2-g-eigen-Dinf-2}
{\theta(\phi,\lambda)}_u \perp \hh(Y,F[du]) .
\end{equation}

Conversely, for $\phi\in\hh(Y,F)$ and $\lambda\in\,(-\delta_Y,\delta_Y)$,
there exists a unique generalized eigensection of $D^F_{X_{\infty},\mathrm{ac}}$
satisfying \eqref{eq-prop-ch2-g-eigen-Dinf-1},
which we denote by $E(\phi,\lambda)$.
\end{prop}

In Proposition \ref{prop-ch2-g-eigen-Dinf},
$\phi\in\hh(Y,F)$ is due to the same argument
as in \eqref{eq-ch2-eigen-expansion}.

We remark that $E(\phi,\lambda)$ depends linearly on $\phi$ and analytically on $\lambda$  (\cite[\textsection 4]{Muller94}).
Since $\hh(Y,F)$ is finite dimensional,
there exists $a>0$ such that,
for any $\phi\in\hh(Y,F)$ and $\lambda\in[-\delta_Y/2,\delta_Y/2]$,
\begin{equation}
\label{eq-rem-ch2-g-eigen}
\big\lVert E(\phi,\lambda) \big\rVert_{X_\infty\backslash Y_{\R_+}}
\leqslant a \big\lVert \phi \big\rVert_Y, \quad
\big\lVert E(\phi,\lambda) \big\rVert_{\partial Y_{\R_+}}
\leqslant a \big\lVert \phi \big\rVert_Y.
\end{equation}

\subsection{Extended $L^2$-solutions}
\label{ch2-3}

Set
\begin{align}
\label{eq-ch2-def-hh}
\begin{split}
\hhl(X_\infty,F) & =
L^2\big(\Omega^\bullet(X_\infty,F)\big) \cap \Ker \big( D^{F,2}_{X_\infty} \big),\\
\hh(X_\infty,F)  & =
\big\{(\omega,\hat{\omega})\in \Ker\big(D^{F,2}_{X_\infty}\big) \oplus \hh(Y,F[d u]) \;:\;
\omega^+ = 0, \; \omega^\mathrm{zm} = \pi^*_Y\hat{\omega} \big\},
\end{split}
\end{align}
where $\omega^+$, $\omega^\mathrm{zm}$ are defined in \eqref{eq-ch2-def-zm-nz-1}.
The elements of $\hhl(X_{\infty},F)$ (resp. $\hh(X_{\infty},F)$)
are called $L^2$-solutions (resp. extended $L^2$-solutions) of $D^{F,2}_{X_\infty}$
and $\hhl(X_{\infty},F)$ is naturally identified
as a vector subspace of $\hh(X_{\infty},F)$
by sending $\omega\in\hhl(X_{\infty},F)$ to $(\omega,0)\in\hh(X_{\infty},F)$.
By Proposition \ref{prop-ch2-g-eigen-Dinf},
$\hh(X_{\infty},F)$ is spanned by $\hhl(X_{\infty},F)$
and generalized eigensections of $D^F_{X_\infty}$ associated with $0$, i.e.,
\begin{equation}
\label{eq-ch2-decomp-hh}
\hh(X_\infty,F) = \hhl(X_\infty,F) \oplus \big\{ E(\phi,0) \;:\; \phi\in\hh(Y,F) \big\}.
\end{equation}

\begin{prop}
\label{prop-ch2-closeness}
For $(\omega,\hat{\omega})\in \hh(X_\infty,F)$,
we have
\begin{equation}
\label{eq-prop-ch2-closeness}
d^F\omega = d^{F,*}\omega = 0.
\end{equation}
\end{prop}
\begin{proof}
By \eqref{eq-ch2-eigen-expansion} and \eqref{eq-ch2-def-hh},
$d^F\omega$ and $d^{F,*}\omega$ are $L^2$-integrable and mutually orthogonal.
Then $d^F\omega + d^{F,*}\omega = D^F\omega = 0$ implies \eqref{eq-prop-ch2-closeness}.
The proof of Proposition \ref{prop-ch2-closeness} is completed. 
\end{proof}

By Proposition \ref{prop-ch2-closeness} and \eqref{eq-ch2-eigen-expansion},
for $(\omega,\hat{\omega})\in \hh(X_\infty,F)$,
we have
\begin{equation}
\label{eq-ch2-sol-expansion-1}
\omega\big|_{Y_{\R_+}} =
\pi^*_Y \hat{\omega} +
\sum_{\mu\neq 0,\; \mu\in\Sp(D^F_Y)} e^{-|\mu| u} \big( \tau_{\mu,1} - du \wedge \tau_{\mu,2} \big),
\end{equation}
with $\tau_{\mu,1}\in \Omega^\bullet(Y,F)$ and $\tau_{\mu,2}\in \Omega^{\bullet-1}(Y,F)$ satisfying
\begin{equation}
\label{eq-ch2-sol-expansion-2}
d^F_Y\tau_{\mu,1} = d^{F,*}_Y\tau_{\mu,2} = 0, \hspace{5mm}
d^{F,*}_Y\tau_{\mu,1} = |\mu|\tau_{\mu,2}, \hspace{5mm}
d^F_Y\tau_{\mu,2} = |\mu|\tau_{\mu,1}.
\end{equation}

We construct linear maps
\begin{equation}
\label{eq-ch2-def-R-1}
\mathscr{R}_{d^F} : \;
\hh(X_{\infty},F) \rightarrow \Omega^{\bullet-1}(Y_{\R_+},F),\quad
\mathscr{R}_{d^{F,*}} : \;
\hh(X_{\infty},F) \rightarrow \Omega^{\bullet+1}(Y_{\R_+},F)
\end{equation}
as follows:
for $(\omega,\hat{\omega})\in \hh(X_\infty,F)$ satisfying \eqref{eq-ch2-sol-expansion-1},
set
\begin{equation}
\label{eq-ch2-def-R-2}
\mathscr{R}_{d^F}(\omega,\hat{\omega}) =
\sum_{\mu\neq 0} \frac{1}{|\mu|}e^{-|\mu| u}\tau_{\mu,2} ,\quad
\mathscr{R}_{d^{F,*}}(\omega,\hat{\omega}) =
\sum_{\mu\neq 0} \frac{1}{|\mu|}e^{-|\mu| u}du\wedge\tau_{\mu,1}.
\end{equation}

\begin{prop}
\label{prop-ch2-R}
The following identities hold on $Y_{\R_+}$,
\begin{align}
\begin{split}
& d^F\mathscr{R}_{d^F}(\omega,\hat{\omega}) =
\omega\big|_{Y_{\R_+}} - \pi_Y^*\hat{\omega}, \quad
d^{F,*}\mathscr{R}_{d^F}(\omega,\hat{\omega}) = 0, \\
& d^{F,*}\mathscr{R}_{d^{F,*}}(\omega,\hat{\omega}) =
\omega\big|_{Y_{\R_+}} - \pi_Y^*\hat{\omega }, \quad
d^F\mathscr{R}_{d^{F,*}}(\omega,\hat{\omega}) = 0.
\end{split}
\end{align}
\end{prop}
\begin{proof}
These identities follow from
\eqref{eq-ch2-sol-expansion-1},
\eqref{eq-ch2-sol-expansion-2}
and \eqref{eq-ch2-def-R-2}.
The proof of Proposition \ref{prop-ch2-R} is completed.
\end{proof}

Set
\begin{equation}
\label{eq-ch2-def-LL}
\LL^\bullet = \big\{ \hat{\omega}\in\hh(Y,F[du])  \;:\;
\text{ there exists } \omega \text{ such that } (\omega,\hat{\omega})\in\hh(X_\infty,F) \big\}.
\end{equation}
The elements of $\LL^\bullet$ are called limiting values of $\hh(X_\infty,F)$.
Let $P_{\LL} : \hh(Y,F[du]) \rightarrow \LL^\bullet$ be the orthogonal projection.
Recall that $C(\lambda)$ is defined by
\eqref{eq-ch2-def-C-lambda-1} and \eqref{eq-ch2-def-C-lambda-2}.
We denote by $C=C(0)$.
By Proposition \ref{prop-ch2-C-lambda} and
\eqref{eq-prop-ch2-g-eigen-Dinf-1}, \eqref{eq-prop-ch2-g-eigen-Dinf-2} for $\lambda=0$,
we have
\begin{equation}
\label{eq-ch2-LL-C}
\LL^\bullet = \mathrm{Im}(C+\Id) = \Ker(C - \Id), \quad
C = 2P_{\LL} - \Id.
\end{equation}
By Proposition \ref{prop-ch2-C-lambda}
and \eqref{eq-ch2-LL-C},
we have the decomposition of vector spaces
\begin{equation}
\label{eq-ch2-def-LL-bd}
\LL^p = \LL^p_\mathrm{abs} \oplus \LL^p_\mathrm{rel}
\end{equation}
with $\LL^p_\mathrm{abs}\subseteq\mathscr{H}^p(Y,F)$ and
$\LL^p_\mathrm{rel}\subseteq\mathscr{H}^{p-1}(Y,F)du$.
From \eqref{eq-ch2-def-C-lambda-2} and \eqref{eq-ch2-LL-C},
\begin{equation}
\label{eq-ch2-def-LL-bd-2}
\LL^{p,\perp}_\mathrm{abs} = \cntrtu \LL^{p+1}_\mathrm{rel}
\subseteq \mathscr{H}^p(Y,F)
\end{equation}
is the orthogonal complement of $\LL^p_\mathrm{abs}$ in $\mathscr{H}^p(Y,F)$.

By \eqref{eq-ch2-def-hh} and \eqref{eq-ch2-def-LL},
we have the following short exact sequence,
\begin{equation}
\label{eq-ch2-ses-LL}
0 \longrightarrow
\hhl(X_\infty,F) \longrightarrow
\hh(X_\infty,F) \longrightarrow
\LL^\bullet \longrightarrow 0.
\end{equation}
Set
\begin{equation}
\label{eq-ch2-def-hh-bd}
\hhar(X_{\infty},F) = \big\{(\omega,\hat{\omega})\in\hh(X_{\infty},F):\;\hat{\omega}\in \LL^\bullet_\mathrm{abs/rel} \big\}.
\end{equation}
By \eqref{eq-ch2-ses-LL} and \eqref{eq-ch2-def-hh-bd},
we get the following short exact sequence,
\begin{equation}
\label{eq-ch2-ses-hh-bd}
0 \longrightarrow
\hhl(X_\infty,F) \longrightarrow
\hhar(X_\infty,F) \longrightarrow
\LL^\bullet_\mathrm{abs/rel} \longrightarrow 0.
\end{equation}

\section{Asymptotics of the spectrum of Hodge-Laplacians}
\label{ch3}

In this section,
we study the asymptotics of the spectrum of the Hodge-Laplacian.
In \textsection \ref{ch3-1},
we introduce the gluing of two manifolds,
which will be the central object in the whole paper.
In \textsection \ref{ch3-2},
we introduce a model,
which could be viewed as the limit of the Hodge-Laplacian in question.
In \textsection \ref{ch3-3} and \textsection \ref{ch3-4},
we study the asymptotics of the spectrum of the Hodge-Laplacian in question.
In \textsection \ref{ch3-5},
we extend the results in \textsection \ref{ch3-3} and \textsection \ref{ch3-4}
to manifolds with boundaries.

\subsection{Gluing two manifolds}
\label{ch3-1}

For $R>2$,
set
\begin{align}
\begin{split}
& Z_{1,R} = Z_1 \cup_Y Y_{[0,R]}, \quad
Z_{2,R} = Z_2 \cup_Y Y_{[-R,0]} ,\\
& Z_{1,\infty} = Z_1 \cup_Y Y_{[0,\infty)}, \quad
Z_{2,\infty} = Z_2 \cup_Y Y_{(-\infty,0]},
\end{split}
\end{align}
where we identify
$\partial Z_j\simeq Y$ ($j=1,2$) with $Y_0:=Y_{\{0\}}$.
Set (cf. Figure \ref{lab-fig-intro})
\begin{equation}
Z_R = Z_{1,R} \cup_Y Z_{2,R},
\end{equation}
We have the canonical embeddings
\begin{equation}
\label{eq-ch3-embedding}
Z_{1,2R} \subseteq Z_R, \quad
Z_{2,2R} \subseteq Z_R.
\end{equation}
We will use the following coordinates on their cylindrical parts:
$(u_1,y)\in Y_{[0,R]}\subseteq Z_{1,2R}$,
$(u_2,y)\in Y_{[-R,0]}\subseteq Z_{2,2R}$ and
$(u,y)\in Y_{[-R,R]}\subseteq Z_R$.
Under the embedding \eqref{eq-ch3-embedding},
we have the following coordinate transformations
\begin{equation}
\label{eq-ch3-coordinates}
(u,y) \leftrightarrow (u_1-R,y) \leftrightarrow (u_2+R,y).
\end{equation}
We will also use the following canonical embedding:
for $R'\leqslant R$,
\begin{equation}
Z_{j,R'}\subseteq Z_{j,R}, \quad \text{for } j=1,2,
\end{equation}
which is compatible with the coordinates $(u_j,y)$.

\subsection{Model of eigenspaces associated with small eigenvalues}
\label{ch3-2}

For $j=1,2$,
let $\hhl(Z_{j,\infty},F)$ (resp. $\hh(Z_{j,\infty},F)$)
be the vector space of $L^2$-solutions (resp. extended $L^2$-solutions)
of $D^{F,2}_{Z_{j,\infty}}$ (see \eqref{eq-ch2-def-hh}).
Set
\begin{equation}
\label{eq-ch3-def-hh-12}
\hh(Z_{12,\infty},F) =
\big\{(\omega_1,\omega_2,\hat{\omega}) \;:\;
(\omega_j,\hat{\omega})\in\hh(Z_{j,\infty},F) \quad \text{for} \, j=1,2 \big\}.
\end{equation}
For $R>2$,
we equip $\hh(Z_{12,\infty},F)$ with the following metric,
\begin{equation}
\label{eq-ch3-def-metric-hh}
\big\lVert (\omega_1,\omega_2,\hat{\omega}) \big\rVert^2_{\hh(Z_{12,\infty},F),R} =
\big\lVert \omega_1 \big\rVert^2_{Z_{1,R}} +
\big\lVert \omega_2 \big\rVert^2_{Z_{2,R}}.
\end{equation}
We will omit the subscript $R$ if $R=0$.
By Lemma \ref{lem-ch2-zm-nz} and \eqref{eq-ch2-def-hh},
there exists $a>0$ such that
\begin{equation}
\label{eq-ch3-Hnorm-increasing}
\big\lVert \,\cdot\, \big\rVert^2_{\hh(Z_{12,\infty},F),R} \leqslant
\big( 1+aR \big) \big\lVert \,\cdot\, \big\rVert^2_{\hh(Z_{12,\infty},F)}.
\end{equation}

By \eqref{eq-ch3-def-hh-12},
we have the following injection,
\begin{align}
\label{eq-ch3-inj-hh-12}
\begin{split}
\hhl(Z_{1,\infty},F) \oplus \hhl(Z_{2,\infty},F) & \rightarrow \hh(Z_{12,\infty},F) \\
(\omega_1,\omega_2) & \mapsto (\omega_1,\omega_2,0).
\end{split}
\end{align}
For $j=1,2$,
let $\LL^\bullet_j$ be the set of limiting values of $\hh(Z_{j,\infty},F)$
(see \eqref{eq-ch2-def-LL}).
We have the following surjection,
\begin{align}
\label{eq-ch3-filt-hh-12}
\begin{split}
\hh(Z_{12,\infty},F) & \rightarrow \LL^\bullet_1 \cap \LL^\bullet_2 \\
(\omega_1,\omega_2,\hat{\omega}) & \mapsto \hat{\omega}.
\end{split}
\end{align}
From \eqref{eq-ch2-ses-LL},
\eqref{eq-ch3-inj-hh-12} and
\eqref{eq-ch3-filt-hh-12},
we obtain the following short exact sequence
\begin{equation}
\label{eq-ch3-ses-hh}
0 \rightarrow \hhl(Z_{1,\infty},F) \oplus \hhl(Z_{2,\infty},F)
\rightarrow \hh(Z_{12,\infty},F)
\rightarrow \LL^\bullet_1 \cap \LL^\bullet_2
\rightarrow 0 .
\end{equation}

For $\lambda\in (-\delta_Y,0)\cup(0,\delta_Y)$
(see \eqref{eq-ch2-def-deltaY}),
set
\begin{align}
\label{eq-ch3-def-Einf}
\begin{split}
\mathscr{E}_\lambda(Z_{1,\infty}, F) & =
\big\{ (\omega, \omega^\mathrm{zm})
\in \Omega^\bullet(Z_{1,\infty},F) \times \Omega^\bullet(Y_{[0,+\infty)},F) \;:\; \\
& \hspace{10mm} \omega \text{ is a generalized eigensection of }  D^F_{Z_{1,\infty},\mathrm{ac}}
\text{ associated with } \lambda, \\
& \hspace{70mm} \omega^\mathrm{zm} \text{ is the zeromode of } \omega \big\},
\end{split}
\end{align}
where the zeromode is defined in \eqref{eq-ch2-def-zm-nz-1}.
We construct $\mathscr{E}_\lambda(Z_{2,\infty}, F)$ in the same way,
i.e., \eqref{eq-ch3-def-Einf} with $Z_{1,\infty}$ replaced by $Z_{2,\infty}$
and $Y_{[0,+\infty)}$ replaced by $Y_{(-\infty,0]}$.
For $R > 2$, set
\begin{align}
\label{eq-ch3-def-Einf-R}
\begin{split}
\mathscr{E}_{\lambda,R}(Z_{12,\infty},F)
= \big\{ (\omega_1, \omega_1^\mathrm{zm}, \omega_2, \omega_2^\mathrm{zm}) & \;:\;
(\omega_j,\omega_j^\mathrm{zm})\in\mathscr{E}_\lambda(Z_{j,\infty},F) \text{ for } j=1,2,\\
& \hspace{10mm} \omega_1^\mathrm{zm}\big|_{Y_{[0,2R]}} = \omega_2^\mathrm{zm}\big|_{Y_{[-2R,0]}}
\text{ under \eqref{eq-ch3-coordinates}} \big\}.
\end{split}
\end{align}

For $j=1,2$,
let $C_j(\lambda)\in\End(\hh(Y,F[du]))$
be the scattering matrix associated with $D^F_{Z_{j,\infty}}$
(see \eqref{eq-ch2-def-C-lambda-1} and \eqref{eq-ch2-def-C-lambda-2}).
From Proposition \ref{prop-ch2-C-lambda},
we have
\begin{equation}
\label{eq-ch3-def-C12}
C_{12}(\lambda) := C_2^{-1}(\lambda)C_1(\lambda)
\in \End(\hh(Y,F)) \oplus \End(\hh(Y,F)du).
\end{equation}
For $R > 2$,
set
\begin{equation}
\label{eq-ch3-def-Lambda}
\Lambda_R = \big\{ \lambda\in\R \;:\;
\det \big( e^{4i\lambda R}C_{12}(\lambda)\big|_{\hh(Y,F)} -\Id \big) = 0 \big\},
\end{equation}
which is viewed as a multiset.

\begin{lemma}
\label{lem-ch3-LambdaR}
We have
\begin{equation}
\big\{\lambda \in\R \;:\;  \mathscr{E}_{\lambda,R}(Z_{12,\infty},F)\neq 0 \big\} = \Lambda_R.
\end{equation}
\end{lemma}
\begin{proof}
Let $(\omega_1, \omega_1^\mathrm{zm}, \omega_2, \omega_2^\mathrm{zm})
\in\mathscr{E}_{\lambda,R}(Z_{12,\infty},F)$.
By Proposition \ref{prop-ch2-g-eigen-Dinf},
we have
\begin{align}
\label{eq1-pf-lem-ch3-LambdaR}
\begin{split}
\omega_1^\mathrm{zm}\big|_{Y_{[0,2R]}} & =
e^{-i\lambda u_1} \big(\phi_1-i\cliffu\phi_1\big) +
e^{i\lambda u_1} C_1(\lambda) \big(\phi_1-i\cliffu\phi_1\big), \\
\omega_2^\mathrm{zm}\big|_{Y_{[-2R,0]}} & =
e^{-i\lambda u_2} \big(\phi_2-i\cliffu\phi_2\big) +
e^{i\lambda u_2} C_2(\lambda) \big(\phi_2-i\cliffu\phi_2\big),
\end{split}
\end{align}
with $\phi_1,\phi_2\in\hh(Y,F)$.
By \eqref{eq-ch3-def-Einf-R},
we have
$\omega_1^\mathrm{zm}\big|_{Y_{[0,2R]}} = \omega_2^\mathrm{zm}\big|_{Y_{[-2R,0]}}$.
Now applying \eqref{eq-ch3-coordinates} and \eqref{eq1-pf-lem-ch3-LambdaR},
we get
\begin{align}
\label{eq2-pf-lem-ch3-LambdaR}
\begin{split}
\phi_1-i\cliffu\phi_1 & =
e^{2iR\lambda}\big(\phi_2-i\cliffu\phi_2\big),\\
C_1(\lambda) \big(\phi_1-i\cliffu\phi_1\big) & =
e^{-2iR\lambda}C_2(\lambda) \big(\phi_2-i\cliffu\phi_2\big).
\end{split}
\end{align}
By \eqref{eq-ch3-def-C12} and \eqref{eq2-pf-lem-ch3-LambdaR},
we have
\begin{equation}
\label{eq3-pf-lem-ch3-LambdaR}
e^{4iR\lambda}C_{12}(\lambda)\big(\phi_2-i\cliffu\phi_2\big)
= \phi_2-i\cliffu\phi_2.
\end{equation}
By Proposition \ref{prop-ch2-C-lambda} and \eqref{eq-ch3-def-C12},
equation \eqref{eq3-pf-lem-ch3-LambdaR} is equivalent to
\begin{equation}
\label{eq4-pf-lem-ch3-LambdaR}
e^{4iR\lambda}C_{12}(\lambda)\phi_2 = \phi_2.
\end{equation}
Thus we obtained a map
\begin{equation}
\label{eq5-pf-lem-ch3-LambdaR}
\mathscr{E}_{\lambda,R}(Z_{12,\infty},F) \rightarrow
\mathrm{Ker} \big( e^{4i\lambda R}C_{12}(\lambda)\big|_{\hh(Y,F)} -\Id \big)
\end{equation}
sending $(\omega_1,\omega_1^\mathrm{zm},\omega_2,\omega_2^\mathrm{zm}) \in \mathscr{E}_{\lambda,R}(Z_{12,\infty},F)$
to $\phi_2\in\hh(Y,F)$ determined by \eqref{eq1-pf-lem-ch3-LambdaR}.
Moreover,
by Proposition \ref{prop-ch2-g-eigen-Dinf},
the map \eqref{eq5-pf-lem-ch3-LambdaR} is bijective.
This completes the proof of Lemma \ref{lem-ch3-LambdaR}.
\end{proof}

\subsection{Approximating the kernels}
\label{ch3-3}

Let $\psi \in \smooth(\R)$ such that
\begin{equation}
\label{eq-ch3-def-chi-1}
0\leqslant \psi'\leqslant 3, \quad
\psi(u) = 0 \quad\text{for } u\leqslant 1/4, \quad
\psi(u) = 1 \quad\text{for } u\geqslant 3/4.
\end{equation}
For $j=1,2$,
we construct $\chi_{j,R} \in \smooth([-R,R])$ as follows,
\begin{equation}
\label{eq-ch3-def-chi-2}
\chi_{j,R}(u) = \psi\big((-1)^j u/R\big) ,
\end{equation}
which we view as a function on $Y_{[-R,R]}$,
i.e., $\chi_{j,R}(u,y) = \chi_{j,R}(u)$ for $(u,y)\in Y_{[-R,R]}$.

Let
\begin{align}
\label{eq-ch3-R}
\begin{split}
\mathscr{R}_{d^F,1} & : \;
\hh(Z_{1,\infty},F) \rightarrow \Omega^{\bullet-1}(Y_{[0,+\infty)},F), \\
\mathscr{R}_{d^{F,*},1} & : \;
\hh(Z_{1,\infty},F) \rightarrow \Omega^{\bullet+1}(Y_{[0,+\infty)},F),\\
\mathscr{R}_{d^F,2} & : \;
\hh(Z_{2,\infty},F) \rightarrow \Omega^{\bullet-1}(Y_{(-\infty,0]},F), \\
\mathscr{R}_{d^{F,*},2} & : \;
\hh(Z_{2,\infty},F) \rightarrow \Omega^{\bullet+1}(Y_{(-\infty,0]},F)
\end{split}
\end{align}
be the maps in \eqref{eq-ch2-def-R-1}
with $X_\infty$ replaced by $Z_{1,\infty}$ and $Z_{2,\infty}$.
We compose the maps in \eqref{eq-ch3-R}
with the restriction maps induced by
\begin{equation}
\label{eq-ch3-YR}
Y_{[-R,R]} \simeq Y_{[0,2R]} \subseteq Y_{[0,+\infty)},\hspace{5mm}
Y_{[-R,R]} \simeq Y_{[-2R,0]} \subseteq Y_{(-\infty,0]}.
\end{equation}
Now these maps take values in $\Omega^\bullet(Y_{[-R,R]},F)$.

We construct
\begin{equation}
F_{Z_R}, G_{Z_R} :\; \hh(Z_{12,\infty},F) \rightarrow \Omega^\bullet(Z_R,F)
\end{equation}
as follows,
\begin{align}
\label{eq-ch3-def-FG}
\begin{split}
F_{Z_R}(\omega_1,\omega_2,\hat{\omega})\big|_{Z_{j,0}} = & \;
G_{Z_R}(\omega_1,\omega_2,\hat{\omega})\big|_{Z_{j,0}} = \omega_j, \quad \text{for } j=1,2,\\
F_{Z_R}(\omega_1,\omega_2,\hat{\omega})\big|_{Y_{[-R,R]}} = & \;
\pi_Y^*\hat{\omega} +
\sum_{j=1}^2 d^F \big( \chi_{j,R}\;\mathscr{R}_{d^F, j}(\omega_j,\hat{\omega}) \big), \\
G_{Z_R}(\omega_1,\omega_2,\hat{\omega})\big|_{Y_{[-R,R]}} = & \;
\pi_Y^*\hat{\omega} +
\sum_{j=1}^2 d^{F,*} \big( \chi_{j,R}\;\mathscr{R}_{d^{F,*}, j}(\omega_j,\hat{\omega}) \big).
\end{split}
\end{align}
By Proposition \ref{prop-ch2-R},
$F_{Z_R}$ and $G_{Z_R}$ are well-defined.
Moreover,
we have
\begin{equation}
\label{eq-ch3-FG-close}
d^F F_{Z_R}(\omega_1,\omega_2,\hat{\omega}) =
d^{F,*} G_{Z_R}(\omega_1,\omega_2,\hat{\omega}) = 0 .
\end{equation}

Let $\phi_R : [-R-1,R+1] \rightarrow [-1,1]$ be an odd smooth function such that
$\phi_{R}'(u) > 0$ and $\phi_R(u) = u+R$ for $u \in [-R-1,-R-1/2]$.
We construct a diffeomorphism $\varphi_R : Z_R \rightarrow Z$ as follows,
\begin{equation}
\label{eq-ch3-def-varphiR}
\varphi_R = \mathrm{Id}\quad\text{on } Z_R\backslash Y_{[-R-1,R+1]}, \quad
\varphi_R(u,y) = (\phi_R(u),y)\quad \text{for } (u,y)\in Y_{[-R-1,R+1]}.
\end{equation}
Then $\varphi_R$ induces an isomorphism
$\varphi_{R,*}: H^\bullet(Z_R,Y) \rightarrow H^\bullet(Z,F)$.

\begin{prop}
\label{prop-ch3-F-ind}
For $(\omega_1,\omega_2,\hat{\omega}) \in \hh(Z_{12,\infty},F)$ with $\hat{\omega}\in\hh(Y,F)$,
under the identification $\varphi_{R,*}$,
the cohomology class
$\big[ F_{Z_R}(\omega_1,\omega_2,\hat{\omega}) \big] \in H^\bullet(Z,F)$
is independent of $R$.
\end{prop}
\begin{proof}
Let $R'\in[R,7R/6]$.
Let $\phi_{R,R'} : [-R,R] \rightarrow [-R',R']$ be an odd function such that
\begin{equation}
\label{eq1-pf-prop-ch3-F-ind}
\phi_{R,R'}(u)  = \left\{
\begin{array}{ll}
u - R' + R & \text{ if } u\in[-R,-\frac{1}{8}R] ,\\
u -(R'-R)\chi_{1,R/8}(u) & \text{ if } u\in[-\frac{1}{8}R,0].
\end{array}\right.
\end{equation}
Then $\phi_{R,R'}'(u)>0$.
Let $\varphi_{R,R'} : Z_R \rightarrow Z_{R'}$ be the diffeomorphism defined by $\phi_{R,R'}$
in the same way as \eqref{eq-ch3-def-varphiR}.
Then $\varphi_{R,R'}$ is homotopic to $\varphi_{R'}^{-1}\circ\varphi_R$.
Let $\mu\in\Omega^\bullet(Z_R,F)$ such that
\begin{equation}
\label{eq2-pf-prop-ch3-F-ind}
\mu\big|_{Z_R\backslash Y_{[-R,R]}} = 0 ,\;\;
\mu\big|_{Y_{[-R,R]}} = \sum_{j=1}^2 \big( \chi_{j,R} - \phi_{R,R'}^*\chi_{j,R'} \big)
\mathscr{R}_{d^F, j}(\omega_j,\hat{\omega})  .
\end{equation}
By \eqref{eq-ch3-def-FG},
\eqref{eq1-pf-prop-ch3-F-ind} and
\eqref{eq2-pf-prop-ch3-F-ind},
we have
\begin{equation}
\label{eq3-pf-prop-ch3-F-ind}
F_{Z_R}(\omega_1,\omega_2,\hat{\omega}) -
\varphi^*_{R,R'} F_{Z_{R'}}(\omega_1,\omega_2,\hat{\omega}) = d^F\mu,
\end{equation}
which implies
$\varphi_{R,*}\big[ F_{Z_R}(\omega_1,\omega_2,\hat{\omega}) \big] =
\varphi_{R',*}\big[ F_{Z_{R'}}(\omega_1,\omega_2,\hat{\omega}) \big] \in H^\bullet(Z,F)$.
The proof of Proposition \ref{prop-ch3-F-ind} is completed.
\end{proof}

For $\alpha>0$,
we denote by $\mathscr{O}(\alpha)$ a number in $[-\alpha\beta,\alpha\beta]$
with $\beta>0$ independent of $R$.
\textbf{In the sequel, $R$ is always supposed to be large enough.}

\begin{prop}
\label{prop-ch3-approx-f-g-total}
There exists $a>0$ such that
for $(\omega_1,\omega_2,\hat{\omega})\in\hh(Z_{12,\infty},F)$,
we have
\begin{equation}
\label{eq-prop-ch3-approx-f-g-total}
\big\lVert F_{Z_R}(\omega_1,\omega_2,\hat{\omega})\big\rVert_{Z_R}
= \Big( 1 + \mathscr{O}\big(e^{-aR}\big) \Big)
\big\lVert (\omega_1,\omega_2,\hat{\omega}) \big\rVert_{\hh(Z_{12,\infty},F),R}.
\end{equation}
\end{prop}
\begin{proof}
By \eqref{eq-ch3-def-FG} on $Y_{[0,R]}\subseteq Z_{1,R}$,
$F_{Z_R}(\omega_1,\omega_2,\hat{\omega}) - \omega_1$ vanishes on $Z_{1,0}$.
By Proposition \ref{prop-ch2-R} and \eqref{eq-ch3-def-FG},
\begin{align}
\label{eq1-pf-prop-ch3-approx-f-g-total}
\begin{split}
& \big( F_{Z_R}(\omega_1,\omega_2,\hat{\omega}) - \omega_1 \big)\big|_{Y_{[0,R]}}
= d^F \big( \chi_{1,R}\;\mathscr{R}_{d^F, 1}(\omega_1,\hat{\omega}) \big)
+ \pi_Y^*\hat{\omega} - \omega_1 \\
& = \Big( \frac{\partial}{\partial u} \chi_{1,R} \Big)
du \wedge \mathscr{R}_{d^F, 1}(\omega_1,\hat{\omega})
+ \big( \chi_{1,R} - 1 \big)(\omega_1 - \pi_Y^*\hat{\omega}).
\end{split}
\end{align}
By \eqref{eq-ch3-def-chi-1} and \eqref{eq-ch3-def-chi-2},
$\frac{\partial}{\partial u} \chi_{1,R}$ is bounded by $1$
and with support in $Y_{[-3R/4,-R/4]}$,
$\chi_{1,R}-1$ is bounded by $1$
and with support in $Y_{[-3R/4,0]}$.
Then we have
\begin{equation}
\label{eq2-pf-prop-ch3-approx-f-g-total}
\big\lVert F_{Z_R}(\omega_1,\omega_2,\hat{\omega}) - \omega_1 \big\rVert_{Z_{1,R}}
\leqslant
\big\lVert \mathscr{R}_{d^F, 1}(\omega_1,\hat{\omega}) \big\rVert_{Y_{[-3R/4,-R/4]}} +
\big\lVert \omega_1 - \pi_Y^*\hat{\omega} \big\rVert_{Y_{[-3R/4,0]}}.
\end{equation}
By \eqref{eq-ch2-def-R-2},
we have
\begin{equation}
\label{eq3-pf-prop-ch3-approx-f-g-total}
\big\lVert \mathscr{R}_{d^F, 1}(\omega_1,\hat{\omega}) \big\rVert^2_{Y_{[-3R/4, -R/4]}}
\leqslant
\delta^{-2}_Ye^{-\frac{1}{2}\delta_YR} \big\lVert \omega_1 \big\rVert^2_{\partial Z_{1,0}}.
\end{equation}
By Lemma \ref{lem-ch2-zm-nz},
\eqref{eq-ch2-nz-decreasing},
\eqref{eq-ch2-def-hh} and
\eqref{eq-ch3-YR},
we have
\begin{align}
\label{eq4-pf-prop-ch3-approx-f-g-total}
\begin{split}
\big\lVert \omega_1 - \pi_Y^*\hat{\omega} \big\rVert^2_{Y_{[-3R/4,0]}}
& \leqslant
2 \delta^{-1}_Y \big( 1-e^{-\frac{3}{8}\delta_YR} \big)^{-2}
\big\lVert \omega_1 - \pi_Y^*\hat{\omega} \big\rVert^2_{Y_{-3R/4} \cup Y_0} \\
& \leqslant
4 \delta^{-1}_Y \big( 1-e^{-\frac{3}{8}\delta_YR} \big)^{-2} e^{-\frac{1}{2}\delta_Y R}
\big\lVert \omega_1 \big\rVert^2_{\partial Z_{1,0}}.
\end{split}
\end{align}

Let $\big\lVert\cdot\big\rVert_{1,Z_{1,0}}$
be the $H^1$-Sobolev norm on $\smooth(Z_{1,0},F)$
induced by $g^{TZ}$ and $h^F$.
By the trace theorem,
$\big\lVert\cdot\big\rVert^2_{\partial Z_{1,0}} =
\mathscr{O}(1) \big\lVert\cdot\big\rVert^2_{1,Z_{1,0}}$.
By ellipticity of $D^F_{Z_{1,0}}$,
we have
$\big\lVert\cdot\big\rVert^2_{1,Z_{1,0}} =
\mathscr{O}(1)\big( \big\lVert\cdot\big\rVert^2_{Z_{1,0}}
+ \big\lVert D^F_{Z_{1,0}}\cdot\big\rVert^2_{Z_{1,0}}\big)$.
Applying these estimates to $\omega_1$
and using the fact that $D^F_{Z_{1,\infty}}\omega_1=0$,
we get
\begin{equation}
\label{eq5-pf-prop-ch3-approx-f-g-total}
\big\lVert \omega_1 \big\rVert_{\partial Z_{1,0}} =
\mathscr{O}(1) \big\lVert \omega_1 \big\rVert_{1,Z_{1,0}} =
\mathscr{O}(1) \big\lVert \omega_1 \big\rVert_{Z_{1,0}}
\end{equation}

Set $a=\delta_Y/5$.
By \eqref{eq2-pf-prop-ch3-approx-f-g-total}-\eqref{eq5-pf-prop-ch3-approx-f-g-total},
we have
\begin{equation}
\label{eq6-pf-prop-ch3-approx-f-g-total}
\big\lVert F_{Z_R}(\omega_1,\omega_2,\hat{\omega}) - \omega_1 \big\rVert_{Z_{1,R}}
= \mathscr{O}\big(e^{-aR}\big) \big\lVert \omega_1 \big\rVert_{Z_{1,R}}.
\end{equation}
For the same reasons,
we have
\begin{equation}
\label{eq7-pf-prop-ch3-approx-f-g-total}
\big\lVert F_{Z_R}(\omega_1,\omega_2,\hat{\omega}) - \omega_2 \big\rVert_{Z_{2,R}}
= \mathscr{O}\big(e^{-aR}\big) \big\lVert \omega_2 \big\rVert_{Z_{2,R}}.
\end{equation}
Then \eqref{eq-prop-ch3-approx-f-g-total} follows from
\eqref{eq-ch3-def-metric-hh},
\eqref{eq6-pf-prop-ch3-approx-f-g-total}
and \eqref{eq7-pf-prop-ch3-approx-f-g-total}.
The proof of Proposition \ref{prop-ch3-approx-f-g-total} is completed.
\end{proof}

For $\omega \in \Omega^\bullet(Z_R,F)$,
we denote
\begin{equation}
\big\Vert \omega \big\rVert_{\continu\!,Z_R} =
\sup_{x\in Z_R} \big| \omega_x \big|_{\Lambda^\bullet(T_x^{*}X)\otimes F_x}.
\end{equation}

\begin{prop}
\label{prop-ch3-sobolev}
Let $k$ be an integer greater than $\frac{n}{2}$.
For $\omega \in \Omega^\bullet(Z_R,F)$,
\begin{equation}
\label{eq-prop-ch3-sobolev}
\big\lVert \omega \big\rVert_{\continu\!,Z_R} =
\mathscr{O}(1) \big( \big\lVert \omega \big\rVert_{Z_R}
+ \big\lVert D^{F,k}_{Z_R} \omega \big\rVert_{Z_R} \big).
\end{equation}
\end{prop}
\begin{proof}
For any $(u,y)\in Y_{[-R,R]}$,
we apply the Sobolev inequality on $Y_{[u-1,u+1]}$.
Since $g^{TZ}$ and $h^F$ are product on $Y_{[u-1,u+1]}$,
there exists $C_1>0$ independent of $R>1$ such that
for any $\omega\in\Omega^\bullet(Z_R,F)$,
we have
\begin{equation}
\label{eq1-pf-prop-ch3-sobolev}
\big| \omega_{(u,y)} \big|
\leqslant C_1\Big( \big\lVert\omega\big\rVert_{Y_{[u-1,u+1]}}
+ \big\lVert D^{F,k}_{Z_R}\omega\big\rVert_{Y_{[u-1,u+1]}} \Big).
\end{equation}
Combing \eqref{eq1-pf-prop-ch3-sobolev}
with the Sobolev inequalities on $Z_{1,1}$ and $Z_{2,1}$,
we obtain \eqref{eq-prop-ch3-sobolev}.
This completes the proof of Proposition \ref{prop-ch3-sobolev}.
\end{proof}

Let $P^{\Ker\left(D^{F,2}_{Z_R}\right)} : \Omega^\bullet(Z_R,F) \rightarrow \Ker\big(D^{F,2}_{Z_R}\big)$ be the orthogonal projection.
Set
\begin{align}
\label{eq-ch3-def-FG-harmonic}
\begin{split}
& \mathscr{F}_{Z_R} = P^{\Ker\left(D^{F,2}_{Z_R}\right)}\circ F_{Z_R}:
\hh(Z_{12,\infty},F) \rightarrow \Ker\big(D^{F,2}_{Z_R}\big), \\
& \mathscr{G}_{Z_R} = P^{\Ker\left(D^{F,2}_{Z_R}\right)}\circ G_{Z_R}:
\hh(Z_{12,\infty},F) \rightarrow \Ker\big(D^{F,2}_{Z_R}\big).
\end{split}
\end{align}

\begin{prop}
\label{prop-ch3-approx-harmonic-tot-inj}
There exists $a>0$ such that
for $(\omega_1,\omega_2,\hat{\omega})\in\hh(Z_{12,\infty},F)$,
\begin{equation}
\label{eq-prop-ch3-approx-harmonic-tot-inj}
\big\lVert (F_{Z_R}-\mathscr{F}_{Z_R})(\omega_1,\omega_2,\hat{\omega}) \big\rVert_{\continu\!,Z_R}
= \mathscr{O}\big(e^{-aR}\big)
\big\lVert (\omega_1,\omega_2,\hat{\omega}) \big\rVert_{\hh(Z_{12,\infty},F),R}.
\end{equation}
As a consequence,
$\mathscr{F}_{Z_R} : \hh(Z_{12,\infty},F) \rightarrow \Ker\big(D^{F,2}_{Z_R}\big)$
is injective.
\end{prop}
\begin{proof}
By Proposition \ref{prop-ch2-R} and \eqref{eq-ch3-def-FG},
we have
\begin{align}
\label{eq1-pf-prop-ch3-approx-harmonic-tot-inj}
\begin{split}
& \big( F_{Z_R} - G_{Z_R} \big) (\omega_1,\omega_2,\hat{\omega}) \big|_{Y_{[-R,R]}} \\
& = \sum_{j=1}^2 \Big( \frac{\partial}{\partial u}\chi_{j,R} \Big)
\Big( du \wedge \mathscr{R}_{d^F, j}(\omega_j,\hat{\omega})
+ \cntrtu \mathscr{R}_{d^{F,*}, j}(\omega_j,\hat{\omega}) \Big) .
\end{split}
\end{align}
By Proposition \ref{prop-ch2-R},
\eqref{eq-ch1-def-D}
and \eqref{eq1-pf-prop-ch3-approx-harmonic-tot-inj},
for $m\in\N$,
we have
\begin{align}
\label{eq2-pf-prop-ch3-approx-harmonic-tot-inj}
\begin{split}
& D^{F,2m}_{Z_R} \big( F_{Z_R} - G_{Z_R} \big) (\omega_1,\omega_2,\hat{\omega}) \big|_{Y_{[-R,R]}} \\
& =  (-1)^m \sum_{j=1}^2 \Big(\frac{\partial^{2m+1}}{\partial u^{2m+1}}\chi_{j,R} \Big)
\Big( du \wedge \mathscr{R}_{d^F, j}(\omega_j,\hat{\omega})
+ \cntrtu \mathscr{R}_{d^{F,*}, j}(\omega_j,\hat{\omega}) \Big).
\end{split}
\end{align}

By \eqref{eq-ch3-def-chi-2},
$\frac{\partial^{2m+1}}{\partial u^{2m+1}}\chi_{1,R}$
(resp. $\frac{\partial^{2m+1}}{\partial u^{2m+1}}\chi_{2,R}$)
is $\mathscr{O}\big(R^{-2m-1}\big)$
and with support in $[-3R/4, -R/4]$
(resp. $[R/4, 3R/4]$).
Then,
by \eqref{eq2-pf-prop-ch3-approx-harmonic-tot-inj},
we have
\begin{align}
\label{eq3-pf-prop-ch3-approx-harmonic-tot-inj}
\begin{split}
& \big\lVert D^{F,2m}_{Z_R}\big(F_{Z_R}-G_{Z_R}\big)(\omega_1,\omega_2,\hat{\omega})\big\rVert^2_{Z_R} \\
& = \mathscr{O}\big(R^{-4m-2}\big) \Big(
\big\lVert \mathscr{R}_{d^F,1}(\omega_1,\hat{\omega}) \big\rVert^2_{Y_{[-3R/4,-R/4]}}
+ \big\lVert \mathscr{R}_{d^F,2}(\omega_2,\hat{\omega}) \big\rVert^2_{Y_{[R/4,3R/4]}} \\
& \hspace{25mm} + \big\lVert \mathscr{R}_{d^{F,*},1}(\omega_1,\hat{\omega}) \big\rVert^2_{Y_{[-3R/4,-R/4]}}
+ \big\lVert \mathscr{R}_{d^{F,*},2}(\omega_2,\hat{\omega}) \big\rVert^2_{Y_{[R/4,3R/4]}}
\Big).
\end{split}
\end{align}
Set $a=\delta_Y/5$.
Proceeding in the same way as in
\eqref{eq3-pf-prop-ch3-approx-f-g-total} and
applying \eqref{eq-ch3-YR}, \eqref{eq5-pf-prop-ch3-approx-f-g-total},
we see that
the norms on the right hand side of \eqref{eq3-pf-prop-ch3-approx-harmonic-tot-inj}
are
$\mathscr{O}\big(e^{-aR}\big)
\big( \big\lVert\omega_1\big\rVert_{Z_{1,R}} + \big\lVert\omega_2\big\rVert_{Z_{2,R}} \big)$.
Combining with \eqref{eq-ch3-def-metric-hh},
we get
\begin{equation}
\label{eq4-pf-prop-ch3-approx-harmonic-tot-inj}
\big\lVert D^{F,2m}_{Z_R}\big(F_{Z_R}-G_{Z_R}\big)(\omega_1,\omega_2,\hat{\omega})\big\rVert_{Z_R}
= \mathscr{O}\big(e^{-aR}\big)
\big\lVert (\omega_1,\omega_2,\hat{\omega}) \big\rVert_{\hh(Z_{12,\infty},F),R}.
\end{equation}

We denote
$\mu_0 = \big(\mathscr{F}_{Z_R}-\mathscr{G}_{Z_R}\big)(\omega_1,\omega_2,\hat{\omega})$,
$\mu_1 = \big(F_{Z_R}-\mathscr{F}_{Z_R}\big)(\omega_1,\omega_2,\hat{\omega})$
and $\mu_2 = \big(G_{Z_R}-\mathscr{G}_{Z_R}\big)(\omega_1,\omega_2,\hat{\omega})$.
Then we have
\begin{equation}
\label{eq5-pf-prop-ch3-approx-harmonic-tot-inj}
\big(F_{Z_R}-G_{Z_R}\big)(\omega_1,\omega_2,\hat{\omega}) = \mu_0 + \mu_1 - \mu_2.
\end{equation}
By Theorem \ref{thm-ch1-Hodge} for the manifold without boundary $Z_R$,
\eqref{eq-ch3-FG-close} and \eqref{eq-ch3-def-FG-harmonic},
we have
\begin{equation}
\label{eq6-pf-prop-ch3-approx-harmonic-tot-inj}
\mu_0 \in \Ker\big(D^{F,2}_{Z_R}\big), \quad
\mu_1 \in \im\big(d^F\big), \quad
\mu_2 \in \im\big(d^{F,*}\big).
\end{equation}
Since $D^{F,2m}_{Z_R}$ commutes with $d^F$ and $d^{F,*}$
(see \eqref{eq-ch1-def-D2}),
\eqref{eq6-pf-prop-ch3-approx-harmonic-tot-inj} holds with
$\mu_j$ ($j=0,1,2$) replaced by $D^{F,2m}_{Z_R}\mu_j$.
As a consequence,
$\big(D^{F,2m}_{Z_R} \mu_j\big)_{j=0,1,2}$ are mutually orthogonal.
Then,
by \eqref{eq4-pf-prop-ch3-approx-harmonic-tot-inj}
and \eqref{eq5-pf-prop-ch3-approx-harmonic-tot-inj},
we have
\begin{align}
\label{eq7-pf-prop-ch3-approx-harmonic-tot-inj}
\begin{split}
& \big\lVert D^{F,2m}_{Z_R}(F_{Z_R}-\mathscr{F}_{Z_R})(\omega_1,\omega_2,\hat{\omega}) \big\rVert_{Z_R}
= \big\lVert D^{F,2m}_{Z_R} \mu_1 \big\rVert_{Z_R} \\
& \leqslant \big\lVert D^{F,2m}_{Z_R}(F_{Z_R}-G_{Z_R})(\omega_1,\omega_2,\hat{\omega}) \big\rVert_{Z_R} \\
& = \mathscr{O}\big(e^{-aR}\big)
\big\lVert (\omega_1,\omega_2,\hat{\omega})\big\rVert_{\hh(Z_{12,\infty},F),R}.
\end{split}
\end{align}
By Proposition \ref{prop-ch3-sobolev}
and \eqref{eq7-pf-prop-ch3-approx-harmonic-tot-inj},
we obtain \eqref{eq-prop-ch3-approx-harmonic-tot-inj}.
From \eqref{eq-prop-ch3-approx-f-g-total}
and \eqref{eq-prop-ch3-approx-harmonic-tot-inj},
we know that the map
$\mathscr{F}_{Z_R} : \hh(Z_{12,\infty},F) \rightarrow \Ker\big(D^{F,2}_{Z_R}\big)$
is injective.
The proof of Proposition \ref{prop-ch3-approx-harmonic-tot-inj} is completed.
\end{proof}

\begin{rem}
In the proof of Proposition \ref{prop-ch3-approx-harmonic-tot-inj},
Hodge decomposition is used in an essential way.
Hence our argument is not valid for general Dirac operators.
\end{rem}

\begin{thm}
\label{thm-ch3-approx-harmonic-tot}
For $R>2$ large enough,
the map $\mathscr{F}_{Z_R}$ is bijective.
\end{thm}
\begin{proof}
By Proposition \ref{prop-ch3-approx-harmonic-tot-inj},
$\mathscr{F}_{Z_R}$ is injective.
It remains to show that $\mathscr{F}_{Z_R}$ is surjective.

Suppose that,
on the contrary,
there exist $R_i\rightarrow+\infty$ and
$\omega_i\in\Omega^\bullet(Z_{R_i},F)$
such that
\begin{equation}
\label{eq1-pf-thm-ch3-approx-harmonic-tot}
D^F_{Z_{R_i}}\omega_i = 0, \quad
\omega_i\neq 0, \quad
\omega_i\perp\im(\mathscr{F}_{Z_{R_i}}).
\end{equation}
By Lemma \ref{lem-ch2-zm-nz},
$\omega_i\big|_{Z_{1,0} \cup Z_{2,0}}\neq 0$.
By multiplying suitable constants,
we may assume that
\begin{equation}
\label{eq2-pf-thm-ch3-approx-harmonic-tot}
\big\lVert \omega_i \big\rVert^2_{Z_{1,0}} +
\big\lVert \omega_i \big\rVert^2_{Z_{2,0}} = 1 .
\end{equation}

By Lemma \ref{lem-ch2-zm-nz},
\eqref{eq5-pf-prop-ch3-approx-f-g-total}
and \eqref{eq2-pf-thm-ch3-approx-harmonic-tot},
there exists $a>0$ such that
for $T\in[0,R_i]$,
\begin{equation}
\label{eq3-pf-thm-ch3-approx-harmonic-tot}
\big\lVert \omega_i \big\rVert_{Z_{1,T}}^2 \leqslant 1 + aT.
\end{equation}
As a consequence,
for $T$ fixed,
the series $\big(\omega_i\big|_{Z_{1,T}}\big)_i$ is $L^2$-bounded.
Moreover,
by \eqref{eq1-pf-thm-ch3-approx-harmonic-tot},
the series $\big(\omega_i\big|_{Z_{1,T}}\big)_i$ is $H^2$-bounded.
By Rellich's lemma,
we may suppose that
$\big(\omega_i\big|_{Z_{1,T}}\big)_i$ $H^1$-converges.
By a diagonal argument (involving $i$ and $T$),
we may supppose that
for any $T>0$,
$\big(\omega_i\big|_{Z_{1,T}}\big)_i$ $H^1$-converges.
Then there exists $\omega_{1,\infty}$
a current on $Z_{1,\infty}$ with values in $F$
such that for any $T\in\N$,
$\big(\omega_i\big|_{Z_{1,T}}\big)_i$ $H^1$-converges to $\omega_{1,\infty}\big|_{Z_{1,T}}$.
Now,
taking $i\rightarrow+\infty$
in \eqref{eq1-pf-thm-ch3-approx-harmonic-tot},
we get
\begin{equation}
\label{eq4-pf-thm-ch3-approx-harmonic-tot}
D^F_{Z_{1,\infty}}\omega_{1,\infty} = 0.
\end{equation}
Thus $\omega_{1,\infty}\in\Omega^\bullet(Z_{1,\infty},F)$.
Taking $i\rightarrow+\infty$
in \eqref{eq3-pf-thm-ch3-approx-harmonic-tot},
we get
\begin{equation}
\label{eq5-pf-thm-ch3-approx-harmonic-tot}
\big\lVert \omega_{1,\infty} \big\rVert_{Z_{1,T}}^2 \leqslant 1 + aT.
\end{equation}
By \eqref{eq-ch2-def-zm-nz-1}
and \eqref{eq5-pf-thm-ch3-approx-harmonic-tot},
we have
\begin{equation}
\label{eq6-pf-thm-ch3-approx-harmonic-tot}
\omega_{1,\infty}^+=0\in\Omega^\bullet(Y_{[0,+\infty)},F).
\end{equation}
By \eqref{eq-ch2-def-hh},
\eqref{eq4-pf-thm-ch3-approx-harmonic-tot}
and \eqref{eq6-pf-thm-ch3-approx-harmonic-tot},
$\omega_{1,\infty}$ is an extended $L^2$-solution of $D^F_{Z_{1,\infty}}$, i.e.,
\begin{equation}
\label{eq7-pf-thm-ch3-approx-harmonic-tot}
(\omega_{1,\infty},\hat{\omega}_1)\in\hh(Z_{1,\infty},F),
\end{equation}
where $\hat{\omega}_1 = \omega_{1,\infty}^\mathrm{zm}\big|_{\partial Z_{1,0}}\in H^\bullet(Y,F[du])$.
Moreover,
since $\omega_i\big|_{Z_{1,R_i}}$ $H^1_\mathrm{loc}$-converges to $\omega_{1,\infty}$,
$\omega_i\big|_{\partial Z_{1,0}}$ $L^2$-converges to $\hat{\omega}_1$.

The same argument on $\omega_i\big|_{Z_{2,T}}$
yields
\begin{equation}
\label{eq8-pf-thm-ch3-approx-harmonic-tot}
(\omega_{2,\infty},\hat{\omega}_2)\in\hh(Z_{2,\infty},F)
\end{equation}
satisfying parallel properties.
In particular,
$\omega_i\big|_{\partial Z_{2,0}}$ $L^2$-converges to $\hat{\omega}_2$.
By \eqref{eq-ch2-def-zm-nz-1} for $\lambda=0$, we have
\begin{equation}
\label{eq9-pf-thm-ch3-approx-harmonic-tot}
\omega_i^\mathrm{zm} \big|_{\partial Z_{2,0}} =
\omega_i^\mathrm{zm} \big|_{\partial Z_{1,0}}.
\end{equation}
Taking $i\rightarrow+\infty$ in \eqref{eq9-pf-thm-ch3-approx-harmonic-tot},
we get $\hat{\omega}_1 = \hat{\omega}_2$.
Then,
by \eqref{eq-ch3-def-hh-12},
\eqref{eq7-pf-thm-ch3-approx-harmonic-tot}
and \eqref{eq8-pf-thm-ch3-approx-harmonic-tot},
we get
\begin{equation}
\label{eq10-pf-prop-ch3-approx-harmonic-tot-surj}
(\omega_{1,\infty}, \omega_{2,\infty},\hat{\omega}_1)\in\hh(Z_{12,\infty},F).
\end{equation}
Set
$\mu_i = F_{Z_{R_i}}(\omega_{1,\infty}, \omega_{2,\infty},\hat{\omega}_1)$ and
$\widetilde{\mu}_i = \mathscr{F}_{Z_{R_i}}(\omega_{1,\infty}, \omega_{2,\infty},\hat{\omega}_1)$.

\noindent \textbf{Case 1. $\hat{\omega}_1 = 0$.}
We have
\begin{equation}
\label{eq11-pf-thm-ch3-approx-harmonic-tot}
\big\langle\omega_i,\omega_{1,\infty}\big\rangle_{Z_{1,R_i}} =
\big\langle\omega_i,\omega_{1,\infty}\big\rangle_{Z_{1,0}} +
\big\langle\omega_i^\mathrm{nz},\omega_{1,\infty}^\mathrm{nz}\big\rangle_{Y_{[0,R_i]}}.
\end{equation}
We know that $\omega_i\big|_{Z_{1,R_i}}$
$L^2_\mathrm{loc}$-converges to $\omega_{1,\infty}$
and that $Z_{1,0}$ is compact.
Then we have
\begin{equation}
\label{eq12-pf-thm-ch3-approx-harmonic-tot}
\big\langle\omega_i,\omega_{1,\infty}\big\rangle_{Z_{1,0}} \rightarrow
\big\lVert\omega_{1,\infty}\big\rVert_{Z_{1,0}}^2,
\quad \text{as } i\rightarrow+\infty.
\end{equation}
We know that $\omega_i^\mathrm{nz}\big|_{Z_{1,R_i}}$
$L^2_\mathrm{loc}$-converges to $\omega_{1,\infty}^\mathrm{nz}$.
Then,
by the dominated convergence theorem and \eqref{eq-ch2-nz-decreasing},
we have
\begin{equation}
\label{eq13-pf-thm-ch3-approx-harmonic-tot}
\big\langle\omega_i^\mathrm{nz},\omega_{1,\infty}^\mathrm{nz}\big\rangle_{Y_{[0,R_i]}} \rightarrow
\big\lVert\omega^\mathrm{nz}_{1,\infty}\big\rVert_{Y_{[0,+\infty)}}^2,
\quad \text{as } i\rightarrow+\infty.
\end{equation}
By \eqref{eq6-pf-prop-ch3-approx-f-g-total},
\eqref{eq3-pf-thm-ch3-approx-harmonic-tot} and
\eqref{eq11-pf-thm-ch3-approx-harmonic-tot}-\eqref{eq13-pf-thm-ch3-approx-harmonic-tot},
we have
\begin{equation}
\label{eq14-pf-thm-ch3-approx-harmonic-tot}
\big\langle\omega_i,\mu_i\big\rangle_{Z_{1,R_i}} \rightarrow
\big\lVert\omega_{1,\infty}\big\rVert_{Z_{1,\infty}}^2,
\quad \text{as } i\rightarrow+\infty.
\end{equation}
The same argument also yields
\begin{equation}
\label{eq15-pf-thm-ch3-approx-harmonic-tot}
\big\langle\omega_i,\mu_i\big\rangle_{Z_{2,R_i}} \rightarrow
\big\lVert\omega_{2,\infty}\big\rVert_{Z_{2,\infty}}^2,
\quad \text{as } i\rightarrow+\infty.
\end{equation}
By Proposition \ref{prop-ch3-approx-harmonic-tot-inj},
\eqref{eq14-pf-thm-ch3-approx-harmonic-tot} and
\eqref{eq15-pf-thm-ch3-approx-harmonic-tot},
we get
\begin{equation}
\label{eq16-pf-thm-ch3-approx-harmonic-tot}
\big\langle\omega_i,\widetilde{\mu}_i\big\rangle_{Z_{R_i}} \rightarrow
\big\lVert\omega_{1,\infty}\big\rVert_{Z_{1,\infty}}^2 +
\big\lVert\omega_{2,\infty}\big\rVert_{Z_{2,\infty}}^2 \neq 0,
\quad \text{as } i\rightarrow+\infty.
\end{equation}
However,
by \eqref{eq1-pf-thm-ch3-approx-harmonic-tot},
we have
$\big\langle\omega_i,\widetilde{\mu}_i\big\rangle_{Z_{R_i}}=0$.
Here we get a contradiction.

\noindent \textbf{Case 2. $\hat{\omega}_1 \neq 0$.}
We have
\begin{equation}
\label{eq21-pf-thm-ch3-approx-harmonic-tot}
\big\langle\omega_i,\omega_{1,\infty}\big\rangle_{Z_{1,R_i}} =
\big\langle\omega_i,\omega_{1,\infty}\big\rangle_{Z_{1,0}} +
\big\langle\omega_i^\mathrm{nz},\omega_{1,\infty}^\mathrm{nz}\big\rangle_{Y_{[0,R_i]}} +
\big\langle\omega_i^\mathrm{zm},\omega_{1,\infty}^\mathrm{zm}\big\rangle_{Y_{[0,R_i]}}.
\end{equation}
Here \eqref{eq12-pf-thm-ch3-approx-harmonic-tot}
and \eqref{eq13-pf-thm-ch3-approx-harmonic-tot} still hold.
We just need to consider
$\big\langle\omega_i^\mathrm{zm},\omega_{1,\infty}^\mathrm{zm}\big\rangle_{Y_{[0,R_i]}}$.
We know that $\omega_i^\mathrm{zm}\big|_{\partial Z_{1,0}}$
$L^2$-converges to $\hat{\omega}_1 = \omega_{1,\infty}^\mathrm{zm}\big|_{\partial Z_{1,0}}$.
Then,
by \eqref{eq-ch2-def-zm-nz-1},
we have
\begin{equation}
\label{eq22-pf-thm-ch3-approx-harmonic-tot}
R_i^{-1} \big\langle\omega_i^\mathrm{zm},\omega_{1,\infty}^\mathrm{zm}\big\rangle_{Y_{[0,R_i]}} \rightarrow
\big\lVert\hat{\omega}_1\big\rVert_Y^2,
\quad \text{as } i\rightarrow+\infty.
\end{equation}
Now,
proceeding in the same way as in Case 1,
we get
\begin{equation}
R_i^{-1} \big\langle\omega_i,\widetilde{\mu}_i\big\rangle_{Z_{R_i}} \rightarrow
2\big\lVert\hat{\omega}_1\big\rVert_Y^2 \neq 0,
\quad \text{as } i\rightarrow+\infty,
\end{equation}
which leads to the same contradiction.

The proof of Theorem \ref{thm-ch3-approx-harmonic-tot} is completed.
\end{proof}


\subsection{Approximating the small eigenvalues}
\label{ch3-4}

\

\begin{thm}
\label{thm-ch3-gap}
There exists $\alpha>0$ such that
\begin{equation}
\label{eq-thm-ch3-gap}
\Sp\big(D^F_{Z_R}\big) \subseteq
(-\infty,-\alpha R^{-1}) \cup \{0\} \cup (\alpha R^{-1},+\infty) .
\end{equation}
\end{thm}
\begin{proof}
The proof consists of several steps.

\noindent\textbf{Step 1. }
We prove a result in linear algebra.

Let
\begin{equation}
(W^\bullet, \partial) : 0 \rightarrow W^0 \rightarrow \cdots \rightarrow W^n \rightarrow 0
\end{equation}
be a chain complex of finite dimensional complex vector space.
Let $h^{W^\bullet}_t = \bigoplus_k h^{W^k}_t$ with $t\in\R$
be a smooth family of Hermitian metrics on $W^\bullet$.
We denote
\begin{equation}
Q_t = \big(h^{W^\bullet}_t\big)^{-1}\frac{\partial}{\partial t}h^{W^\bullet}_t
\in \mathrm{End}\big(W^\bullet\big).
\end{equation}
Let $\partial^*_t$ be the adjoint of $\partial$ with respect to $h^{W^\bullet}_t$.
Set $D_t = \partial + \partial^*_t$.
By the perturbation theory \cite[\textsection 2.6, Theorem 6.1]{Kato95}
(cf. Theorem \ref{lab-thm-holomophic-eigenvalue-c-matrix} in the appendix),
there exist smooth functions $\lambda_1(t),\cdots,\lambda_m(t)$ such that
\begin{equation}
\big\{\lambda_1(t),\cdots,\lambda_m(t)\big\} = \Sp\big(D_t\big).
\end{equation}
For $A\in\mathrm{End}\big(W^\bullet\big)$,
we denote by $\big\lVert A \big\rVert_t$
the operator norm of $A$ with respect to $h^{W^\bullet}_t$.
We will show that
\begin{equation}
\left|\frac{\partial}{\partial t}\lambda_k(t)\right|
\leqslant \big\lVert Q_t\big\rVert_t \big|\lambda_k(t)\big|.
\end{equation}

For ease of notation,
we will write $\lambda(t)$ rather than $\lambda_k(t)$.
Let $P(t): W^\bullet \rightarrow W^\bullet$
be the projection corresponding to $\lambda(t)$
in Theorem \ref{lab-thm-holomophic-eigenvalue-c-matrix}.
We remark that $t\mapsto P(t)$ is smooth.
Take $w\in W^\bullet$ such that $P(0)w\neq 0$.
For $t$ close to $0$,
take
\begin{equation}
w(t) =  \Big(h_t^{W^\bullet}\big(P(t)w,P(t)w\big)\Big)^{-1/2} P(t)w \in W^\bullet.
\end{equation}
Then $t\mapsto w(t)$ is smooth.
Moreover,
we have
\begin{equation}
\label{eq1-pf-prop-variation}
D_t w(t) = \lambda(t)w(t),
\end{equation}
and
\begin{equation}
\label{eq2-pf-prop-variation}
h^{W^\bullet}_t\big(w(t),w(t)\big) = 1.
\end{equation}
As a consequence, we have
\begin{equation}
\label{eq3-pf-prop-variation}
\mathrm{Re} \, h^{W^\bullet}_t\Big(w(t),\frac{\partial}{\partial t}w(t)\Big) = 0.
\end{equation}

Taking the derivative of \eqref{eq1-pf-prop-variation},
we get
\begin{equation}
\label{eq4-pf-prop-variation}
\Big(\frac{\partial}{\partial t}D_t\Big) w(t) +
D_t \Big(\frac{\partial}{\partial t}w_t\Big) =
\Big(\frac{\partial}{\partial t}\lambda(t)\Big)w(t) +
\lambda(t)\Big(\frac{\partial}{\partial t}w(t)\Big).
\end{equation}
By \eqref{eq1-pf-prop-variation}-\eqref{eq4-pf-prop-variation},
we get
\begin{equation}
\label{eq5-pf-prop-variation}
\frac{\partial}{\partial t}\lambda(t) = \mathrm{Re} \,
h^{W^\bullet}_t\left(w(t),\Big(\frac{\partial}{\partial t}D_t\Big)w(t)\right).
\end{equation}

Let $\overline{W^\bullet}^*$ be the anti-dual of $W^\bullet$.
We may view $h^{W^\bullet}_t$ as a map $h^{W^\bullet}_t: W^\bullet \rightarrow \overline{W^\bullet}^*$,
defined as follows:
$h^{W^\bullet}_t(w,v) = h^{W^\bullet}_t(w)(v)$
for $w,v\in W^\bullet$.
Let $\overline{\partial}: \overline{W^\bullet}^*\rightarrow\overline{W^\bullet}^*$,
determined as follows:
$(\overline{\partial}\tau)(w) = \tau(\partial w)$
for $w\in W^\bullet$, $\tau\in\overline{W^\bullet}^*$.
Then
\begin{equation}
\label{eq-dstar}
\partial^*_t = \big(h^{W^\bullet}_t\big)^{-1} \overline{\partial} h^{W^\bullet}_t.
\end{equation}
Taking the derivative of \eqref{eq-dstar},
we get
\begin{equation}
\label{eq6-pf-prop-variation}
\frac{\partial}{\partial t} D_t =
\frac{\partial}{\partial t} \partial^*_t =
\partial^*_t Q_t
- Q_t \partial^*_t.
\end{equation}

By \eqref{eq5-pf-prop-variation}
and \eqref{eq6-pf-prop-variation},
we get
\begin{align}
\label{eq7-pf-prop-variation}
\begin{split}
\frac{\partial}{\partial t}\lambda(t)
& = \mathrm{Re} \, h^{W^\bullet}_t\Big(w(t),\partial^*_tQ_tw(t)\Big)
- \mathrm{Re} \, h^{W^\bullet}_t\Big(w(t),Q_t\partial^*_tw(t)\Big) \\
& = \mathrm{Re} \, h^{W^\bullet}_t\Big(\partial w(t),Q_tw(t)\Big)
- \mathrm{Re} \, h^{W^\bullet}_t\Big(Q_tw(t),\partial^*_tw(t)\Big) \\
& = \mathrm{Re} \, h^{W^\bullet}_t\Big(\big(\partial-\partial^*_t\big)w(t),Q_tw(t)\Big).
\end{split}
\end{align}
Using \eqref{eq7-pf-prop-variation}
and the fact that $\big(\partial-\partial^*_t\big)^2 = - \big(\partial+\partial^*_t\big)^2$,
we get
\begin{align}
\begin{split}
\left| \frac{\partial}{\partial t}\lambda(t) \right|^2
& \leqslant
h_t^{W^\bullet}\Big(\big(\partial-\partial^*_t\big)w(t),\big(\partial-\partial^*_t\big)w(t)\Big)
h_t^{W^\bullet}\Big(Q_tw(t),Q_tw(t)\Big) \\
& =
h_t^{W^\bullet}\Big(\big(\partial+\partial^*_t\big)w(t),\big(\partial+\partial^*_t\big)w(t)\Big)
h_t^{W^\bullet}\Big(Q_tw(t),Q_tw(t)\Big) \\
& \leqslant \big|\lambda(t)\big|^2\big\lVert Q_t \big\rVert_t^2.
\end{split}
\end{align}

The result proved in Step 1 also holds true for
the de Rham complex $\big(\Omega^\bullet(Z,F),d^F\big)$
equipped with $L^2$-metrics.

\noindent\textbf{Step 2. }
We prove \eqref{eq-thm-ch3-gap}.

For $\epsilon>0$ and $t\geqslant 0$,
we perturb $g^{TZ_R}$ on the cylindrical part $Y_{[-R,R]}$ as follows,
\begin{equation}
\label{eq1-pf-cor-variation}
\Big(g^{TZ_R}_{\epsilon,t}\Big)_{(u,y)} = \Big(g^{TZ_R}\Big)_{(u,y)}
+ \psi\Big(\frac{R-|u|}{\epsilon}\Big)\frac{2t \, du^2}{R},
\end{equation}
where $\psi$ is constructed in \eqref{eq-ch3-def-chi-1}.
Set
\begin{equation}
\label{eq2-pf-cor-variation}
R_{\epsilon,t} = \int_0^R \sqrt{1 + \psi\Big(\frac{R-u}{\epsilon}\Big)\frac{2t}{R}} \; du.
\end{equation}
Then $Z_R$ (viewed as a differential manifold)
equipped with the Riemannian metric $g^{TZ_R}_{\epsilon,t}$
is isometric to $Z_{R_{\epsilon,t}}$.
A direct calculation yields
\begin{equation}
\label{eq3-pf-cor-variation}
\lim_{\epsilon\rightarrow 0} \frac{\partial R_{\epsilon,t}}{\partial t}\Big|_{t=0} = 1.
\end{equation}

Let $h^{\Omega^\bullet(Z_R,F)}_{\epsilon,t}$ be the $L^2$-metric on $\Omega^\bullet(Z_R,F)$
induced by $g^{TZ_R}_{\epsilon,t}$ and $h^F$.
By \cite[Proposition 4.15]{BZ92},
we have
\begin{align}
\label{eq6-pf-cor-variation}
\begin{split}
& \left(\Big(h^{\Omega^\bullet(Z_R,F)}_{\epsilon,t}\Big)^{-1}
\frac{\partial}{\partial t}h^{\Omega^\bullet(Z_R,F)}_{\epsilon,t}\Big|_{t=0}\right)_{(u,y)} \\
& = - \psi\Big(\frac{1-|u|}{\epsilon}\Big) R^{-1} \cliffu\hatcliffu
 \in \mathrm{End}\Big(\Lambda^\bullet(T^*Z_R) \otimes F \Big)_{(u,y)}.
\end{split}
\end{align}
As a consequence,
\begin{equation}
\label{eq7-pf-cor-variation}
\left\lVert \Big(h^{\Omega^\bullet(Z_R,F)}_{\epsilon,t}\Big)^{-1}
\frac{\partial}{\partial t}h^{\Omega^\bullet(Z_R,F)}_{\epsilon,t}\Big|_{t=0}\right\rVert
\leqslant R^{-1}.
\end{equation}

We denote
\begin{equation}
\label{eqa1-pf-cor-variation}
\delta_R =  \min \Big\{ |\lambda| \;:\; 0\neq\lambda\in\Sp\big(D^F_{Z_R}\big) \Big\}.
\end{equation}
For each $R_0>0$,
we take $0<b<\delta_{R_0}<a$ such that
\begin{equation}
\label{eqa2-pf-cor-variation}
\big\{-a,a\big\} \cap \Sp\big(D^F_{Z_{R_0}}\big) = \emptyset, \quad
[-b,b] \cap \Sp\big(D^F_{Z_{R_0}}\big) \subseteq \{0\}.
\end{equation}
Since $\dim\Ker\big(D^F_{Z_R}\big)$ is constant,
\eqref{eqa2-pf-cor-variation} holds with $R_0$ replaced by $R$ close enough to $R_0$.
Applying Theorem \ref{lab-thm-holomophic-eigenvalue-c-matrix}
to the restriction of $D^F_{Z_R}$ to its eigenspace
associated with eigenvalues in $[-a,-b]\cup[b,a]$,
we get
\begin{equation}
\big([-a,-b]\cup[b,a]\big) \cap \Sp\big(D^F_{Z_R}\big) =
\big\{\lambda_1(R),\cdots,\lambda_m(R)\big\},
\end{equation}
where $\lambda_1(R),\cdots,\lambda_m(R)$ depends smoothly on $R$.
By Step 1 and \eqref{eq7-pf-cor-variation},
we have
\begin{equation}
\label{eq8-pf-cor-variation}
\frac{\partial}{\partial t} \big|\lambda_k(R_{\epsilon,t})\big| \Big|_{t=0}
\geqslant - R^{-1} \big|\lambda_k(R)\big|, \quad
\text{for } k=1,\cdots,m.
\end{equation}
Taking $\epsilon\rightarrow 0$ in \eqref{eq8-pf-cor-variation}
and applying \eqref{eq3-pf-cor-variation},
we get
\begin{equation}
\label{eq9-pf-cor-variation}
\frac{\partial}{\partial R} \big|\lambda_k(R)\big|
\geqslant -R^{-1} \big|\lambda_k(R)\big|.
\end{equation}
As a consequence,
the function $R\mapsto R\big|\lambda_k(R)\big|$ is non decreasing.
Then so is the function $R\mapsto R\delta_R$.
This completes the proof of Theorem \ref{thm-ch3-gap}.
\end{proof}

For $j=1,2$,
let $D^F_{Z_{j,\infty},\mathrm{pp}}$ be the restriction of $D^F_{Z_{j,\infty}}$
to the eigenspace associated with the purely point spectrum (cf. \textsection \ref{ch2-2}).
We fix $\delta_{Z_j}>0$ such that
\begin{equation}
\Sp\big(D^F_{Z_{j,\infty},\mathrm{pp}}\big) \cap [-\delta_{Z_j},\delta_{Z_j}] \subseteq \{0\}.
\end{equation}
For $j=1,2$,
we fix $\delta_{C_j}>0$ such that
for $0<a\leqslant \delta_{C_j}$,
the Fourier expansion
\begin{equation}
\label{eq-ch3-def-deltaC}
C_j(\lambda)\big|_{\lambda\in[-a,a]}
= \sum_{k\in\Z} \exp\big( i \pi a^{-1} k\lambda \big) C_{j,a,k}
\end{equation}
satisfies
\begin{equation}
\label{eq-ch3-def-deltaC-2}
\big\lVert C_{j,a,0} - \mathrm{Id} \big\rVert \leqslant \frac{1}{3},\hspace{5mm}
\big\lVert C_{j,a,k} \big\rVert \leqslant \frac{1}{3},\hspace{5mm}\text{for } k \neq 0,
\end{equation}
where $\big\lVert\cdot\big\rVert$ is the operator norm
on $\mathrm{End}\big(\hh(Y,F[du])\big)$ with respect to $\big\lVert\cdot\big\rVert_Y$.
Set
\begin{equation}
\label{eq-ch3-def-delta}
\delta = \frac{1}{2}\mathrm{min}\big\{\delta_Y,\delta_{Z_1},\delta_{Z_2},\delta_{C_1},\delta_{C_2}\big\}.
\end{equation}
Here we indicate the usage of each term in \eqref{eq-ch3-def-delta}.
In the proof of Proposition \ref{prop-ch3-s-eigen-proj-inj},
we only use the fact that $\delta \leqslant \delta_Y/2$.
In the proof of Lemma \ref{lem-ch3-zm},
we use $\delta \leqslant \frac{1}{2}\mathrm{min}\{ \delta_Y, \delta_{Z_1}, \delta_{Z_2} \}$.
In the proof of Lemma \ref{lem-ch3-zm-L2},
we use the whole construction of $\delta$.

For $A\subseteq(-\delta_Y,0)\cup(0,\delta_Y)$,
set
\begin{equation}
\label{eq-ch3-def-Einf-AR}
\mathscr{E}_{A,R}(Z_{12,\infty},F) =
\bigoplus_{\lambda\in A} \mathscr{E}_{\lambda,R}(Z_{12,\infty},F).
\end{equation}
We construct
\begin{equation}
J_{A, Z_R} : \; \mathscr{E}_{A,R}(Z_{12,\infty},F) \rightarrow \Omega^\bullet(Z_R,F)
\end{equation}
as follows:
under the identifications \eqref{eq-ch3-YR} and \eqref{eq-ch3-def-chi-2},
\begin{align}
\label{eq-ch3-def-J}
\begin{split}
J_{A, Z_R}(\omega_1, \omega_1^\mathrm{zm}, \omega_2, \omega_2^\mathrm{zm}) \big|_{Z_{j,0}}
& = \omega_j ,\hspace{5mm} \text{ for } j=1,2 , \\
J_{A, Z_R}(\omega_1, \omega_1^\mathrm{zm}, \omega_2, \omega_2^\mathrm{zm}) \big|_{Y_{[-R,R]}}
& = \chi_{1,R} \; \omega_1 \big|_{Y_{[0,2R]}}
+ \chi_{2,R} \; \omega_2 \big|_{Y_{[-2R,0]}} \\
& \hspace{5mm} + \big(1-\chi_{1,R} - \chi_{2,R}\big) \omega_1^\mathrm{zm} \big|_{Y_{[0,2R]}}.
\end{split}
\end{align}

For $B\subseteq\R$,
we denote by $\mathscr{E}_{B}(Z_R,F) \subseteq \Omega^\bullet(Z_R, F)$
the eigenspaces of $D^F_{Z_R}$ associated with eigenvalues in $B$.
Let $P^B_{Z_R} : \Omega^\bullet(Z_R,F) \rightarrow \mathscr{E}_B(Z_R,F)$
be the orthogonal projection.
Set
\begin{equation}
\label{eq-ch3-def-mathscrJ}
\mathscr{J}_{A, B, Z_R} = P^B_{Z_R} \circ J_{A, Z_R} : \;
\mathscr{E}_{A,R}(Z_{12,\infty},F) \rightarrow \mathscr{E}_B(Z_R,F).
\end{equation}

For $A, B \subseteq \R$ and $\alpha>0$,
we denote $A\subseteq_\alpha B$, if $(x-\alpha,x+\alpha) \subseteq B$ for each $x\in A$.

\begin{prop}
\label{prop-ch3-s-eigen-proj-inj}
There exists $a>0$ such that
for $A\subseteq_{e^{-aR}} B \subseteq (-\delta,0)\cup(0,\delta)$
and $(\omega_1,\omega_1^\mathrm{zm},\omega_2,\omega_2^\mathrm{zm})
\in\mathscr{E}_{A,R}(Z_{12,\infty},F)$,
we have
\begin{equation}
\label{eq-prop-ch3-s-eigen-proj-inj}
\big\lVert \big(\mathscr{J}_{A,B,Z_R}-J_{A,Z_R}\big)
(\omega_1,\omega_1^\mathrm{zm},\omega_2, \omega_2^\mathrm{zm}) \big\rVert_{\continu\!,Z_R}
= \mathscr{O}\big(e^{-aR}\big)
\big( \big\lVert \omega_1 \big\rVert_{Z_{1,0}} +
\big\lVert \omega_2 \big\rVert_{Z_{2,0}} \big).
\end{equation}
As a consequence, $\mathscr{J}_{A, B, Z_R}$ is injective.
\end{prop}
\begin{proof}
For $\lambda_0\in A\subseteq(-\delta,0)\cup(0,\delta)$ and
$(\omega_1, \omega_1^\mathrm{zm}, \omega_2, \omega_2^\mathrm{zm})
\in\mathscr{E}^\bullet_{\lambda_0,R}(Z_{12,\infty},F)$,
similarly to \eqref{eq4-pf-prop-ch3-approx-harmonic-tot-inj},
we could show that
for any $m\in\N$,
there exists $a>0$ such that
\begin{equation}
\label{eq1-pf-prop-ch3-s-eigen-proj-inj}
\big\lVert D^{F,m}_{Z_R}\big(D^F_{Z_R}-\lambda_0\big) J_{A,Z_R}
(\omega_1,\omega_1^\mathrm{zm},\omega_2,\omega_2^\mathrm{zm}) \big\rVert_{Z_R}
= \mathscr{O}\big(e^{-2aR}\big)
\big( \big\lVert \omega_1 \big\rVert_{Z_{1,R}} +
\big\lVert \omega_2 \big\rVert_{Z_{2,R}} \big).
\end{equation}
On the other hand,
for $B\subseteq\R$ satisfying $\big\{\lambda_0\big\}\subseteq_{e^{-aR}}B$,
we have
\begin{align}
\label{eq2-pf-prop-ch3-s-eigen-proj-inj}
\begin{split}
\big\lVert D^{F,m}_{Z_R} \big(J_{A,Z_R}-\mathscr{J}_{A,B,Z_R}\big) (\cdot) \big\rVert_{Z_R}
& =
\big\lVert D^{F,m}_{Z_R} \big(\mathrm{Id}-P^B_{Z_R}\big) J_{A,Z_R}(\cdot) \big\rVert_{Z_R} \\
& \leqslant e^{aR}
\big\lVert D^{F,m}_{Z_R} \big(D^F_{Z_R}-\lambda_0\big) J_{A,Z_R} (\cdot) \big\rVert_{Z_R}.
\end{split}
\end{align}
Now \eqref{eq-prop-ch3-s-eigen-proj-inj} follows from
Proposition \ref{prop-ch3-sobolev},
\eqref{eq1-pf-prop-ch3-s-eigen-proj-inj} and
\eqref{eq2-pf-prop-ch3-s-eigen-proj-inj}.
The proof of Proposition \ref{prop-ch3-s-eigen-proj-inj} is completed.
\end{proof}

\begin{lemma}
\label{lem-ch3-zm}
There exists $a>0$ such that
for $R>2$ large enough and
$\omega\in\Omega^\bullet(Z_R,F)$
an eigensection of $D^F_{Z_R}$ associated with $\lambda\in(-\delta,0)\cup(0,\delta)$,
we have
\begin{equation}
\label{eq-lem-ch3-zm}
\big\lVert \omega^\mathrm{zm} \big\rVert_Y
\geqslant a \big\lVert \omega \big\rVert_{Z_{1,0} \cup Z_{2,0}}.
\end{equation}
In particular,
$\omega^\mathrm{zm}$ is non zero.
\end{lemma}
\begin{proof}
Suppose, on the contrary,
that there exist
$R_i\rightarrow+\infty$,
$\omega_i\in\Omega^\bullet(Z_{R_i},F)$
and $\lambda_i\in(-\delta,0)\cup(0,\delta)$
such that
\begin{equation}
\label{eq1-pf-lem-ch3-zm}
D^F_{Z_{R_i}} \omega_i = \lambda_i \omega_i ,
\end{equation}
and
\begin{equation}
\label{eq2-pf-lem-ch3-zm}
\big\lVert \omega_i \big\rVert_{Z_{1,0}\cup Z_{2,0}} = 1, \hspace{5mm}
\big\lVert \omega_i^\mathrm{zm} \big\rVert_Y \rightarrow 0, \quad
\text{as } i \rightarrow +\infty.
\end{equation}
We may assume that $\lambda_i\rightarrow\lambda_\infty$.

Proceeding in the same way as in the proof of
Theorem \ref{thm-ch3-approx-harmonic-tot},
we may assume that there exist
$\omega_{1,\infty}\in\Omega^\bullet(Z_{1,\infty},F)$ and $\omega_{2,\infty}\in\Omega^\bullet(Z_{2,\infty},F)$
such that for any $T>0$,
$\big(\omega_i\big|_{Z_{j,T}}\big)_i$ $H^1$-converges to
$\omega_{j,\infty}\big|_{Z_{j,T}}$ ($j=1,2$).
Moreover,
\begin{equation}
\label{eq3-pf-lem-ch3-zm}
\omega_{1,\infty}^+ = 0, \quad
\omega_{2,\infty}^- = 0.
\end{equation}
Now, taking $i\rightarrow+\infty$ in
\eqref{eq1-pf-lem-ch3-zm} and
\eqref{eq2-pf-lem-ch3-zm},
we get
\begin{equation}
\label{eq4-pf-lem-ch3-zm}
\big\lVert\omega_{1,\infty}\big\rVert^2_{Z_{1,0}}  +
\big\lVert\omega_{2,\infty}\big\rVert^2_{Z_{2,0}} = 1, \quad
D^F_{Z_{j,\infty}}\omega_{j,\infty} = \lambda_\infty\omega_{j,\infty}, \quad
\omega_{j,\infty}^\mathrm{zm} = 0, \quad \text{for } j=1,2.
\end{equation}
Without loss of generality,
we assume that $\omega_{1,\infty}\neq 0$.
By Lemma \ref{lem-ch2-zm-nz},
\eqref{eq3-pf-lem-ch3-zm} and
\eqref{eq4-pf-lem-ch3-zm},
$\omega_{1,\infty}$ is a $L^2$-eigensection of $D^F_{Z_{1,\infty}}$.
Then we have $\lambda_\infty\in\Sp\big(D^F_{Z_{1,\infty},\mathrm{pp}}\big)$.
Since
$|\lambda_\infty| \leqslant \delta < \delta_{Z_1}$,
which is the lower bound for the non zero elements in $\Sp\big(D^F_{Z_{1,\infty},\mathrm{pp}}\big)$,
we must have $\lambda_\infty=0$.
Thus
\begin{equation}
\omega_{1,\infty}\in\hhl(Z_{1,\infty},F).
\end{equation}

Proceeding in the same way as in the proof of
Theorem \ref{thm-ch3-approx-harmonic-tot},
we could show that
\begin{equation}
\big\langle \omega_i,\mathscr{F}_{Z_{R_i}}(\omega_{1,\infty},0,0) \big\rangle_{Z_{R_i}}
\rightarrow \big\lVert \omega_{1,\infty} \big\rVert^2_{Z_{1,\infty}} \neq 0,
\quad \text{as } i \rightarrow+\infty.
\end{equation}
However,
since $\omega_i$ is an eigensection of $D^F_{Z_{R_i}}$ associated with $\lambda_i\neq 0$
while
$\mathscr{F}_{Z_{R_i}}(\omega_{1,\infty},0,0) \in \Ker\big(D^F_{Z_{R_i}}\big)$,
we should have
$\big\langle \omega_i,\mathscr{F}_{Z_{R_i}}(\omega_{1,\infty},0,0) \big\rangle_{Z_{R_i}} = 0$.
Here we get a contradiction.
The proof of Lemma \ref{lem-ch3-zm} is completed.
\end{proof}

\begin{lemma}
\label{lem-ch3-zm-L2}
There exists $a>0$ such that
for $R>2$ large enough and
$\omega\in\Omega^\bullet(Z_R,F)$
sum of several eigensections of $D^F_{Z_R}$ associated with eigenvalues in $(-\delta,0)\cup(0,\delta)$,
we have
\begin{equation}
\label{eq1-lem-ch3-zm-L2}
\big\lVert \omega^\mathrm{zm} \big\rVert_{Y_{[-R,R]}}
\geqslant a \big\lVert \omega \big\rVert_{Z_{1,0} \cup Z_{2,0}}.
\end{equation}
As a consequence,
\begin{equation}
\label{eq2-lem-ch3-zm-L2}
\big\lVert \omega^\mathrm{zm} \big\rVert_{Y_{[-R,R]}}
\geqslant a \big\lVert \omega \big\rVert_{Z_R}.
\end{equation}
\end{lemma}
\begin{proof}
Suppose, on the contrary,
that there exist
$R_i\rightarrow+\infty$,
and $\omega_i\in\Omega^\bullet(Z_{R_i},F)$
such that
$\omega_i$ is sum of eigensections of $D^F_{Z_{R_i}}$
associated with eigenvalues in $(-\delta,0)\cup(0,\delta)$
and
\begin{equation}
\label{eq3-pf-lem-ch3-zm-L2}
\big\lVert \omega_i \big\rVert_{Z_{1,0}\cup Z_{2,0}} = 1, \hspace{5mm}
\big\lVert \omega_i^\mathrm{zm} \big\rVert_{Y_{[-R_i,R_i]}} \rightarrow 0, \quad
\text{as } i \rightarrow+\infty.
\end{equation}

Proceeding in the same way as in the proof of Lemma \ref{lem-ch3-zm},
we obtain $\omega_{j,\infty}\in\Omega^\bullet(Z_{j,\infty},F)$ ($j=1,2$)
satisfying the same properties except for
the second identity in \eqref{eq4-pf-lem-ch3-zm}.
Instead, we know that
$\omega_{j,\infty}$ only consists of (generalized) eigensections
of $D^F_{Z_{j,\infty}}$ associated with (generalized) eigenvalues in $[-\delta,\delta]$.

Without loss of generality,
we assume that $\omega_{1,\infty}\neq 0$.
Let
\begin{equation}
\label{eq4-pf-lem-ch3-zm-L2}
\omega_{1,\infty} = \omega_{1,\infty}^\mathrm{pp} + \omega_{1,\infty}^\mathrm{ac}
\end{equation}
be the decomposition with respect to \eqref{eq-ch2-pp-ac}.
Proceeding in the same way as in the proof of Lemma \ref{lem-ch3-zm},
we have $\omega_{1,\infty}^\mathrm{pp}\in\hhl(Z_{1,\infty},F)$.
In particular,
the zeromode of $\omega_{1,\infty}^\mathrm{pp}$ vanishes.
On the other hand,
by \eqref{eq4-pf-lem-ch3-zm},
the zeromode of $\omega_{1,\infty}$ vanishes.
Then so does $\omega_{1,\infty}^\mathrm{ac}$.
For $\phi\in\hh(Y,F)$ and $\lambda\in[-\delta,\delta]$,
we denote by $E(\phi,\lambda)$ be the generalized eigensection of $D^F_{Z_{1,\infty}}$
associated with $\lambda$
satisfying Proposition \ref{prop-ch2-g-eigen-Dinf}.
There exists a family
$\big(\phi_\lambda\big)_{\lambda\in[-\delta,\delta]}$
such that
\begin{equation}
\label{eq5-pf-lem-ch3-zm-L2}
\omega_{1,\infty}^\mathrm{ac} =
\int_{-\delta}^\delta E(\phi_\lambda,\lambda) d\lambda.
\end{equation}
Taking the zeromode of $\omega_{1,\infty}^\mathrm{ac}$
and applying Proposition \ref{prop-ch2-g-eigen-Dinf},
we get
\begin{align}
\label{eq6-pf-lem-ch3-zm-L2}
\begin{split}
0 = \big(\omega_{1,\infty}^\mathrm{ac}\big)^{\mathrm{zm},-}
= \big(1-i\cliffu\big) \int_{-\delta}^\delta e^{-i\lambda u_1}  \phi_\lambda d\lambda, \\
0 = \big(\omega_{1,\infty}^\mathrm{ac}\big)^{\mathrm{zm},+}
= \big(1+i\cliffu\big) \int_{-\delta}^\delta e^{i\lambda u_1}  C_1(\lambda)\phi_\lambda d\lambda.
\end{split}
\end{align}
Since $\phi_\lambda\in\hh(Y,F)$
and $C_1(\lambda)\phi_\lambda\in\hh(Y,F)$
(see Proposition \ref{prop-ch2-C-lambda}),
we have
\begin{equation}
\label{eq7-pf-lem-ch3-zm-L2}
\int_{-\delta}^\delta e^{-i\lambda u_1}  \phi_\lambda d\lambda =
\int_{-\delta}^\delta e^{i\lambda u_1}  C_1(\lambda)\phi_\lambda d\lambda = 0.
\end{equation}
We remark that $u_1\geqslant 0$.
Then, by \eqref{eq7-pf-lem-ch3-zm-L2},
the Fourier expansion of
$\big(\phi_\lambda\big)_{\lambda\in[-\delta,\delta]}$
(resp. $\big(C_1(\lambda)\phi_\lambda\big)_{\lambda\in[-\delta,\delta]}$)
only possesses strictly negative (resp. strictly positive) frequency.
Now, applying \eqref{eq-ch3-def-deltaC} and \eqref{eq-ch3-def-deltaC-2},
we see that $\phi_\lambda=0$ for $\lambda\in[-\delta,\delta]$.
Thus $\omega_{1,\infty}^\mathrm{ac}=0$
and $\omega_{1,\infty} = \omega_{1,\infty}^\mathrm{pp}\in\hhl(Z_{1,\infty},F)$.
This leads to the same contradiction as in the proof of Lemma \ref{lem-ch3-zm}.

The inequality \eqref{eq2-lem-ch3-zm-L2}
is a consequence of
Lemma \ref{lem-ch2-zm-nz},
\eqref{eq1-lem-ch3-zm-L2}
and the obvious identity
\begin{equation}
\big\lVert \omega \big\rVert_{Z_R}^2 =
\big\lVert \omega \big\rVert_{Z_{1,0}\cup Z_{2,0}}^2 +
\big\lVert \omega^\mathrm{zm} \big\rVert_{Y_{[-R,R]}}^2 +
\big\lVert \omega^\mathrm{nz} \big\rVert_{Y_{[-R,R]}}^2.
\end{equation}
This completes the proof of Lemma \ref{lem-ch3-zm-L2}.
\end{proof}

\begin{lemma}
\label{lem-ch3-cut}
There exists $a>0$ such that
for $\omega\in\Omega^\bullet(Z_R,F)$
an eigensection of $D^F_{Z_R}$ associated with $\lambda\in(-\delta,0)\cup(0,\delta)$,
we have
\begin{equation}
\label{eq1-lem-ch3-cut}
\Big\lVert C_j(\lambda)\omega^{\mathrm{zm},-}\big|_{\partial Z_{j,0}} - \omega^{\mathrm{zm},+}\big|_{\partial Z_{j,0}} \Big\rVert_Y
= \mathscr{O}\big(e^{-2aR}\big)
\big\lVert \omega \big\rVert_{Z_{1,0}\cup Z_{2,0}},
\quad \text{for } j=1,2.
\end{equation}
In particular,
\begin{equation}
\label{eq2-lem-ch3-cut}
\Big\lVert \big(e^{4i\lambda R}C_{12}(\lambda)-\mathrm{Id}\big)
\omega^{\mathrm{zm},-}\big|_{\partial Z_{1,0}} \Big\rVert_Y
= \mathscr{O}\big(e^{-2aR}\big)
\big\lVert \omega \big\rVert_{Z_{1,0}\cup Z_{2,0}}.
\end{equation}
\end{lemma}
\begin{proof}
We will follow the argument in \cite[$\S$ 8]{Muller94}.
Here we only consider $j=1$.

By \eqref{eq-ch2-eigen-expansion}-\eqref{eq-ch2-decomp-zm-nz},
on $Y_{[-R,R]}\subseteq Z_R$,
there exist $\phi_1,\phi'_1\in\hh(Y,F[du])$ such that
\begin{equation}
\label{eq0-pf-lem-ch3-cut}
\omega\big|_{Y_{[-R,R]}} =
e^{-i\lambda u}\phi_1 + e^{i\lambda u}\phi'_1 + \omega^\mathrm{nz}.
\end{equation}
Under the identification \eqref{eq-ch3-YR},
on $Y_{[0,2R]}$,
we have
\begin{equation}
\label{eq1-pf-lem-ch3-cut}
\omega\big|_{Y_{[0,2R]}} =
e^{-i\lambda u_1} e^{i\lambda R} \phi_1
+ e^{i\lambda u_1} e^{-i\lambda R} \phi'_1 + \omega^\mathrm{nz}.
\end{equation}
Thus
\begin{equation}
\label{eq2-pf-lem-ch3-cut}
\omega^{\mathrm{zm},-}\big|_{\partial Z_{1,0}} = e^{i\lambda R}\phi_1 =: \phi, \quad
\omega^{\mathrm{zm},+}\big|_{\partial Z_{1,0}} = e^{-i\lambda R}\phi'_1 = : \phi'.
\end{equation}
By Proposition \ref{prop-ch2-g-eigen-Dinf},
there exists
$(\widetilde{\omega},{\widetilde{\omega}}^\mathrm{zm})\in\mathscr{E}_\lambda(Z_{1,\infty},F)$
(cf. \eqref{eq-ch3-def-Einf})
such that
\begin{equation}
\label{eq3-pf-lem-ch3-cut}
\widetilde{\omega}\big|_{Y_{[0,R]}} =
e^{-i\lambda u_1}\phi + e^{i\lambda u_1}C_1(\lambda)\phi' + \widetilde{\omega}^\mathrm{nz}.
\end{equation}
Set
$\mu = \omega\big|_{Z_{1,R}} - \tilde{\omega}\big|_{Z_{1,R}}\in\Omega^\bullet(Z_{1,R},F)$.
We have
\begin{equation}
\label{eq4-pf-lem-ch3-cut}
\big\langle D^F_{Z_{1,R}}\mu,\mu \big\rangle_{Z_{1,R}} -
\big\langle \mu,D^F_{Z_{1,R}}\mu \big\rangle_{Z_{1,R}}
= \big\langle \lambda \mu,\mu \big\rangle_{Z_{1,R}} -
\big\langle \mu,\lambda\mu \big\rangle_{Z_{1,R}}
= 0 .
\end{equation}
On the other hand,
note that $\partial Z_{1,R} = Y_R$,
by \eqref{eq-ch2-green} and
\eqref{eq1-pf-lem-ch3-cut}-\eqref{eq3-pf-lem-ch3-cut},
we have
\begin{align}
\label{eq5-pf-lem-ch3-cut}
\begin{split}
& \big\langle D^F_{Z_{1,R}}\mu,\mu \big\rangle_{Z_{1,R}}
- \big\langle \mu,D^F_{Z_{1,R}}\mu \big\rangle_{Z_{1,R}}
= \big\langle \cliffu \mu,\mu \big\rangle_{\partial Z_{1,R}} \\
& = - i\Big\lVert C_1(\lambda)\omega^{\mathrm{zm},-}\big|_{\partial Z_{1,0}}
- \omega^{\mathrm{zm},+}\big|_{\partial Z_{1,0}} \Big\rVert^2_Y
+  \big\langle \cliffu \mu^\mathrm{nz},\mu^\mathrm{nz} \big\rangle_{\partial Z_{1,R}}.
\end{split}
\end{align}
We remark that the last equality in \eqref{eq5-pf-lem-ch3-cut} needs the identities
\begin{equation}
\cliffu \phi'_1 = - i \phi'_1, \quad
\cliffu C(\lambda) \phi'_1 = i C(\lambda) \phi'_1,
\end{equation}
which come from \eqref{eq-ch2-def-zm-nz-1} and \eqref{eq-ch2-zm-norm}.
By \eqref{eq4-pf-lem-ch3-cut}
and \eqref{eq5-pf-lem-ch3-cut},
we get
\begin{equation}
\label{eq6-pf-lem-ch3-cut}
\Big\lVert C_1(\lambda)\omega^{\mathrm{zm},-}|_{\partial Z_{1,0}}
- \omega^{\mathrm{zm},+}|_{\partial Z_{1,0}} \Big\rVert_Y
\leqslant  \big\lVert \mu^\mathrm{nz} \big\rVert_{\partial Z_{1,R}}
\leqslant \big\lVert \omega^\mathrm{nz} \big\rVert_{\partial Z_{1,R}}
+ \big\lVert \widetilde{\omega}^\mathrm{nz} \big\rVert_{\partial Z_{1,R}}.
\end{equation}

By \eqref{eq-ch2-nz-decreasing},
we have
\begin{equation}
\label{eq7-pf-lem-ch3-cut}
\big\lVert \omega^\mathrm{nz} \big\rVert_{\partial Z_{1,R}}
= \mathscr{O}\big(e^{-\delta R}\big)
\big\lVert \omega \big\rVert_{\partial Z_{1,0} \cup \partial Z_{2,0}}, \quad
\big\lVert \widetilde{\omega}^\mathrm{nz} \big\rVert_{\partial Z_{1,R}}
=  \mathscr{O}\big(e^{-\delta R}\big)
\big\lVert \widetilde{\omega} \big\rVert_{\partial Z_{1,0}}.
\end{equation}
Proceeding in the same way as in \eqref{eq5-pf-prop-ch3-approx-f-g-total},
we get
\begin{equation}
\label{eq8-pf-lem-ch3-cut}
\big\lVert \omega \big\rVert_{\partial Z_{1,0} \cup \partial Z_{2,0}}
= \mathscr{O}(1) \big\lVert \omega \big\rVert_{Z_{1,0} \cup Z_{2,0}}, \quad
\big\lVert \widetilde{\omega} \big\rVert_{\partial Z_{1,0}}
= \mathscr{O}(1) \big\lVert \widetilde{\omega} \big\rVert_{Z_{1,0}}.
\end{equation}
By \eqref{eq-rem-ch2-g-eigen},
\eqref{eq2-pf-lem-ch3-cut} and
\eqref{eq3-pf-lem-ch3-cut},
we have
\begin{equation}
\label{eq10-pf-lem-ch3-cut}
\big\lVert \widetilde{\omega} \big\rVert_{Z_{1,0}}
= \mathscr{O}(1) \big\lVert \phi \big\rVert_Y
= \mathscr{O}(1) \big\lVert \omega \big\rVert_{\partial Z_{1,0}}.
\end{equation}
By \eqref{eq6-pf-lem-ch3-cut}-\eqref{eq10-pf-lem-ch3-cut},
we obtain \eqref{eq1-lem-ch3-cut}.

Using \eqref{eq1-lem-ch3-cut}
and proceeding in the same way as in the proof of Lemma \ref{lem-ch3-LambdaR},
we get \eqref{eq2-lem-ch3-cut}.
The proof of Lemma \ref{lem-ch3-cut} is completed.
\end{proof}

\begin{thm}
\label{thm-ch3-s-eigen-proj}
There exists $a>0$ such that
for $A,B\subseteq\R$ satisfying
\begin{equation}
\label{eq-thm-ch3-s-eigen-proj}
\mathrm{Sp}\big(D^F_{Z_R}\big) \cap B
\subseteq_{e^{-aR}} A
\subseteq_{e^{-aR}} B
\subseteq (-\delta,0)\cup(0,\delta),
\end{equation}
the map
$\mathscr{J}_{A, B, Z_R}:
\mathscr{E}_{A,R}(Z_{12,\infty},F) \rightarrow \mathscr{E}_B(Z_R,F)$
is bijective.
\end{thm}
\begin{proof}
The proof consists of several steps.

\noindent\textbf{Step 1. }
We prove the following result.
For $\lambda_0\in(-\delta,\delta)$
and $v\in \hh(Y,F)$,
if
\begin{equation}
\epsilon := \Big\lVert e^{4iR\lambda_0}C_{12}(\lambda_0)v - v \Big\rVert_Y
< \big\lVert v \big\rVert_Y,
\end{equation}
then there exist an orthogonal decomposition
$v = v_1+\cdots+v_m$,
$\lambda_1,\cdots,\lambda_m\in\R$ and
$w_1,\cdots,w_m\in \hh(Y,F)$ satisfying
\begin{align}
\label{eq0-pf-thm-ch3-s-eigen-proj}
\begin{split}
& e^{4iR\lambda_j}C_{12}(\lambda_j)w_j - w_j = 0, \\
& \big|\lambda_j-\lambda_0\big| < \epsilon^{1/2} \big\lVert v \big\rVert_Y^{-1/2}, \quad
\big\lVert v_j - w_j \big\rVert_Y < \epsilon^{1/2} \big\lVert v \big\rVert_Y^{1/2}.
\end{split}
\end{align}

Let $\theta_1(\lambda),\cdots,\theta_m(\lambda)\in\R$ and $P_1(\lambda),\cdots,P_m(\lambda)\in\mathrm{End}\big(\hh(Y,F)\big)$
such that \eqref{eq-thm-app-C} holds with $C$ replaced by $C_{12}$.
Set $v_j = P_j(\lambda_0)v$.
Then $v=v_1+\cdots+v_m$ is an orthogonal decomposition.
Moreover,
we have
\begin{equation}
\label{eq01-pf-thm-ch3-s-eigen-proj}
\epsilon^2= \Big\lVert e^{4iR\lambda_0}C_{12}(\lambda_0)v - v \Big\rVert_Y^2
= \sum_{j=1}^m \Big| e^{4iR\lambda_0+i\theta_j(\lambda_0)} - 1 \Big|^2 \big\lVert v_j \big\rVert_Y^2.
\end{equation}
If $\big\lVert v_j \big\rVert_Y^2 < \epsilon \big\lVert v \big\rVert_Y$,
set $w_j = 0$ and $\lambda_j = \lambda_0$.
Then \eqref{eq0-pf-thm-ch3-s-eigen-proj} holds trivially.
Otherwise,
by \eqref{eq01-pf-thm-ch3-s-eigen-proj}
we have
\begin{equation}
\label{eq02-pf-thm-ch3-s-eigen-proj}
\Big| e^{4iR\lambda_0+i\theta_j(\lambda_0)} - 1 \Big|^2 \leqslant
\epsilon \big\lVert v \big\rVert_Y^{-1}  < 1.
\end{equation}
Using the inequality $\big| e^{ix}-1 \big|^2 > x^2/4$ for $|x|<\pi/3$
and \eqref{eq02-pf-thm-ch3-s-eigen-proj},
we know that there exists $k_j\in\Z$ such that
\begin{equation}
\label{eq03-pf-thm-ch3-s-eigen-proj}
\big| 4R\lambda_0 + \theta_j(\lambda_0) -2k_j\pi \big| \leqslant
2 \epsilon^{1/2} \big\lVert v \big\rVert_Y^{-1/2}.
\end{equation}
For $R$ large enough,
the derivative of the function
$[-3\delta/2,3\delta/2] \ni \lambda \mapsto 4R\lambda + \theta_j(\lambda)$
is greater than $R$.
Thus,
by \eqref{eq03-pf-thm-ch3-s-eigen-proj},
there exists a unique $\lambda_j\in\R$ such that
$4R\lambda_j + \theta_j(\lambda_j) = 2k_j\pi$,
and
\begin{equation}
\label{eq04-pf-thm-ch3-s-eigen-proj}
\big|\lambda_j - \lambda_0\big| <
2 R^{-1} \epsilon^{1/2} \big\lVert v \big\rVert_Y^{-1/2}
\leqslant R^{-1}\big| 4R\lambda_0 + \theta_j(0) - 2k_j\pi \big|.
\end{equation}
Set $w_j = P(\lambda_j)v$.
We have
\begin{equation}
\label{eq05-pf-thm-ch3-s-eigen-proj}
\big\lVert v_j - w_j \big\rVert_Y =
\Big\lVert \big( P_j(\lambda_0) - P_j(\lambda_j) \big) v \Big\rVert_Y =
\mathscr{O}(1) \big| \lambda_0 - \lambda_j \big| \big\lVert v \big\rVert_Y.
\end{equation}
By \eqref{eq04-pf-thm-ch3-s-eigen-proj}
\eqref{eq05-pf-thm-ch3-s-eigen-proj} and
\eqref{eq-thm-app-C},
we get \eqref{eq0-pf-thm-ch3-s-eigen-proj}.

\noindent\textbf{Step 2. }
We show that
for $\omega\in\Omega^\bullet(Z_R,F)$
an eigensection of $D^F_{Z_R}$ associated with $\lambda_0\in B$,
there exists $\mu\in\Omega^\bullet(Z_R,F)$
lying in the image of $\mathscr{J}_{A, B, Z_R}$
such that
\begin{equation}
\label{eq1-pf-thm-ch3-s-eigen-proj}
\big\lVert \omega - \mu \big\rVert_{Z_R}
= \mathscr{O}\big(e^{-aR}\big)
\big\lVert \omega \big\rVert_{Z_R}.
\end{equation}

By \eqref{eq-ch2-zm-norm},
\eqref{eq2-pf-lem-ch3-cut}, \eqref{eq3-pf-lem-ch3-cut}
and \eqref{eq8-pf-lem-ch3-cut},
we have
\begin{equation}
\label{eq12-pf-thm-ch3-s-eigen-proj}
\big\lVert \omega^\mathrm{zm} \big\rVert_Y
\leqslant \big\lVert \omega \big\rVert_{\partial Z_{1,0}}
= \mathscr{O}(1) \big \lVert \omega \big\rVert_{Z_{1,0}\cup Z_{2,0}}.
\end{equation}
On the other hand,
by Lemma \ref{lem-ch3-zm},
we have
\begin{equation}
\label{eq13-pf-thm-ch3-s-eigen-proj}
\big \lVert \omega \big\rVert_{Z_{1,0}\cup Z_{2,0}}
= \mathscr{O}(1) \big\lVert \omega^\mathrm{zm} \big\rVert_Y.
\end{equation}
Since $C_1(\lambda)$ and $C_2(\lambda)$ are unitary,
by Lemma \ref{lem-ch3-cut},
\eqref{eq12-pf-thm-ch3-s-eigen-proj} and
\eqref{eq13-pf-thm-ch3-s-eigen-proj},
we have
\begin{equation}
\label{eq14-pf-thm-ch3-s-eigen-proj}
\big\lVert \omega^{\mathrm{zm},\pm} \big\rVert_Y
= \mathscr{O}(1) \big \lVert \omega \big\rVert_{Z_{1,0}\cup Z_{2,0}}, \quad
\big \lVert \omega \big\rVert_{Z_{1,0}\cup Z_{2,0}}
= \mathscr{O}(1) \big\lVert \omega^{\mathrm{zm},\pm} \big\rVert_Y.
\end{equation}

By Lemma \ref{lem-ch3-cut} and Step 1,
there exist integer $m\leqslant \dim \hh(Y,F)$ and
\begin{equation}
\label{eq15-pf-thm-ch3-s-eigen-proj}
\lambda_j\in(\lambda_0-e^{-aR},\lambda_0+e^{-aR}), \quad
\phi_j,\varphi_j\in \big(1-i\cliffu\big)\hh(Y,F) \quad
\text{with } j=1,\cdots,m,
\end{equation}
such that
$\big(\phi_j\big)_{1\leqslant j\leqslant m}$ are mutually orthogonal
and
\begin{equation}
\label{eq16-pf-thm-ch3-s-eigen-proj}
\omega^{\mathrm{zm},-}\big|_{\partial Z_{1,0}} = \sum_{j=1}^m \phi_j, \quad
e^{4iR\lambda_j}C_{12}(\lambda_j)\varphi_j = \varphi_j, \quad
\big\lVert \varphi_j - \phi_j \big\rVert_Y  <
e^{-aR} \big\lVert \omega^{\mathrm{zm},-} \big\rVert_Y.
\end{equation}
By Proposition \ref{prop-ch2-g-eigen-Dinf},
\eqref{eq-ch3-def-Einf-R} and
the second identity in \eqref{eq15-pf-thm-ch3-s-eigen-proj},
for each $j=1,\cdots,m$,
there exists
$(\widetilde{\omega}_{1,j},\widetilde{\omega}_{1,j}^\mathrm{zm},
\widetilde{\omega}_{2,j},\widetilde{\omega}_{2,j}^\mathrm{zm})
\in \mathscr{E}_{\lambda_j,R}(Z_{12,\infty},F)$
such that
\begin{equation}
\label{eq17-pf-thm-ch3-s-eigen-proj}
\widetilde{\omega}_{1,j}^\mathrm{zm} =
e^{-i\lambda_ju_1} \varphi_j +
e^{ i\lambda_ju_1} C_1(\lambda_j)\varphi_j.
\end{equation}
By \eqref{eq-ch3-def-J},
set
\begin{equation}
\label{eq18-pf-thm-ch3-s-eigen-proj}
\widetilde{\omega} = \sum_{j=1}^m  J_{A,Z_R}
(\widetilde{\omega}_{1,j},\widetilde{\omega}_{1,j}^\mathrm{zm},
\widetilde{\omega}_{2,j},\widetilde{\omega}_{2,j}^\mathrm{zm}).
\end{equation}
By \eqref{eq-ch3-def-J},
\eqref{eq17-pf-thm-ch3-s-eigen-proj}
and \eqref{eq18-pf-thm-ch3-s-eigen-proj},
we have
\begin{equation}
\label{eq19-pf-thm-ch3-s-eigen-proj}
\widetilde{\omega}^\mathrm{zm}
= \sum_{j=1}^m \widetilde{\omega}_{1,j}^\mathrm{zm}
= \sum_{j=1}^m \Big( e^{-i\lambda_ju_1} \varphi_j +
e^{ i\lambda_ju_1} C_1(\lambda_j)\varphi_j \Big).
\end{equation}
Now we take the difference between
\eqref{eq0-pf-lem-ch3-cut} and
\eqref{eq19-pf-thm-ch3-s-eigen-proj}.
Applying Lemma \ref{lem-ch3-cut} and
\eqref{eq14-pf-thm-ch3-s-eigen-proj}-\eqref{eq16-pf-thm-ch3-s-eigen-proj},
we get
\begin{equation}
\label{eq110-pf-thm-ch3-s-eigen-proj}
\big\lVert \omega^\mathrm{zm} - \tilde{\omega}^\mathrm{zm} \big\rVert_{Y_{[-R,R]}}
= \mathscr{O}\big(e^{-aR}\big)
\big \lVert \omega \big\rVert_{Z_{1,0}\cup Z_{2,0}}.
\end{equation}

By \eqref{eq-ch3-def-mathscrJ}, set
\begin{equation}
\label{eq111-pf-thm-ch3-s-eigen-proj}
\mu = \sum_{j=1}^m  \mathscr{J}_{A,B,Z_R}
(\widetilde{\omega}_{1,j},\widetilde{\omega}_{1,j}^\mathrm{zm},
\widetilde{\omega}_{2,j},\widetilde{\omega}_{2,j}^\mathrm{zm}).
\end{equation}
Now we take the difference between
\eqref{eq18-pf-thm-ch3-s-eigen-proj}
and \eqref{eq111-pf-thm-ch3-s-eigen-proj}.
Applying Proposition \ref{prop-ch3-s-eigen-proj-inj},
we get
\begin{equation}
\label{eq112-pf-thm-ch3-s-eigen-proj}
\big\lVert \widetilde{\omega}- \mu \big\rVert_{Z_R}
= \mathscr{O}\big(e^{-aR}\big) \sum_{j=1}^m
\Big( \big\lVert \widetilde{\omega}_{1,j} \big\rVert_{Z_{1,0}}
+ \big\lVert \widetilde{\omega}_{2,j} \big\rVert_{Z_{2,0}} \Big).
\end{equation}
By \eqref{eq-rem-ch2-g-eigen},
\eqref{eq14-pf-thm-ch3-s-eigen-proj},
\eqref{eq16-pf-thm-ch3-s-eigen-proj} and
\eqref{eq112-pf-thm-ch3-s-eigen-proj},
we get
\begin{equation}
\label{eq113-pf-thm-ch3-s-eigen-proj}
\big\lVert \widetilde{\omega}^\mathrm{zm}- \mu^\mathrm{zm} \big\rVert_{Y_{[-R,R]}}
\leqslant \big\lVert \widetilde{\omega}- \mu \big\rVert_{Z_R}
= \mathscr{O}\big(e^{-aR}\big) \sum_{j=1}^m \big\lVert \varphi_j \big\rVert_Y
= \mathscr{O}\big(e^{-aR}\big) \big \lVert \omega \big\rVert_{Z_{1,0}\cup Z_{2,0}}.
\end{equation}
By Lemma \ref{lem-ch3-zm-L2},
\eqref{eq110-pf-thm-ch3-s-eigen-proj} and
\eqref{eq113-pf-thm-ch3-s-eigen-proj},
we obtain \eqref{eq1-pf-thm-ch3-s-eigen-proj}.

\noindent\textbf{Step 3. }
We prove the following result.
Let $\big(V,\lVert\cdot\rVert\big)$
be a Hermitian vector space of dimension $m$.
Let $\{v_j\}_{1\leqslant j\leqslant m}$ be an orthogonal basis.
Let $\{w_j\}_{1\leqslant j\leqslant m}$ be a family of vectors in $V$
satisfying
$\big\lVert v_j-w_j \big\rVert < m^{-1/2} \big\lVert v_j \big\rVert$.
Then $\{w_j\}_{1\leqslant j\leqslant m}$ is a basis of $V$.

Since the condition is homogeneous,
we may assume that $\{v_j\}_{1\leqslant j\leqslant m}$ is an orthonormal basis.
Let $x_1,\cdots,x_m\in\C$ such that
$\sum_{j=1}^m x_jw_j=0$.
We have
\begin{equation}
\sum_{j=1}^m x_jv_j = \sum_{j=1}^m x_j(v_j-w_j).
\end{equation}
Taking the norm of both sides and applying Cauchy-Schwarz inequality,
we get
\begin{equation}
\sum_{j=1}^m |x_j|^2 \leqslant
\Big( \sum_{j=1}^m |x_j| \big\lVert v_j-w_j \big\rVert \Big)^2 \leqslant
\Big( \sum_{j=1}^m |x_j|^2 \Big)
\Big( \sum_{j=1}^m \big\lVert v_j-w_j \big\rVert^2 \Big).
\end{equation}
By the hypothesis, we have $\sum_{j=1}^m \big\lVert v_j-w_j \big\rVert^2 < 1$.
Then we must have $x_1=\cdots=x_m=0$.
Hence $\{w_j\}_{1\leqslant j\leqslant m}$ is a basis.

\noindent\textbf{Step 4. }
We show that $\mathscr{J}_{A, B, Z_R}$ is bijective.

We fix $R_0>2$.
By the proof of Theorem \ref{thm-ch3-gap},
for any $\alpha > 0$,
the function
$R \mapsto \dim \mathscr{E}_{[-\alpha R^{-1},\alpha R^{-1}]}(Z_R,F)$
is non increasing.
Taking $\alpha = \delta R$ and applying Weyl's law,
we get
\begin{align}
\label{eq31-pf-thm-ch3-s-eigen-proj}
\begin{split}
\dim \mathscr{E}_B(Z_R,F)
& \leqslant \dim \mathscr{E}_{[-\delta,\delta]}(Z_R,F) \\
& \leqslant  \dim \mathscr{E}_{[-\delta R/R_0,\delta R/R_0]}(Z_{R_0},F)
= \mathscr{O}\big( R^n \big).
\end{split}
\end{align}
By Step 2, Step 3 and \eqref{eq31-pf-thm-ch3-s-eigen-proj},
$\mathscr{J}_{A, B, Z_R}$ is surjective.
On the other hand,
by Proposition \ref{prop-ch3-s-eigen-proj-inj},
$\mathscr{J}_{A, B, Z_R}$ is injective.
The proof of Theorem \ref{thm-ch3-s-eigen-proj} is completed.
\end{proof}

Recall that the $\Lambda_R$ was defined in \eqref{eq-ch3-def-Lambda}.
We denote
\begin{align}
\begin{split}
\Lambda_R\backslash\{0\} & =
\big\{ \lambda_k \;:\;k\in\Z\backslash\{0\}\big\}, \quad
\text{with } \cdots\leqslant\lambda_{-1}<0<\lambda_1\leqslant\lambda_2\leqslant\cdots ,\\
\Sp\big(D^F_{Z_R}\big) \backslash\{0\} & =
\big\{ \rho_k \;:\;k\in\Z\backslash\{0\}\big\}, \quad
\text{with } \cdots\leqslant\rho_{-1}<0<\rho_1\leqslant\rho_2\leqslant\cdots.
\end{split}
\end{align}

\begin{thm}
\label{thm-ch3-s-eigen}
There exists $a>0$ such that
for any $|\lambda_k|<\delta/2$,
\begin{equation}
\lambda_k-\rho_k = \mathscr{O}\big(e^{-aR}\big).
\end{equation}
\end{thm}
\begin{proof}
We only consider $k>0$.
By the paragraph in Appendix containing \eqref{eq-app-e-C},
\eqref{eq-ch3-def-Lambda} and
\eqref{eq31-pf-thm-ch3-s-eigen-proj},
we have
\begin{equation}
\# \big(\Lambda_R \cap [-\delta,\delta]\big)
= \mathscr{O}\big(R\big), \quad
\# \big(\Sp\big( D^F_{Z_R}\big) \cap [-\delta,\delta]\big)
= \mathscr{O}\big( R^n \big).
\end{equation}
We fix $a>0$ such that Theorem \ref{thm-ch3-s-eigen-proj} holds.
By the pigeonhole principle,
we could construct $0<b_1<\cdots<b_l$ such that
\begin{equation}
b_1 \leqslant R^{-2},\quad
b_l \geqslant 2\delta/3, \quad
b_{j+1} \leqslant b_j + e^{-aR/2}, \quad \text{for } j=1,\cdots,l-1,
\end{equation}
and
\begin{equation}
[b_j-2e^{-aR},b_j+2e^{-aR}] \cap
\Big( \Lambda_R \cup \Sp\big( D^F_{Z_R}\big) \Big) = \emptyset. \\
\end{equation}
Now applying Theorem \ref{thm-ch3-s-eigen-proj}
with $A=(b_j,b_{j+1})$ and $B = (b_j-e^{-aR},b_{j+1}+e^{-aR})$,
we see that
the number of $\lambda_k$ lying in $(b_j,b_{j+1})$
equals to the number of $\rho_k$ lying in $(b_j,b_{j+1})$.
On the other hand,
by Theorem \ref{thm-ch3-gap}
and the paragraph in Appendix containing \eqref{eq-app-e-C},
there is neither $\lambda_k$ nor $\rho_k$ lying in $(0,b_1)$.
Thus $\lambda_k\in(b_j,b_{j+1})$ if and only if $\rho_k\in(b_j,b_{j+1})$.
We get
\begin{equation}
| \lambda_k - \rho_k | < b_{j+1} - b_j \leqslant e^{-aR/2}.
\end{equation}
The proof of Theorem \ref{thm-ch3-s-eigen} is completed.
\end{proof}

For $p=0,\cdots,n$,
set
\begin{align}
\begin{split}
& C^p_{12}(\lambda) = C_{12}(\lambda) \big|_{\hhe^p(Y,F)\oplus \hhe^{p-1}(Y,F)du}, \\
& \Lambda^p_R = \big\{ \lambda>0 \; : \;
\det \big( e^{4i\lambda R} C^p_{12}(\lambda) -\Id \big) = 0 \big\} .
\end{split}
\end{align}
Let $D^{F,2,(p)}_{Z_R}$ be the restriction of $D^{F,2}_{Z_R}$ to $\Omega^p(Z_R,F)$.
We denote
\begin{align}
\begin{split}
\Lambda^p_R & = \big\{ \lambda_k \;:\;k=1,2,\cdots\big\}, \quad
\text{with } 0<\lambda_1\leqslant\lambda_2\leqslant\cdots ,\\
\Sp\big( D^{F,2,(p)}_{Z_R}\big) \backslash\{0\} & = \big\{ \rho_k \;:\;k=1,2,\cdots\big\}, \quad
\text{with } 0<\rho_1\leqslant\rho_2\leqslant\cdots.
\end{split}
\end{align}

\begin{thm}
\label{thm-ch3-s-eigen-degree}
There exists $a>0$ such that
for $\lambda_k<\delta/2$,
\begin{equation}
\label{eq-thm-ch3-s-eigen-degree}
\lambda^2_k-\rho_k = \mathscr{O}\big(e^{-aR}\big).
\end{equation}
\end{thm}
\begin{proof}
We denote
\begin{equation}
\mathscr{E}^p_{A,R}(Z_{12,\infty},F) =
\Big\{s\in\mathscr{E}^p_{A,R}(Z_{12,\infty},F) \;:\;
J_{A,Z_R}(s)\subseteq\Omega^p(Z_R,F)\Big\}.
\end{equation}
Using Propositions \ref{prop-ch2-C-lambda}, \ref{prop-ch2-g-eigen-Dinf},
\eqref{eq-ch2-def-C-lambda-2},
\eqref{eq-ch3-def-Einf-R} and \eqref{eq-ch3-def-Einf-AR},
we could show that
if $A\subseteq\R$ is symmetric
(i.e., $\lambda\in A \Rightarrow -\lambda\in A$),
$\mathscr{E}_{A,R}(Z_{12,\infty},F)$ is homogeneous, i.e.,
\begin{equation}
\mathscr{E}_{A,R}(Z_{12,\infty},F) =
\bigoplus_{p=0}^n \mathscr{E}^p_{A,R}(Z_{12,\infty},F).
\end{equation}
Moreover,
\begin{equation}
\big\{\lambda>0\;:\;
\mathscr{E}^p_{\{\lambda,-\lambda\},R}(Z_{12,\infty},F)\neq 0 \big\}
= \Lambda^p_R.
\end{equation}
On the other hand,
if $B\subseteq\R$ is symmetric,
$\mathscr{E}_B(Z_R,F)$ is homogeneous.
Now,
by Theorem \ref{thm-ch3-s-eigen-proj},
for symmetric subset $A,B\subseteq\R$
satisfying \eqref{eq-thm-ch3-s-eigen-proj},
$\mathscr{J}_{A,B,Z_R}$
maps $\mathscr{E}^p_{A,R}(Z_{12,\infty},F)$ bijectively to
$\mathscr{E}_B^p(Z_R,F)  :=
\mathscr{E}_B(Z_R,F) \cap \Omega^p(Z_R,F)$.
Restricting $\mathscr{J}_{A,B,Z_R}$ to $\mathscr{E}^p_{A,R}(Z_{12,\infty},F)$ and
proceeding in the same way as in the proof of Theorem \ref{thm-ch3-s-eigen},
we obtain \eqref{eq-thm-ch3-s-eigen-degree}.
This completes the proof of Theorem \ref{thm-ch3-s-eigen-degree}.
\end{proof}

\subsection{Boundary case}
\label{ch3-5}

We use the convention from \eqref{eq-ch2-def-hh-bd},
\begin{equation}
\hh_{\bd}(Z_{1,\infty},F) = \hh_\mathrm{rel}(Z_{1,\infty},F), \quad
\hh_{\bd}(Z_{2,\infty},F) = \hh_\mathrm{abs}(Z_{2,\infty},F).
\end{equation}
Recall that the maps
$\mathscr{R}_{d^F, 1}$ and
$\mathscr{R}_{d^F, 2}$
are defined in \eqref{eq-ch3-R}.
For $j=1,2$,
we construct
\begin{equation}
F_{Z_{j,R}} : \; \hh_{\bd}(Z_{j,\infty},F) \rightarrow \Omega_{\bd}^\bullet(Z_{j,R},F)
\end{equation}
as follows,
\begin{align}
\label{eq-ch3-def-FG-bd}
\begin{split}
F_{Z_{j,R}}(\omega,\hat{\omega})\big|_{Z_{j,0}} = \; &
\omega, \quad \text{for } j=1,2, \\
F_{Z_{1,R}}(\omega,\hat{\omega})\big|_{Y_{[0,R]}} = \; &
d^F \big( \chi_{1,R}(\cdot+R)\;\mathscr{R}_{d^F, 1}(\omega,\hat{\omega}) \big) +
\pi_Y^*\hat{\omega},\\
F_{Z_{2,R}}(\omega,\hat{\omega})\big|_{Y_{[-R,0]}} = \; &
d^F \big( \chi_{2,R}(\cdot-R)\;\mathscr{R}_{d^F, 2}(\omega,\hat{\omega}) \big) +
\pi_Y^*\hat{\omega}.
\end{split}
\end{align}
By Proposition \ref{prop-ch2-R},
$F_{Z_{j,R}}$ is well-defined.
Moreover,
we have
\begin{equation}
\label{eq-ch3-F-close-bd}
d^F F_{Z_{j,R}}(\omega,\hat{\omega}) = 0.
\end{equation}

Recall that $\varphi_R: Z_R \rightarrow Z$
is defined by \eqref{eq-ch3-def-varphiR}.
Set
\begin{equation}
\label{eq-ch3-def-varphiR-bd}
\varphi_{j,R} = \varphi_R \big|_{Z_{j,R}} : Z_{j,R} \rightarrow Z_j,
\end{equation}
which is a diffeomorphism.
Then $\varphi_{j,R}$ induces an isomorphism
$\varphi_{j,R,*}: H^\bullet_\bd(Z_{j,R},Y) \rightarrow H^\bullet_\bd(Z_j,F)$.

The following proposition is parallel to Proposition \ref{prop-ch3-F-ind}.

\begin{prop}
\label{prop-ch3-F-ind-bd}
For $\omega_1\in\hhl(Z_{1,\infty},F)$
and $(\omega_2,\hat{\omega})\in\hh_\mathrm{abs}(Z_{2,\infty},F)$,
under the identification $\varphi_{j,R,*}$,
the cohomology classes
\begin{equation}
\label{eq-prop-ch3-F-ind-bd}
\big[ F_{Z_{1,R}}(\omega_1, 0) \big] \in H^\bullet_\mathrm{rel}(Z_1,F), \quad
\big[ F_{Z_{2,R}}(\omega_2,\hat{\omega}) \big] \in H^\bullet_\mathrm{abs}(Z_2,F)
\end{equation}
are independent of $R$.
\end{prop}
\begin{proof}
For $j=1,2$,
let $\overline{Z}_{j,R}$ be another copy of $Z_{j,R}$
with inverse coordinates on its cylindrical part,
i.e.,
we identify the cylindrical part of
$\overline{Z}_{1,R}$ (resp. $\overline{Z}_{2,R}$)
with $Y_{[-R,0]}$ (resp. $Y_{[0,R]}$).
Set $Z_{j,R}^\mathrm{db} = Z_{j,R} \cup_Y \overline{Z}_{j,R}$,
which is a compact manifold without boundary.
Then the cylindrical part of $Z_{j,R}^\mathrm{db}$
is identified with $Y_{[-R,R]}$.

Let $\iota: Z_{j,R}^\mathrm{db}\rightarrow Z_{j,R}^\mathrm{db}$
be the diffeomorphism exchanging the two copies of $Z_{j,R}$.
We extend $(F,h^F)$ $\iota$-equivariantly to $Z_{j,R}^\mathrm{db}$.
We denote by $\iota^*$ the action on
$H^\bullet(Z_{j,R}^\mathrm{db},F)$
induced by $\iota$.
Let
$\Big(H^\bullet(Z_{j,R}^\mathrm{db},F)\Big)^\pm
\subseteq H^\bullet(Z_{j,R}^\mathrm{db},F)$
be the eigenspaces of $\iota^*$ associated with $\pm 1$.
The embedding  $Z_{j,R}\hookrightarrow Z_{j,R}^\mathrm{db}$ induces the following isomorphisms
\begin{equation}
\label{eq1-pf-prop-ch3-F-ind-bd}
H^\bullet_\mathrm{rel}(Z_{1,R},F) = \Big(H^\bullet(Z_{1,R}^\mathrm{db},F)\Big)^-, \quad
H^\bullet_\mathrm{abs}(Z_{2,R},F) = \Big(H^\bullet(Z_{2,R}^\mathrm{db},F)\Big)^+.
\end{equation}

Let $\hh(Z_{1,\infty}^\mathrm{db},F)$ (resp. $\hh(Z_{2,\infty}^\mathrm{db},F)$)
be $\hh(Z_{12,\infty},F)$ (defined by \eqref{eq-ch3-def-hh-12})
with $(Z_{1,\infty},Z_{2,\infty})$ replaced by
$(Z_{1,\infty},\overline{Z}_{1,\infty})$
(resp. $(\overline{Z}_{2,\infty},Z_{2,\infty})$).
We have the following embedding
\begin{align}
\label{eq2-pf-prop-ch3-F-ind-bd}
\begin{split}
\hh_\mathrm{bd}(Z_{j,\infty},F) & \rightarrow \hh(Z_{j,\infty}^\mathrm{db},F) \\
(\omega,\hat{\omega}) & \mapsto (\omega,(-1)^j\omega,\hat{\omega}).
\end{split}
\end{align}
We construct an involution
\begin{align}
\label{eq3-pf-prop-ch3-F-ind-bd}
\begin{split}
\iota^\mathscr{H} : \;
\hh(Z_{j,\infty}^\mathrm{db},F) & \rightarrow \hh(Z_{j,\infty}^\mathrm{db},F) \\
(\omega_1,\omega_2,\hat{\omega}) & \mapsto \left\{
\begin{array}{ll}
(\omega_2,\omega_1, \hat{\omega}) & \text{ if } \hat{\omega}\in\hh(Y,F),\\
(\omega_2,\omega_1,-\hat{\omega}) & \text{ if } \hat{\omega}\in\hh(Y,F)du.
\end{array}\right.
\end{split}
\end{align}
Let
$\Big(\hh(Z_{j,\infty}^\mathrm{db},F)\Big)^\pm
\subseteq \hh(Z_{j,\infty}^\mathrm{db},F)$
be the eigenspace of $\iota^\mathscr{H}$ associated with $\pm 1$.

Under the identification \eqref{eq2-pf-prop-ch3-F-ind-bd},
by \eqref{eq-ch2-def-hh},
\eqref{eq-ch2-def-hh-bd} and
\eqref{eq3-pf-prop-ch3-F-ind-bd},
we have
\begin{equation}
\label{eq4-pf-prop-ch3-F-ind-bd}
\Big(\hh(Z_{1,\infty}^\mathrm{db},F)\Big)^- = \hh_\mathrm{rel}(Z_{1,\infty},F),\quad
\Big(\hh(Z_{2,\infty}^\mathrm{db},F)\Big)^+ = \hh_\mathrm{abs}(Z_{2,\infty},F).
\end{equation}

By \eqref{eq-ch3-def-FG}
and \eqref{eq3-pf-prop-ch3-F-ind-bd},
the following diagram commutes
\begin{equation}
\label{eq5-pf-prop-ch3-F-ind-bd}
\xymatrix{
\hh(Z_{j,\infty}^\mathrm{db},F)
\ar[r]^{\iota^\mathscr{H}}
\ar[d]_{\left[F_{Z_{j,R}^\mathrm{db}}\right]}
& \hh(Z_{j,\infty}^\mathrm{db},F)
\ar[d]^{\left[F_{Z_{j,R}^\mathrm{db}}\right]} \\
H^\bullet(Z_{j,R}^\mathrm{db},F)
\ar[r]^{\iota^*}
& H^\bullet(Z_{j,R}^\mathrm{db},F).}
\end{equation}
Taking the eigenspace associated with $\pm 1$ in \eqref{eq5-pf-prop-ch3-F-ind-bd}
and applying \eqref{eq1-pf-prop-ch3-F-ind-bd} and \eqref{eq4-pf-prop-ch3-F-ind-bd},
we get the following maps
\begin{align}
\label{eq6-pf-prop-ch3-F-ind-bd}
\begin{split}
& \big[F_{Z_{1,R}^\mathrm{db}}\big] :\;
\hh_\mathrm{rel}(Z_{1,\infty},F) \rightarrow
H^\bullet_\mathrm{rel}(Z_{1,R},F), \\
& \big[F_{Z_{2,R}^\mathrm{db}}\big] :\;
\hh_\mathrm{abs}(Z_{2,\infty},F) \rightarrow
H^\bullet_\mathrm{abs}(Z_{2,R},F).
\end{split}
\end{align}
From \eqref{eq-ch3-def-FG}
and \eqref{eq-ch3-def-FG-bd},
we could verify that
the maps in \eqref{eq6-pf-prop-ch3-F-ind-bd}
coincide with $\big[F_{Z_{1,R}}\big]$ and $\big[F_{Z_{2,R}}\big]$.
Now Proposition \ref{prop-ch3-F-ind-bd}
follows from Proposition \ref{prop-ch3-F-ind}.
This completes the proof of Proposition \ref{prop-ch3-F-ind-bd}.
\end{proof}

In the rest of this subsection,
we will state several results parallel to those
in \textsection \ref{ch3-3}, \ref{ch3-4}.

For $R>2$,
we equip $\hh_\bd(Z_{j,\infty},F)$ with the following metric,
\begin{equation}
\label{eq-ch3-def-metric-hh-bd}
\big\lVert (\omega,\hat{\omega}) \big\rVert^2_{\hh_\mathrm{bd}(Z_{j,\infty},F),R} =
\big\lVert \omega \big\rVert^2_{Z_{j,R}}.
\end{equation}

By taking the $(-1)^j$-eigenspace of $\iota$ in $\Omega^\bullet(Z^\mathrm{bd}_{j,R},F)$
and applying Proposition \ref{prop-ch3-approx-f-g-total} to $F_{Z^\mathrm{bd}_{j,R}}$,
we get the following result.

\begin{prop}
\label{prop-ch3-approx-f-g-bd}
There exists $a>0$ such that for
$(\omega,\hat{\omega})\in\hh_\mathrm{bd}(Z_{j,\infty},F)$,
we have
\begin{equation}
\label{eq-prop-ch3-approx-f-g-bd}
\big\lVert F_{Z_{j,R}}(\omega,\hat{\omega})\big\rVert_{Z_{j,R}}
= \Big( 1 + \mathscr{O}\big(e^{-aR}\big) \Big)
\big\lVert (\omega,\hat{\omega}) \big\rVert_{\hh_\mathrm{bd}(Z_{j,\infty},F),R}.
\end{equation}
\end{prop}

Let
$P^{\Ker\left(D^{F,2}_{Z_{j,R}}\right)}:
\Omega^\bullet_\mathrm{bd}(Z_{j,R},F) \rightarrow \Ker\big(D^{F,2}_{Z_{j,R}}\big)$
be the orthogonal projections.
Set
\begin{equation}
\label{eq-ch3-def-FG-harmonic-bd}
\mathscr{F}_{Z_{j,R}} = P^{\Ker\left(D^{F,2}_{Z_{j,R}}\right)} \circ F_{Z_{j,R}} :
\hh_\mathrm{bd}(Z_{j,\infty},F) \rightarrow \Ker\big(D^{F,2}_{Z_{j,R}}\big).
\end{equation}

Again,
by taking the $(-1)^j$-eigenspace of $\iota$ in $\Omega^\bullet(Z^\mathrm{bd}_{j,R},F)$
and applying Proposition \ref{prop-ch3-approx-harmonic-tot-inj} to $F_{Z^\mathrm{bd}_{j,R}}$,
we get the following result.

\begin{prop}
\label{prop-ch3-approx-harmonic-bd-inj}
There exist $a>0$ such that
for $(\omega,\hat{\omega})\in\hh_\mathrm{bd}(Z_{j,\infty},F)$,
we have
\begin{equation}
\big\lVert (F_{Z_{j,R}}-\mathscr{F}_{Z_{j,R}})(\omega,\hat{\omega}) \big\rVert_{\continu\!, Z_{j,R}}
= \mathscr{O}\big(e^{-aR}\big)
\big\lVert (\omega,\hat{\omega}) \big\rVert_{\hh_\mathrm{bd}(Z_{j,\infty},F),R}.
\end{equation}
\end{prop}

Again, by using the $\Z_2$-equivariant version of Theorem \ref{thm-ch3-approx-harmonic-tot} as above,
we get the following result.

\begin{thm}
\label{thm-ch3-approx-harmonic-bd}
For $R>2$ large enough,
the maps $\mathscr{F}_{Z_{j,R}}$ is bijective.
\end{thm}

The following result is a direct consequence of Theorem \ref{thm-ch3-gap},
which is compatible with the $\Z_2$-action.

\begin{thm}
\label{thm-ch3-gap-bd}
There exists $\alpha>0$ such that
\begin{equation}
\Sp\big(D^F_{Z_{j,R}}\big) \subseteq
(-\infty,-\alpha R^{-1}) \cup \{0\} \cup (\alpha R^{-1},+\infty) .
\end{equation}
\end{thm}

For $j = 1,2$,
set
\begin{equation}
\label{eq-ch3-def-C-bd}
C_{j,\bd}(\lambda) = (-1)^j \big( C_j(\lambda)\big|_{\mathscr{H}^\bullet(Y,F)} - C_j(\lambda)\big|_{\mathscr{H}^\bullet(Y,F)du} \big).
\end{equation}
For $p=0,\cdots,n$,
set
\begin{align}
\begin{split}
& C_{j,\bd}^p(\lambda) = C_{j,\bd}(\lambda) \big|_{\hhe^p(Y,F)\oplus\hhe^{p-1}(Y,F)du}, \\
& \Lambda_{j,R}^p = \big\{ \lambda>0 \;:\;
\det \big( e^{2i\lambda R}C_{j,\bd}^p(\lambda) -\Id \big) = 0 \big\}.
\end{split}
\end{align}

Finally the $\Z_2$-equivariant version of Theorem \ref{thm-ch3-s-eigen-degree}
is as follows.

\begin{thm}
\label{thm-ch3-s-eigen-degree-bd}
Theorem \ref{thm-ch3-s-eigen-degree}
holds true for $\big( \Sp \big( D^{F,2,(p)}_{Z_{j,R}} \big), \Lambda_{j,R}^p\big)$.
\end{thm}

We sketch the idea of the proof of Theorem \ref{thm-ch3-s-eigen-degree-bd}.
A direct application of Theorem \ref{thm-ch3-s-eigen-degree} to $Z_{j,R}^\mathrm{db}$
yields the following equation
\begin{equation}
\label{eq-pf-thm-ch3-s-eigen-degree-bd}
0 = \det \big(e^{4i\lambda R}C_{j,\bd}^{p,2}(\lambda) -\Id \big)
= \det \big(e^{2i\lambda R}C_{j,\bd}^p(\lambda) -\Id \big)
\det \big(e^{2i\lambda R}C_{j,\bd}^p(\lambda) + \Id \big).
\end{equation}
And the first factor on the right hand side of \eqref{eq-pf-thm-ch3-s-eigen-degree-bd}
corresponds to the $(-1)^j$-eigenspace.

\section{Asymptotics of zeta determinants}
\label{ch4}

In this section,
we study the asymptotics of zeta determinants
of the Hodge-Laplacians in question.
In \textsection \ref{ch4-1},
we introduce a simplified version of
the model constructed in \textsection \ref{ch3-2}.
In \textsection \ref{ch4-2},
we decompose the zeta determinant into short/large time contribution
and estimate the short time contribution.
In \textsection \ref{ch4-3},
we estimate the large time contribution.

\subsection{Model operators}
\label{ch4-1}

For $R>2$,
we denote
$I_{1,R}=[-R,0]$,
$I_{2,R}=[0,R]$ and
$I_R=[-R,R]$.
We sometimes add a sub-index $0$ to objects associated with $I_R$,
e.g., $I_{0,R}=I_R$.

We denote by $\Omega^\bullet\big(I_R,\hh(Y,F)\big)$
the vector space of differential forms on $I_R$ with values in $\hh(Y,F)$.
For $\omega\in\Omega^p(I_R,\hhe^q(Y,F))$,
we denote $\deg \omega = p+q$,
which we call the total degree of $\omega$.
We have the canonical identification
\begin{equation}
\Omega^\bullet\big(I_R,\hh(Y,F)\big) \simeq \smooth\big(I_R,\hh(Y,F[du])\big).
\end{equation}
For $\omega\in\Omega^\bullet\big(I_R,\hh(Y,F)\big)$ and $u\in I_R$,
we denote by $\omega_u\in\hh(Y,F[du])$ its value at $u$ as in \eqref{eq-ch2-iden-duform}.

Recall that $\LL^\bullet_j\subseteq\hh(Y,F[du])$ ($j=1,2$)
is the set of limiting values of $\hh(Z_{j,\infty},F)$
(see \eqref{eq-ch2-def-LL}).
Let $D_{I_R} = \cliffu \frac{\partial}{\partial u}$ with
\begin{equation}
\label{eq-ch4-dom-D}
\mathrm{Dom} \big( D_{I_R} \big) =
\big\{ \omega\in\Omega^\bullet(I_R,\hh(Y,F)) \;:\;
\omega_{-R}\in\LL_1^\bullet,\;\omega_R\in\LL_2^\bullet \big\}.
\end{equation}
Let $D_{I_{1,R}} = \cliffu \frac{\partial}{\partial u}$
and $D_{I_{2,R}}=\cliffu \frac{\partial}{\partial u}$ with
\begin{align}
\label{eq-ch4-dom-D12}
\begin{split}
\mathrm{Dom} \big( D_{I_{1,R}} \big) & =
\big\{ \omega\in\Omega^\bullet(I_{1,R},\hh(Y,F)) \;:\;
\omega_{-R}\in\LL_1^\bullet,\;\omega_0\in\hh(Y,F)du \big\},\\
\mathrm{Dom} \big( D_{I_{2,R}} \big) & =
\big\{ \omega\in\Omega^\bullet(I_{2,R},\hh(Y,F)) \;:\;
\omega_R\in\LL_2^\bullet,\;\omega_0\in\hh(Y,F) \big\}.
\end{split}
\end{align}
We use the convention $\LL^\bullet_{1,\bd}= \LL^\bullet_{1,\mathrm{rel}}$ and $\LL^\bullet_{2,\bd}= \LL^\bullet_{2,\mathrm{abs}}$.
The operator $D^2_{I_{j,R}}$ ($j=0,1,2$) preserves the total degree.
We denote by $D^{2,(p)}_{I_{j,R}}$ its restriction to the component of total degree $p$.
Let $\LL^\bullet_{j,\mathrm{abs/rel}}\subseteq\LL^\bullet_j$ ($j=1,2$)
be the absolute/relative component
(see \eqref{eq-ch2-def-LL-bd}).
By \eqref{eq-ch4-dom-D} and \eqref{eq-ch4-dom-D12},
we have
\begin{equation}
\Ker \big( D^{2,(p)}_{I_R} \big) = \LL_1^p \cap \LL_2^p, \quad
\Ker \big( D^{2,(p)}_{I_{j,R}} \big) = \LL_{j,\bd}^p, \quad \text{for } j=1,2,
\end{equation}
where the elements in $\LL_1^p \cap \LL_2^p$ (resp. $\LL_{j,\bd}^p$)
are viewed as constant functions on $I_R$ (resp. $I_{j,R}$).
Here we remark that the boundary condition for $D^2_{I_{j,R}}$
is defined in the same way as in \eqref{eq-ch1-def-bd-D2}.

Let
\begin{equation}
\label{eq-ch4-def-alpha}
\alpha_{p,\LL} : \; \LL_{1,\mathrm{rel}}^p \rightarrow \LL_1^p\cap \LL_2^p
\end{equation}
be the composition of
the orthogonal projection
$\LL_{1,\mathrm{rel}}^p \rightarrow \LL_{1, \mathrm{rel}}^p \cap \LL_{2,\mathrm{rel}}^p$
and the embedding
$\LL_{1, \mathrm{rel}}^p \cap \LL_{2,\mathrm{rel}}^p \hookrightarrow \LL_1^p\cap \LL_2^p$.
Let
\begin{equation}
\label{eq-ch4-def-beta}
\beta_{p,\LL} : \; \LL_1^p\cap \LL_2^p \rightarrow  \LL_{2,\mathrm{abs}}^p
\end{equation}
be the composition of
the orthogonal projection
$\LL_1^p\cap \LL_2^p \rightarrow \LL_{1,\mathrm{abs}}^p\cap \LL_{2, \text{abs}}^p$
and the embedding
$\LL_{1,\mathrm{abs}}^p \cap \LL_{2, \text{abs}}^p \hookrightarrow  \LL_{2,\mathrm{abs}}^p$.
Let
\begin{equation}
\label{eq-ch4-def-delta}
\delta_{p,\LL} : \; \LL_{2,\mathrm{abs}}^p \rightarrow \LL_{1, \mathrm{rel}}^{p+1}
\end{equation}
be the composition of
the map
$du\wedge: \LL_{2,\mathrm{abs}}^p \rightarrow \LL_{2,\mathrm{abs}}^p du = \LL_{2,\mathrm{rel}}^{p+1, \perp}$ (cf. \eqref{eq-ch2-def-LL-bd-2}),
the orthogonal projection
$\LL_{2,\mathrm{rel}}^{p+1, \perp} \rightarrow \LL_{1, \mathrm{rel}}^{p+1}\cap \LL_{2,\text{rel}}^{p+1, \perp}$
and the embedding
$\LL_{1, \mathrm{rel}}^{p+1}\cap \LL_{2, \text{rel}}^{p+1, \perp} \hookrightarrow \LL_{1, \mathrm{rel}}^{p+1}$.
The maps in
\eqref{eq-ch4-def-alpha},
\eqref{eq-ch4-def-beta} and
\eqref{eq-ch4-def-delta}
form an exact sequence
\begin{equation}
\label{eq-ch4-mv-LL}
\xymatrix{
\cdots \ar[r] &
\LL_{1,\mathrm{rel}}^p   \ar[r]^{\alpha_{p,\LL}} &
\LL_1^p \cap \LL_2^p \ar[r]^{\beta_{p,\LL}} &
\LL_{2,\mathrm{abs}}^p \ar[r]^{\delta_{p,\LL}} &
\cdots.
}
\end{equation}
The exactness of \eqref{eq-ch4-mv-LL}
is guaranteed by the following identities
from \eqref{eq-ch2-def-LL-bd} and \eqref{eq-ch2-def-LL-bd-2},
\begin{align}
\label{eq-ch4-exact-grnt}
\begin{split}
& \im(\alpha_{p,\LL}) =
\Ker(\beta_{p,\LL}) =
\LL_{1, \mathrm{rel}}^p \cap \LL_{2,\mathrm{rel}}^p, \\
& \im(\beta_{p,\LL}) =
\Ker(\delta_{p,\LL}) =
\LL_{1, \mathrm{abs}}^p \cap \LL_{2,\mathrm{abs}}^p, \\
& \im(\delta_{p,\LL}) =
\Ker(\alpha_{p+1,\LL}) =
\LL_{1, \mathrm{rel}}^{p+1} \cap \LL_{2, \mathrm{rel}}^{p+1, \perp}.
\end{split}
\end{align}

We denote $C_{12}=C_{12}(0)$ (resp. $C_{j,\bd}=C_{j,\bd}(0)$).
Let $C_{12}^p$ (resp. $C_{j,\bd}^p$) be its restriction to $\hhe^p(Y,F)\oplus\hhe^{p-1}(Y,F) du$.
By \eqref{eq-ch2-LL-C},
$C_j$ is the reflection with respect to $\LL_j$.
Then,
by \eqref{eq-ch2-def-LL-bd}, \eqref{eq-ch2-def-LL-bd-2}
and \eqref{eq-ch3-def-C-bd},
we have
\begin{align}
\label{eq-ch4-ker-C}
\begin{split}
\Ker \big( C_{1,\mathrm{rel}}^p -\Id \big) & =
\LL_{1,\mathrm{rel}}^p \oplus \cntrtu \LL_{1,\mathrm{rel}}^{p+1},\\
\Ker \big( C_{2,\mathrm{abs}}^p -\Id\big) & =
\LL_{2,\mathrm{abs}}^p \oplus \LL_{2,\mathrm{abs}}^{p-1}du ,\\
\Ker \big( C_{12}^p - \Id \big) & =
\big( \LL_1^p \cap \LL_2^p \big) \oplus
\cntrtu \big( \LL_{1,\mathrm{rel}}^{p+1} \cap \LL_{2,\mathrm{rel}}^{p+1} \big) \oplus
\big( \LL_{1,\mathrm{abs}}^{p-1} \cap \LL_{2,\mathrm{abs}}^{p-1} \big)du.
\end{split}
\end{align}
Rcall that $\chi'$ is defined in \eqref{eq-ch0-def-chi}.
For $C\in\big\{C_{12}, C_{1,\mathrm{rel}}, C_{2,\mathrm{abs}}\big\}$,
we denote
\begin{equation}
\label{eq-ch4-def-chi-C}
\chi'(C) = \sum_p (-1)^p p \dim \Ker \big( C^p - \Id\big).
\end{equation}

\begin{lemma}
\label{lem-ch4-chi}
The following identity holds,
\begin{equation}
\label{eq-lem-ch4-chi}
\chi'(C_{12}) - \chi'(C_{1,\mathrm{abs}}) - \chi'(C_{2,\mathrm{rel}}) =  2\chi'.
\end{equation}
\end{lemma}
\begin{proof}
We denote
\begin{align}
\label{eq0-pf-lem-ch4-chi}
\begin{split}
& h_p = \dim \hhe^p(Y,F), \quad
x_p = \dim \LL_{1,\mathrm{abs}}^p, \quad
y_p = \dim \LL_{2,\mathrm{abs}}^p, \\
& u_p = \dim \big(\LL_{1,\mathrm{abs}}^p \cap \LL_{2,\mathrm{abs}}^p\big), \quad
v_p = \dim \big(\LL_{1,\mathrm{abs}}^{p,\perp} \cap \LL_{2,\mathrm{abs}}^{p,\perp}\big).
\end{split}
\end{align}
By \eqref{eq-ch2-def-LL-bd} and \eqref{eq-ch2-def-LL-bd-2},
we have
\begin{equation}
\label{eq1-pf-lem-ch4-chi}
\dim \LL_{1,\mathrm{rel}}^{p+1} = h_p - x_p, \quad
\dim \LL_{2,\mathrm{rel}}^{p+1} = h_p - y_p, \quad
\dim \big( \LL_{1,\mathrm{rel}}^{p+1} \cap \LL_{2,\mathrm{rel}}^{p+1} \big) = v_p.
\end{equation}
Since
$\hhe^p(Y,F) = (\LL_{1,\mathrm{abs}}^p + \LL_{2,\mathrm{abs}}^p)
\oplus (\LL_{1,\mathrm{abs}}^{p,\perp}\cap\LL_{2,\mathrm{abs}}^{p,\perp})$,
we have
\begin{equation}
\label{eq2-pf-lem-ch4-chi}
h_p = x_p + y_p - u_p + v_p.
\end{equation}
By \eqref{eq-ch2-def-LL-bd},
\eqref{eq-ch2-def-LL-bd-2},
\eqref{eq-ch4-ker-C},
\eqref{eq-ch4-def-chi-C},
\eqref{eq1-pf-lem-ch4-chi} and
\eqref{eq2-pf-lem-ch4-chi},
we have
\begin{align}
\label{eq3-pf-lem-ch4-chi}
\begin{split}
\chi'(C_{12}) - \chi'(C_{1,\bd}) - \chi'(C_{2,\bd}) & = \sum_p 2 (-1)^p (y_p-u_p), \\
\sum_p (-1)^pp
\Big\{ \dim \big( \LL_1^p \cap \LL_2^p \big)  - \dim \LL_{1,\mathrm{rel}}^p - \dim \LL_{2,\mathrm{abs}}^p \Big\}
& = \sum_p (-1)^p (y_p-u_p).
\end{split}
\end{align}
By \eqref{eq-ch0-def-chi}
and \eqref{eq3-pf-lem-ch4-chi},
it remains to show that
\begin{align}
\label{eq4-pf-lem-ch4-chi}
\begin{split}
& \dim \big( \LL_1^p \cap \LL_2^p \big) - \dim \LL_{1,\mathrm{rel}}^p - \dim \LL_{2,\mathrm{abs}}^p  \\
& = \dim H^p(Z,F) - \dim H^p_\mathrm{rel}(Z_1,F) - \dim H^p_\mathrm{abs}(Z_2,F) .
\end{split}
\end{align}
By Theorems \ref{thm-ch1-Hodge},
\ref{thm-ch3-approx-harmonic-tot},
\ref{thm-ch3-approx-harmonic-bd},
\eqref{eq4-pf-lem-ch4-chi} is equivalent to
\begin{align}
\begin{split}
& \dim \LL_1^p \cap \LL_2^p  - \dim \LL_{1,\mathrm{rel}}^p - \dim \LL_{2,\mathrm{abs}}^p  \\
& = \dim \hhe^p(Z_{12,\infty},F) - \dim \hhe^p_\mathrm{rel}(Z_{1,\infty},F) - \dim \hhe^p_\mathrm{abs}(Z_{2,\infty},F),
\end{split}
\end{align}
which follows from \eqref{eq-ch2-ses-hh-bd} and \eqref{eq-ch3-ses-hh}.
The proof of Lemma \ref{lem-ch4-chi} is completed.
\end{proof}

\begin{lemma}
\label{lem-ch4-chi-abd}
The following identities holds,
\begin{align}
\begin{split}
\label{eq-lem-ch4-chi-abd}
\chi' & = \sum_p (-1)^p \dim \im(\delta_{p,\LL}), \\
\chi'(C_{12}) & = \sum_p (-1)^p \Big(\dim \im (\alpha_{p,\LL}) - \dim \im (\beta_{p,\LL})\Big).
\end{split}
\end{align}
\end{lemma}
\begin{proof}
By \eqref{eq-ch4-exact-grnt} and \eqref{eq1-pf-lem-ch4-chi},
we have
\begin{equation}
\label{eq1-pf-lem-ch4-chi-abd}
\dim \mathrm{Im}(\delta_{p,\LL}) =
h_p - x_p - v_p = y_p - u_p.
\end{equation}
From
\eqref{eq-ch4-exact-grnt},
\eqref{eq-lem-ch4-chi},
\eqref{eq0-pf-lem-ch4-chi},
\eqref{eq3-pf-lem-ch4-chi} and
\eqref{eq1-pf-lem-ch4-chi-abd},
we get \eqref{eq-lem-ch4-chi-abd}.
This completes the proof of Lemma \ref{lem-ch4-chi-abd}.
\end{proof}

For $R>2$, set
\begin{align}
\begin{split}
\Lambda^{*,p}_R & = \big\{ \lambda>0 \;:\; \det \big( e^{4i\lambda R}C^p_{12} -\Id \big) = 0 \big\}, \\
\Lambda^{*,p}_{j,R} & = \big\{ \lambda>0 \;:\; \det \big( e^{2i\lambda R}C^p_{j,\bd} -\Id \big) = 0 \big\} ,\hspace{5mm}\text{for } j = 1,2.
\end{split}
\end{align}

\begin{prop}
\label{prop-ch4-Lambda-star}
We have
\begin{equation}
\label{eq-prop-ch4-Lambda-star}
\Sp\big( D^{2,(p)}_{I_{j,R}} \big) \backslash \{0\} =
\big\{ \lambda^2 \;:\; \lambda\in\Lambda^{*,p}_{j,R} \big\}, \quad \text{for } j=0,1,2.
\end{equation}
\end{prop}
\begin{proof}
We may proceed in the same way as in \textsection \ref{ch3}
with $Z_{1,R}$, $Z_{2,R}$ and $F$
replaced by $I_{1,R}$, $I_{2,R}$ and $\hh(Y,F)$.
In particular,
if $\omega\in\Omega^\bullet(I_{j,R},\hh(Y,F))$
is an eigensection of $D^{2,(p)}_{I_{j,R}}$ associated with eigenvalue $\lambda^2$,
then
\begin{equation}
\label{eq1-pf-prop-ch4-Lambda-star}
\omega =
e^{-i\lambda u}\big(\phi^--i\cliffu\phi^-\big) +
e^{i\lambda u}\big(\phi^++i\cliffu\phi^+\big)
\end{equation}
with $\phi^\pm\in\hh(Y,F)$.
From the boundary condition (see \eqref{eq-ch4-dom-D} and \eqref{eq-ch4-dom-D12})
and \eqref{eq1-pf-prop-ch4-Lambda-star},
we get \eqref{eq-prop-ch4-Lambda-star}.
This completes the proof of Proposition \ref{prop-ch4-Lambda-star}.
\end{proof}

Let $\theta^{*}_{j,R}(s)$ ($j=0,1,2$)
be the zeta-type functions associated with $D^2_{I_{j,R}}$
(see Definition \ref{def-ch1-zeta}).

\begin{prop}
\label{prop-ch4-zeta-star}
The following identities hold,
\begin{align}
\begin{split}
\theta^{*}_{R}{}'(0) & =  \chi'(C_{12}) \log(2R) - \chi(Y)\mathrm{rk}(F) \log 2 \\
& \quad +
\sum_{p=0}^{\dim Y}  \frac{p}{2} (-1)^p \log {\det}^* \Big( \frac{2- C_{12}^p - (C_{12}^p)^{-1}}{4} \Big), \\
\theta^{*}_{j,R}{}'(0) & =  \chi'(C_{j,\bd})\log R - \chi(Y,F)\log 2, \quad \text{for } j=1,2.
\end{split}
\end{align}
\end{prop}
\begin{proof}
Both identities are consequences of
Proposition \ref{prop-ch4-Lambda-star}
and Proposition \ref{prop-Lambdazeta-der0} in the Appendix.
The first (resp. second) identity is a weighted sum of \eqref{eq-Lambdazeta-der0} with
$V$ replaced by $\hhe^p(Y,F)\oplus\hhe^{p-1}(Y,F)du$ and
$C$ replaced by $C^p_{12}$ (resp. $C^p_{j,\bd}$).
Since $\Sp\big(C^p_{j,\bd}\big)\subseteq\big\{-1,1\big\}$,
there is no $\log\det^*$ term in the second identity.
This completes the proof of Proposition \ref{prop-ch4-zeta-star}.
\end{proof}

\subsection{Small time contributions}
\label{ch4-2}

Henceforth we add a sub-index $0$ to objects associated with $Z_R$,
e.g., $Z_{0,R}=Z_{R}$, $\theta_{0,R}=\theta_R$.

We fix $\e\in(0,1)$.
For $j=0,1,2$,
we denote by $\theta^{S}_{j,R}(s)$ (resp. $\theta^{L}_{j,R}(s)$)
be the contribution of $\int_0^{R^{2-\e}}$ (resp. $\int_{R^{2-\e}}^\infty$)
to $\theta_{j,R}(s)$ in Definition \ref{def-ch1-zeta}.
Then we have
\begin{equation}
\theta_{j,R}(s) = \theta^S_{j,R}(s) + \theta^L_{j,R}(s).
\end{equation}
We define $\theta^{*S/L}_{j,R}$ ($j=0,1,2$) in the same way.

We denote
\begin{align}
\label{eq-ch4-def-theta}
\begin{split}
\Theta_R(t) & =  \sum_{j=0}^2 (-1)^{\frac{(j-1)(j-2)}{2}}
\tr\Big[(-1)^NN\exp\big(-t D^{F,2}_{Z_{j,R}}\big)\Big], \\
\Theta^*_R(t) & = \sum_{j=0}^2 (-1)^{\frac{(j-1)(j-2)}{2}}
\tr\Big[(-1)^NN\exp\big(-t D^2_{I_{j,R}}\big)\Big].
\end{split}
\end{align}
By Definition \ref{def-ch1-zeta},
and the constructions of $\theta^S_{j,R}(s)$, $\theta^{*S}_{j,R}(s)$
we have
\begin{equation}
\label{eq-ch4-s-zeta-theta}
\sum_{j=0}^2 (-1)^{\frac{(j-1)(j-2)}{2}} \big( \theta^S_{j,R}(s) - \theta^{*S}_{j,R}(s) \big)
=  - \frac{1}{\Gamma(s)} \int_0^{R^{2-\e}} t^{s-1} \big( \Theta_R(t) - \Theta^*_R(t) \big) dt.
\end{equation}

\begin{thm}
\label{thm-ch4-s-time}
There exists $a>0$ such that
\begin{equation}
\sum_{j=0}^2 (-1)^{(j-1)(j-2)/2} \big( \theta^S_{j,R}{}'(0) - \theta^{*S}_{j,R}{}'(0) \big)
= \mathscr{O}\big(e^{-aR^{\e/2}}\big).
\end{equation}
\end{thm}
\begin{proof}
We will follow \cite[\textsection 13(b)]{BL91}.

Let $f\in\smooth(\R)$ be an even function such that
$f(u)=1$ for $|u|\leqslant 1/2$ and
$f(u)=0$ for $|u|\geqslant 1$.
For $t,\varsigma>0$ and $z\in\C$,
set
\begin{align}
\label{eq1-pf-thm-ch4-s-time}
\begin{split}
F_{t,\varsigma}(z) & = \int_{-\infty}^{+\infty} e^{i\sqrt{2}vz} e^{-\frac{1}{2}v^2} f\big(\sqrt{\varsigma t}v\big) \frac{dv}{\sqrt{2\pi}} ,\\
G_{t,\varsigma}(z) & = \int_{-\infty}^{+\infty} e^{i\sqrt{2}vz} e^{-\frac{1}{2}v^2/t} \left( 1 - f\big(\sqrt{\varsigma}v\big) \right) \frac{dv}{\sqrt{2\pi t}} .\\
\end{split}
\end{align}
We have
\begin{equation}
\label{eq2-pf-thm-ch4-s-time}
F_{t,\varsigma}\big(\sqrt{t}z\big) + G_{t,\varsigma}\big(z\big) = \exp\big(-tz^2\big) .
\end{equation}
Let
\begin{align}
\label{eq3-pf-thm-ch4-s-time}
\begin{split}
& F_{t,\varsigma}\big(\sqrt{t}D^F_{Z_{j,R}}\big)(x,y), \quad
G_{t,\varsigma}\big(D^F_{Z_{j,R}}\big)(x,y) \\
& \quad \in
\left(\Lambda^\bullet\big(T^*Z_{j,R}\big) \otimes F\right)_x \otimes
\left(\Lambda^\bullet\big(T^* Z_{j,R}\big) \otimes F\right)^*_y,
\quad \text{where } x,y\in Z_R,
\end{split}
\end{align}
be the smooth kernel of operators $F_{t,\varsigma}\big(\sqrt{t}D^F_{Z_{j,R}}\big)$ and $G_{t,\varsigma}\big(D^F_{Z_{j,R}}\big)$
(with respect to the Riemannian volume form).

By the construction of $G_{t,\varsigma}(z)$,
for any $k\in\N$,
there exist $C>0$ and $a>0$ such that
for $t>0$, $0<\varsigma<1$ and $z\in\C$,
we have (cf. the argument of \cite[(1.6.16)]{MaMa07})
\begin{equation}
\label{eq4-pf-thm-ch4-s-time}
\big| z^k G_{t,\varsigma}(z) \big| \leqslant C e^{-a/{\varsigma t}} .
\end{equation}
Since $0<\varsigma<1$,
we have
\begin{equation}
\label{eq4a-pf-thm-ch4-s-time}
e^{-a/{\varsigma t}} \leqslant
e^{-a/{2t}} e^{-a/{2\varsigma t}} \leqslant
2a^{-1}te^{-a/{2\varsigma t}}.
\end{equation}
By \eqref{eq4-pf-thm-ch4-s-time} and
\eqref{eq4a-pf-thm-ch4-s-time},
for $k,k'\in\N$, there exist $C>0$ and $a>0$ such that
for $0<t<R^{2-\e}$ and $0<\varsigma<R^{-2+\e/2}$,
we have
\begin{equation}
\label{eq5-pf-thm-ch4-s-time}
\left\lVert D^{F,k}_{Z_{j,R}}
G_{t,\varsigma}\big( D^F_{Z_{j,R}} \big)
D^{F,k'}_{Z_{j,R}} \right\rVert^{0,0}
\leqslant C t e^{-a R^{\e/2}},
\end{equation}
where $\lVert\,\cdot\,\rVert^{0,0}$ is the operator norm associated with the $L^2$-norm.
By Proposition \ref{prop-ch3-sobolev}
and \eqref{eq5-pf-thm-ch4-s-time},
there exist $C>0$ and $a>0$ such that
for $0<t<R^{2-\e}$, $0<\varsigma<R^{-2+\e/2}$
and $x,y\in Z_{j,R}$,
we have
\begin{equation}
\label{eq6-pf-thm-ch4-s-time}
\big| G_{t,\varsigma}\big(D^F_{Z_{j,R}}\big)(x,y) \big|
\leqslant C t e^{-a R^{\e/2}}.
\end{equation}

In the rest of the proof,
we take $\varsigma = R^{-2+\e/3}$.
By the finite propagation speed of the wave equation for $D^F_{Z_{j,R}}$ (cf. \cite[Appendix D.2]{MaMa07}),
if the distance between $x$ and $y$ is greater than $\varsigma^{-1/2}$,
we have $F_{t,\varsigma}\big(\sqrt{t}D^F_{Z_{j,R}}\big)(x,y)=0$.
Moreover,
$F_{t,\varsigma}\big(\sqrt{t}D^F_{Z_{j,R}}\big)(x,\cdot)$
only depends on the restriction of $D^F_{Z_{j,R}}$
on the ball of radius $\varsigma^{-1/2}$ centered at $x$.
Thus, for $x\in Z_{j,R/2}\subseteq Z_{j,R}\subseteq Z_R$ ($j=1,2$),
we have
\begin{equation}
\label{eq7-pf-thm-ch4-s-time}
F_{t,\varsigma}\big(\sqrt{t}D^F_{Z_{j,R}}\big)(x,x) =
F_{t,\varsigma}\big(\sqrt{t}D^F_{Z_R}\big)(x,x).
\end{equation}
We may view $Y_{(-\frac{R}{2},\frac{R}{2})}\subset Z_R$ as a subset of $Y_\R$.
Let $D^F_{Y_\R}$ be the Hodge-de Rham operator on $\Omega^\bullet(Y_\R,F)$.
Let $\iota$ be the involution on $Y_\R$ sending $(u,y)$ to $(-u,y)$.
For $x\in Y_{(-\frac{R}{2},\frac{R}{2})}\cap Z_{j,R}$ ($j=1,2$) (cf. \eqref{eq-ch3-YR}),
by the argument in the proof of Proposition \ref{prop-ch3-F-ind-bd}
involving the involution $\iota$,
we have
\begin{align}
\label{eq8-pf-thm-ch4-s-time}
\begin{split}
F_{t,\varsigma}\big(\sqrt{t}D^F_{Z_{j,R}}\big)(x,x)
& = F_{t,\varsigma}\big(\sqrt{t}D^F_{Z^\mathrm{db}_{j,R}}\big)(x,x) +
(-1)^j F_{t,\varsigma}\big(\sqrt{t}D^F_{Z^\mathrm{db}_{j,R}}\big)(x,\iota x) \\
& = F_{t,\varsigma}\big(\sqrt{t}D^F_{Y_\R}\big)(x,x) +
(-1)^j F_{t,\varsigma}\big(\sqrt{t}D^F_{Y_\R}\big)(x,\iota x).
\end{split}
\end{align}
As a consequence,
for $x\in Y_{(-\frac{R}{2},\frac{R}{2})} \cap Z_{1,R} =Y_{(-\frac{R}{2},0)}$,
we have
\begin{align}
\label{eq9-pf-thm-ch4-s-time}
\begin{split}
& F_{t,\varsigma}\big(\sqrt{t}D^F_{Z_{1,R}}\big)(x,x) +
\iota^* F_{t,\varsigma}\big(\sqrt{t}D^F_{Z_{2,R}}\big)(\iota x,\iota x) \\
& = F_{t,\varsigma}\big(\sqrt{t}D^F_{Z_R}\big) (x,x) +
\iota^* F_{t,\varsigma}\big(\sqrt{t}D^F_{Z_R}\big)(\iota x,\iota x)
\in \End \left(\Lambda^\bullet\big(T^*Z_{j,R}\big) \otimes F\right)_x.
\end{split}
\end{align}

By \eqref{eq2-pf-thm-ch4-s-time},
we may decompose $\Theta_R(t)$ into the contributions of $F_{t,\varsigma}$ and $G_{t,\varsigma}$.
By \eqref{eq7-pf-thm-ch4-s-time} and \eqref{eq9-pf-thm-ch4-s-time},
the contribution of $F_{t,\varsigma}$ to \eqref{eq-ch4-s-zeta-theta}
vanishes identically.
By \eqref{eq6-pf-thm-ch4-s-time},
the contribution of $G_{t,\varsigma}$ to \eqref{eq-ch4-s-zeta-theta}
together with its derivative at $s=0$ is $\mathscr{O}\big(e^{-aR^{\e/2}}\big)$.
The same argument works for $\Theta^*_R(t)$.
The proof of Theorem \ref{thm-ch4-s-time} is completed.
\end{proof}

\subsection{Large time contributions and a proof of Theorem \ref{thm-ch0-1}}
\label{ch4-3}

By Definition \ref{def-ch1-zeta},
Lemma \ref{lem-ch4-chi} and
\eqref{eq-ch4-def-theta},
we have
\begin{equation}
\sum_{j=0}^2 (-1)^{(j-1)(j-2)/2} \big( \theta^L_{j,R}(s) - \theta^{*L}_{j,R}(s) \big)
=-\frac{1}{\Gamma(s)}\int_{R^{2-\e}}^{+\infty} t^{s-1} \big( \Theta_R(t) - \Theta^*_R(t) \big) dt .
\end{equation}
We fix $\kappa\in(\e,1)$.
We denote by $\Theta_R^I(t)$ (resp. $\Theta_R^{I\!I}(t)$)
the contribution of $\Sp\big(D^{F,2}_{Z_{j,R}}\big) \cap (0,R^{-2+\kappa})$
(resp. $\Sp\big(D^{F,2}_{Z_{j,R}}\big) \cap [R^{-2+\kappa},+\infty)$)
to $\Theta_R(t)$.
We define $\Theta^{*,I}_R(t)$ and $\Theta^{*,I\!I}_R(t)$ in the same way.

\begin{prop}
\label{prop-ch4-traceI}
We have
\begin{equation}
\label{eq-prop-ch4-traceI}
\int^{+\infty}_{R^{2-\e}} \Theta_R^{I\!I}(t) \frac{dt}{t} =
\mathscr{O}\big(e^{-\frac{1}{2}R^{\kappa-\e}}\big), \quad
\int^{+\infty}_{R^{2-\e}} \Theta^{*,I\!I}_R(t) \frac{dt}{t} =
\mathscr{O}\big(e^{-\frac{1}{2}R^{\kappa-\e}}\big).
\end{equation}
\end{prop}
\begin{proof}
Let $\{\lambda_k(R)\}$ be the set of eigenvalues of $D_{Z_{j,R}}^{F,2}$ ($j=0,1,2$)
such that $\lambda_k(R) \geqslant R^{-2+\kappa}$.
Then,
for $t\geqslant R^{2-\e}$,
we have
\begin{align}
\label{eq1-pf-prop-ch4-traceI}
\begin{split}
\left| \Theta^{I\!I}_R(t) \right| &
\leqslant n \sum_k e^{-t\lambda_k(R)}
\leqslant n e^{-(t-1)R^{-2+\kappa}}  \sum_k e^{-\lambda_k(R)} \\
& \leqslant n e^{-(t-1)R^{-2+\kappa}} \sum_{j=0}^2 \tr\big[\exp\big(- D^{F,2}_{Z_{j,R}}\big)\big].
\end{split}
\end{align}
Let $\exp\big(- D^{F,2}_{Z_{j,R}}\big)(x,y)$ ($x,y\in Z_{j,R}$)
be the smooth kernel of the operator  $\exp\big(- D^{F,2}_{Z_{j,R}}\big)$.
Proceeding in the same way as in the proof of Theorem \ref{thm-ch4-s-time},
there exists $a>0$ such that for any $x,y \in Z_{j,R}$,
$\left| \exp\big(- D^{F,2}_{Z_{j,R}}\big)(x,y) \right| \leqslant a$.
As a consequence,
there exist $a,b>0$, such that
\begin{equation}
\label{eq2-pf-prop-ch4-traceI}
\tr\big[\exp\big(- D^{F,2}_{Z_{j,R}}\big)\big]
\leqslant a \mathrm{Vol}(Z_{j,R})
\leqslant bR,
\quad \text{for } j=0,1,2.
\end{equation}
By \eqref{eq1-pf-prop-ch4-traceI}
and \eqref{eq2-pf-prop-ch4-traceI},
we get the first estimate in \eqref{eq-prop-ch4-traceI}.
The second one can be established in the same way.
The proof of Proposition \ref{prop-ch4-traceI} is completed.
\end{proof}

\begin{prop}
\label{prop-ch4-traceII}
We have
\begin{equation}
\label{eq-prop-ch4-traceII}
\int^{+\infty}_{R^{2-\e}} \big( \Theta_R^I(t) - \Theta^{*,I}_R(t) \big) \frac{dt}{t} =
\mathscr{O}\left(R^{\kappa-1}\right).
\end{equation}
\end{prop}
\begin{proof}
For $\lambda>0$,
we denote
\begin{equation}
\label{eq1-pf-prop-ch4-traceII}
e_R(\lambda) =
\int_{R^{2-\e}}^{+\infty} t^{-1}e^{-t\lambda}dt =
\int_{R^{2-\e}\lambda}^{+\infty} t^{-1} e^{-t} dt.
\end{equation}
By splitting the integral into $\int_1^{+\infty} + \int_{R^{2-\e}\lambda}^1$
(if $R^{2-\e}\lambda\leqslant 1$),
we have
\begin{equation}
\label{eq2-pf-prop-ch4-traceII}
\big|e_R(\lambda)\big|
\leqslant 1 + \max \big\{ -\log\big(R^{2-\e}\lambda\big),\; 0 \big\}, \quad
\big|e_R{}'(\lambda)\big|
\leqslant \lambda^{-1} .
\end{equation}

For a finite multiset $\Lambda\subseteq\R$,
we denote
$e_R[\Lambda] = \sum_{\lambda\in\Lambda}e_R(\lambda)$.
We have
\begin{align}
\label{eq3-pf-prop-ch4-traceII}
\begin{split}
& \int^{+\infty}_{R^{2-\e}} \big( \Theta_R^I(t) - \Theta^{*,I}_R(t) \big) \frac{dt}{t} \\
& = \sum_{j=0}^2 \sum_p (-1)^{\frac{(j-1)(j-2)}{2}+p}p \Big\{
e_R\big[\Sp \big( D^{F,2,(p)}_{Z_{j,R}} \big)\cap(0,R^{\kappa-2})\big] \\
& \hspace{50mm} - e_R\big[\Sp \big( D^{2,(p)}_{I_{j,R}} \big)\cap(0,R^{\kappa-2})\big] \Big\} .
\end{split}
\end{align}
By Theorems \ref{thm-ch3-s-eigen-degree} and \ref{thm-ch3-s-eigen-degree-bd},
we have
\begin{equation}
\label{eq4-pf-prop-ch4-traceII}
e_R\big[\Sp \big( D^{F,2,(p)}_{Z_R}\big)\cap(0,R^{\kappa-2})\big] =
\sum_{\rho\in\Lambda_R^p,\;0<|\rho|<R^{\kappa/2-1}} e_R(\rho^2) + \mathscr{O}(e^{-aR}).
\end{equation}
By Proposition \ref{prop-ch4-Lambda-star},
we have
\begin{equation}
\label{eq5-pf-prop-ch4-traceII}
e_R\big[\Sp \big( D^{2,(p)}_{I_R} \big)\cap\,(0,R^{\kappa-2})\big] =
\sum_{\lambda\in\Lambda_R^{*,p},\;0<|\lambda|<R^{\kappa/2-1}} e_R(\lambda^2).
\end{equation}
By Proposition \ref{lab-prop-Lambda-Lambda-star-estimation} in Appendix
and \eqref{eq2-pf-prop-ch4-traceII},
we have
\begin{equation}
\label{eq6-pf-prop-ch4-traceII}
\sum_{\rho\in\Lambda_R^p,\;0<|\rho|<R^{-1+\kappa/2}} e_R(\rho^2)
- \sum_{\lambda\in\Lambda_R^{*,p},\;0<|\lambda|<R^{-1+\kappa/2}} e_R(\lambda^2)
= \mathscr{O}\big(R^{\kappa-1}\big).
\end{equation}
By \eqref{eq3-pf-prop-ch4-traceII}-\eqref{eq6-pf-prop-ch4-traceII},
we get \eqref{eq-prop-ch4-traceII}.
The proof of Proposition \ref{prop-ch4-traceII} is completed.
\end{proof}

\begin{thm}
\label{thm-ch4-l-time}
We have
\begin{equation}
\sum_{j=0}^2 (-1)^{\frac{(j-1)(j-2)}{2}} \big( \theta^L_{j,R}{}'(0) - \theta^{*L}_{j,R}{}'(0) \big)
= \mathscr{O}\big(R^{\kappa-1}\big) .
\end{equation}
\end{thm}
\begin{proof}
It is a direct consequence of Propositions \ref{prop-ch4-traceI}, \ref{prop-ch4-traceII}.
\end{proof}

\begin{proof}[Proof of Theorem \ref{thm-ch0-1} : ]
It is a direct consequence of
Lemma \ref{lem-ch4-chi},
Proposition \ref{prop-ch4-zeta-star}
and Theorems \ref{thm-ch4-s-time}, \ref{thm-ch4-l-time}.
\end{proof}

\section{Asymptotics of torsions associated with the Mayer-Vietoris exact sequence}
\label{ch5}

In this section,
we study the asymptotics of the torsion
associated with the Mayer-Vietoris exact sequence.
In \textsection \ref{ch5-1},
we introduce a filtration of the Mayer-Vietoris exact sequence in question.
In \textsection \ref{ch5-2},
we estimate the $L^2$-metrics on the filtered Mayer-Vietoris exact sequence.
In \textsection \ref{ch5-3},
we estimate the maps in the quotient of the filtered Mayer-Vietoris exact sequence.
In \textsection \ref{ch5-4},
we establish the asymptotics of the torsion
associated with the Mayer-Vietoris exact sequence.
In \textsection \ref{ch5-5},
we prove Theorem \ref{thm-ch0-3}.

\subsection{A filtration of the Mayer-Vietoris exact sequence}
\label{ch5-1}

In the whole section,
we use the following identifications
\begin{equation}
H^\bullet(Z_R,F) = H^\bullet(Z,F), \quad
H^\bullet_\bd(Z_{j,R},F) = H^\bullet_\bd(Z_j,F), \quad \text{for } j=1,2,
\end{equation}
which are induced by
$\varphi_R : Z_R \rightarrow Z$ and
$\varphi_{j,R} : Z_{j,R} \rightarrow Z_j$
(see \eqref{eq-ch3-def-varphiR}
and \eqref{eq-ch3-def-varphiR-bd}).

Recall that the maps
$\mathscr{F}_{Z_R} : \hh(Z_{12,\infty},F) \rightarrow \hh(Z_R,F)$ and
$\mathscr{F}_{Z_{j,R}} : \hh_\bd(Z_{j,\infty},F) \rightarrow \hh_\bd(Z_{j,R},F)$
are defined by
\eqref{eq-ch3-def-FG-harmonic} and
\eqref{eq-ch3-def-FG-harmonic-bd}.
Let
\begin{align}
\begin{split}
\big[\mathscr{F}_{Z_R}\big] & : \;  \hh(Z_{12,\infty},F) \rightarrow H^\bullet(Z,F),\\
\big[\mathscr{F}_{Z_{j,R}}\big] & : \;  \hh_\bd(Z_{j,\infty},F) \rightarrow H^\bullet_\bd(Z_j,F)
\end{split}
\end{align}
be the compositions of
$\mathscr{F}_{Z_R}$, $\mathscr{F}_{Z_{j,R}}$
with the identification \eqref{eq-thm-ch1-Hodge-3}.
We denote
\begin{equation}
\label{eq-ch5-hhl-12}
\hhl(Z_{12,\infty},F) = \hhl(Z_{1,\infty},F) \oplus \hhl(Z_{2,\infty},F).
\end{equation}
Set
\begin{align}
\begin{split}
K^\bullet_{12} & =  \big[\mathscr{F}_{Z_R}\big] \big( \hhl(Z_{12,\infty},F) \big)
\subseteq H^\bullet(Z,F), \\
K^\bullet_j & = \big[\mathscr{F}_{Z_{j,R}}\big] \big( \hhl(Z_{j,\infty},F) \big)
\subseteq H^\bullet_{\bd}(Z_j,F), \quad
\text{for } j=1,2.
\end{split}
\end{align}
In this section,
$\dashrightarrow$ in a commutative diagram
means the unique map making the diagram commutative.
The following diagram commutes,
\begin{align}
\label{eq-ch5-K}
\begin{split}
\xymatrix{
0 \ar[r] &
\hhel^p(Z_{1,\infty},F) \ar[r]\ar[d]^{\left[\mathscr{F}_{Z_{1,R}}\right]} &
\hhel^p(Z_{12,\infty},F) \ar[r]\ar[d]^{\left[\mathscr{F}_{Z_R}\right]} &
\hhel^p(Z_{2,\infty},F) \ar[r]\ar[d]^{\left[\mathscr{F}_{Z_{2,R}}\right]} &
0 \\
0 \ar[r] &
\hspace{2mm} K^p_1 \hspace{2mm} \ar@{-->}[r] &
\hspace{2mm} K^p_{12} \hspace{2mm} \ar@{-->}[r] &
\hspace{2mm} K^p_2 \hspace{2mm} \ar[r] &
0,
}
\end{split}
\end{align}
where the maps in the first row are the canonical injection/projection.
By Propositions \ref{prop-ch3-F-ind}, \ref{prop-ch3-F-ind-bd},
$K^\bullet_1$, $K^\bullet_2$ and $K^\bullet_{12}$
together the maps between them
are independent of $R$.
Set
\begin{equation}
L_{12}^\bullet = H^\bullet(Z,F)/K^\bullet_{12}, \quad
L_{j,\bd}^\bullet = H^\bullet_{\bd}(Z_j,F)/K^\bullet_j, \quad
\text{for } j=1,2.
\end{equation}

\begin{prop}
\label{prop-ch5-filt-mv}
We have the following commutative diagram with exact rows and columns,
\begin{align}
\label{eq-ch5-filt-mv}
\begin{split}
\xymatrix{
& 0 \ar[d] & 0 \ar[d] & 0 \ar[d] & \\
\cdots \ar[r] &
K^p_1 \ar[r]\ar[d] &
K^p_{12} \ar[r]\ar[d] &
K^p_2  \ar[r]\ar[d] & \cdots \\
\cdots \ar[r] &
H^p_\mathrm{rel}(Z_1,F) \ar[r]^{\alpha_p}\ar[d] &
H^p(Z,F)\ar[r]^{\beta_p} \ar[d] &
H^p_\mathrm{abs}(Z_2,F) \ar[r]^{\hspace{5mm}\delta_p}\ar[d] &
\cdots \\
\cdots \ar@{-->}[r] &
L_{1,\mathrm{rel}}^p \ar@{-->}[r]^{\bar{\alpha}_p}\ar[d] &
L_{12}^p \ar@{-->}[r]^{\bar{\beta}_p}\ar[d] &
L_{2,\mathrm{abs}}^p \ar@{-->}[r]^{\bar{\delta}_p}\ar[d] &
\cdots \\
& 0 & 0 & 0  & ,
}
\end{split}
\end{align}
where
the vertical maps are canonical injection/projection,
the first row is given by \eqref{eq-ch5-K},
the second row is the Mayer-Vietoris exact sequence.
\end{prop}
\begin{proof}
We show that the upper left square commutes.
It is equivalent to show that
\begin{equation}
\label{eq1-pf-prop-ch5-filt-mv}
\alpha_p \left( \big[\mathscr{F}_{Z_{1,R}}(\omega,0)\big] \right) =
\big[\mathscr{F}_{Z_R}(\omega,0,0)\big] \in H^p(Z,F), \quad
\text{for } \omega\in\hhel^p(Z_{1,\infty},F).
\end{equation}
By \eqref{eq-ch3-def-FG} and \eqref{eq-ch3-def-FG-bd},
we have
\begin{equation}
\label{eq2-pf-prop-ch5-filt-mv}
F_{Z_R}(\omega,0,0) \big|_{Z_{1,R}} = F_{Z_{1,R}}(\omega,0), \quad
F_{Z_R}(\omega,0,0) \big|_{Z_{2,R}} = 0.
\end{equation}
By Proposition \ref{prop-ch1-mv}
and \eqref{eq2-pf-prop-ch5-filt-mv},
we have
\begin{equation}
\label{eq3-pf-prop-ch5-filt-mv}
\alpha_p \left( \big[ F_{Z_{1,R}}(\omega,0) \big] \right)
= \big[ F_{Z_R}(\omega,0,0) \big] \in H^p(Z,F).
\end{equation}
By \eqref{eq-ch3-def-FG-harmonic}
and \eqref{eq-ch3-def-FG-harmonic-bd},
we have
\begin{align}
\label{eq4-pf-prop-ch5-filt-mv}
\begin{split}
\big[ \mathscr{F}_{Z_{1,R}}(\omega,0) \big]
& = \big[ F_{Z_{1,R}}(\omega,0) \big] \in H^p_\mathrm{rel}(Z_1,F),\\
\big[ \mathscr{F}_{Z_R}(\omega,0,0) \big]
& = \big[ F_{Z_R}(\omega,0,0) \big] \in H^p(Z,F).
\end{split}
\end{align}
Then \eqref{eq1-pf-prop-ch5-filt-mv} follows from
\eqref{eq3-pf-prop-ch5-filt-mv} and
\eqref{eq4-pf-prop-ch5-filt-mv}.

Proceeding in the same way,
we can show that the upper right square commutes.

Now we show that $\delta_p\big(K^p_2\big)=0$.
It is equivalent to show that
\begin{equation}
\label{eq21-pf-prop-ch5-filt-mv}
\delta_p \left( \big[\mathscr{F}_{Z_{2,R}}(\omega,0)\big] \right) =
0 \in H^{p+1}_\mathrm{rel}(Z_1,F), \quad
\text{for } \omega\in\hhel^p(Z_{2,\infty},F).
\end{equation}
By \eqref{eq-ch3-def-FG-bd},
we have
\begin{equation}
\label{eq22-pf-prop-ch5-filt-mv}
F_{Z_{2,R}}(\omega,0) \big|_{Z_{1,R}} = 0.
\end{equation}
By Proposition \ref{prop-ch1-mv}
and \eqref{eq22-pf-prop-ch5-filt-mv},
we have
\begin{equation}
\label{eq23-pf-prop-ch5-filt-mv}
\delta_p \left( \big[F_{Z_{2,R}}(\omega,0)\big] \right) =
0 \in H^{p+1}_\mathrm{rel}(Z_1,F).
\end{equation}
By \eqref{eq-ch3-def-FG-harmonic-bd},
we have
\begin{equation}
\label{eq24-pf-prop-ch5-filt-mv}
\big[\mathscr{F}_{Z_{2,R}}(\omega,0)\big] =
\big[F_{Z_{2,R}}(\omega,0)\big] \in H^p_\mathrm{abs}(Z_2,F).
\end{equation}
Then \eqref{eq21-pf-prop-ch5-filt-mv} follows from
\eqref{eq23-pf-prop-ch5-filt-mv} and \eqref{eq24-pf-prop-ch5-filt-mv}.

The rest follows from a diagram chasing argument.
This completes the proof of Proposition \ref{prop-ch5-filt-mv}.
\end{proof}

Let $\LL^\bullet_j$ ($j=1,2$)
be the set of limiting values (see \eqref{eq-ch2-def-LL})
of $\hh(Z_{j,\infty},F)$.
Let $\LL^\bullet_{j,\mathrm{abs/rel}}\subseteq\LL^\bullet_j$
be the absolute/relative component (see \eqref{eq-ch2-def-LL-bd})
of $\LL^\bullet_j$.
We continue to use the convention
$\LL^\bullet_{1,\mathrm{bd}}=\LL^\bullet_{1,\mathrm{rel}}$ and
$\LL^\bullet_{2,\mathrm{bd}}=\LL^\bullet_{2,\mathrm{abs}}$
We have the following commutative diagram with exact rows,
\begin{align}
\label{eq-ch5-filt-H-bd}
\begin{split}
\xymatrix{
0 \ar[r] &
\hhel^p(Z_{j,\infty},F) \ar[r]\ar[d]^{\left[\mathscr{F}_{Z_{j,R}}\right]} &
\hhe^p_\mathrm{bd}(Z_{j,\infty},F) \ar[r]\ar[d]^{\left[\mathscr{F}_{Z_{j,R}}\right]} &
\LL_{j,\mathrm{bd}}^p \ar[r]\ar@{-->}[d] & 0 \\
0 \ar[r] & K^p_j \ar[r] & H^p_\mathrm{bd}(Z_j,F) \ar[r] & L^p_{j,\mathrm{bd}} \ar[r] & 0,
}
\end{split}
\end{align}
where the first row is defined by \eqref{eq-ch2-ses-hh-bd},
the second row consists of canonical injection/projection.
We also have the following commutative diagram with exact rows,
\begin{align}
\label{eq-ch5-filt-H}
\begin{split}
\xymatrix{
0 \ar[r] &
\hhel^p(Z_{12,\infty},F) \ar[r]\ar[d]^{\left[\mathscr{F}_{Z_R}\right]} &
\hhe^p(Z_{12,\infty},F) \ar[r]\ar[d]^{\left[\mathscr{F}_{Z_R}\right]} &
\LL_1^p \cap \LL_2^p \ar[r]\ar@{-->}[d] & 0 \\
0 \ar[r] & K^p_{12} \ar[r] & H^p(Z,F) \ar[r] & L^p_{12} \ar[r] & 0,
}
\end{split}
\end{align}
where the first row is defined by \eqref{eq-ch3-filt-hh-12}
and \eqref{eq-ch5-hhl-12},
the second row consists of canonical injection/projection.
Finally we get the following commutative diagram with exact rows and columns,
\begin{align}
\label{eq-ch5-filt-mv-harmonic}
\begin{split}
\xymatrix{
& 0 \ar[d] & 0 \ar[d] & 0 \ar[d] & \\
\cdots \ar[r] &
\hhel^p(Z_{1,\infty},F) \ar[r]\ar[d] &
\hhel^p(Z_{12,\infty},F) \ar[r]\ar[d] &
\hhel^p(Z_{2,\infty},F)  \ar[r]\ar[d] &
\cdots \\
\cdots \ar[r] &
\hhe^p_\mathrm{rel}(Z_{1,\infty},F) \ar[r]^{\alpha_p(R)}\ar[d] &
\hhe^p(Z_{12,\infty},F) \ar[r]^{\beta_p(R)}\ar[d] &
\hhe^p_\mathrm{abs}(Z_{2,\infty},F) \ar[r]^{\hspace{10mm}\delta_p(R)}\ar[d] &
\cdots \\
\cdots \ar@{-->}[r] &
\LL_{1,\mathrm{rel}}^p \ar@{-->}[r]^{\bar{\alpha}_p(R)}\ar[d] &
\LL_1^p \cap \LL_2^p \ar@{-->}[r]^{\bar{\beta}_p(R)}\ar[d] &
\LL_{2,\mathrm{abs}}^p \ar@{-->}[r]^{\hspace{5mm}\bar{\delta}_p(R)}\ar[d] &
\cdots \\
& 0 & 0 & 0  &
}
\end{split}
\end{align}
where the columns are defined by \eqref{eq-ch5-filt-H-bd} and \eqref{eq-ch5-filt-H},
the first row consists of canonical injection/projection,
the maps in the second row is given by
\begin{align}
\label{eq-ch5-def-alpha-beta-delta-R}
\begin{split}
& \alpha_p(R) = \big[\mathscr{F}_{Z_R}\big]^{-1} \circ \alpha_p \circ \big[\mathscr{F}_{Z_{1,R}}\big], \quad
\beta_p(R) = \big[\mathscr{F}_{Z_{2,R}}\big]^{-1} \circ \beta_p \circ \big[\mathscr{F}_{Z_R}\big], \\
& \delta_p(R) = \big[\mathscr{F}_{Z_{1,R}}\big]^{-1} \circ \delta_p \circ \big[\mathscr{F}_{Z_{2,R}}\big].
\end{split}
\end{align}

\subsection{Asymptotics of $L^2$-metrics}
\label{ch5-2}

Recall that $\big\lVert\,\cdot\,\big\rVert_{Z_R}$
(resp. $\big\lVert\,\cdot\,\big\rVert_{Z_{j,R}}$ with $j=1,2$)
is the $L^2$-metric on $\Omega^\bullet(Z_R,F)$
(resp. $\Omega^\bullet(Z_{j,R},F)$).
Recall that the metrics
$\big\lVert \,\cdot\, \big\rVert_{\hh(Z_{12,\infty},F),R}$ and
$\big\lVert \,\cdot\, \big\rVert_{\hh_\mathrm{bd}(Z_{j,\infty},F),R}$
are defined by \eqref{eq-ch3-def-metric-hh} and \eqref{eq-ch3-def-metric-hh-bd}.

\begin{prop}
\label{prop-ch5-metric-H}
There exists $a>0$ such that
\begin{align}
\begin{split}
\big\lVert \mathscr{F}_{Z_R}(\cdot)  \big\rVert_{Z_R}
& = \Big( 1 + \mathscr{O}\big(e^{-aR}\big) \Big)
\big\lVert \,\cdot\, \big\rVert_{\hh(Z_{12,\infty},F),R}, \\
\big\lVert \mathscr{F}_{Z_{j,R}}(\cdot)  \big\rVert_{Z_{j,R}}
& = \Big( 1 + \mathscr{O}\big(e^{-aR}\big) \Big)
\big\lVert \,\cdot\, \big\rVert_{\hh_\mathrm{bd}(Z_{j,\infty},F),R},\quad\text{for } j=1,2.
\end{split}
\end{align}
\end{prop}
\begin{proof}
They are direct consequences of
Propositions
\ref{prop-ch3-approx-f-g-total},
\ref{prop-ch3-approx-harmonic-tot-inj},
\ref{prop-ch3-approx-f-g-bd},
\ref{prop-ch3-approx-harmonic-bd-inj}.
\end{proof}

Let $\big\lVert\,\cdot\,\big\rVert_{\LL_1^\bullet \cap \LL_2^\bullet, R}$
be the quotient metric on $\LL_1^\bullet \cap \LL_2^\bullet$
induced by $\big\lVert\,\cdot\,\big\rVert_{\hh(Z_{12,\infty},F),R}$
via the map $\hh(Z_{12,\infty},F)\rightarrow\LL_1^\bullet \cap \LL_2^\bullet$
in \eqref{eq-ch5-filt-H}.
Let $\big\lVert\,\cdot\,\big\rVert_{\LL_1^\bullet \cap \LL_2^\bullet}$
be the metric on $\LL_1^\bullet \cap \LL_2^\bullet$
induced by the $L^2$-metric on $\hh(Y,F[du])$
via the inclusion $\LL_1^\bullet \cap \LL_2^\bullet\subseteq\hh(Y,F[du])$.
In the same way,
we construct metrics
$\big\lVert\,\cdot\,\big\rVert_{\LL_{j,\bd}^\bullet,R}$ and
$\big\lVert\,\cdot\,\big\rVert_{\LL_{j,\bd}^\bullet}$
on $\LL_{j,\bd}^\bullet$.

\begin{prop}
\label{prop-ch5-metric-L}
We have
\begin{align}
\label{eq-prop-ch5-metric-L}
\begin{split}
\big\lVert\,\cdot\,\rVert^2_{\LL_1^\bullet \cap \LL_2^\bullet, R}
& = \Big( 2R + \mathscr{O}(1) \Big)
\big\lVert\,\cdot\,\big\rVert^2_{\LL_1^\bullet \cap \LL_2^\bullet}, \\
\big\lVert\,\cdot\,\big\rVert^2_{\LL_{j,\bd}^\bullet, R}
& = \Big( R + \mathscr{O}(1) \Big)
\big\lVert\,\cdot\,\big\rVert^2_{\LL_{j,\bd}^\bullet},
\quad \text{for } j=1,2.
\end{split}
\end{align}
\end{prop}
\begin{proof}
We only prove the second identity with $j=2$.

By the definition of quotient metric,
for $\hat{\omega}\in\LL_{2,\mathrm{abs}}^\bullet$,
we have
\begin{equation}
\label{eq1-pf-prop-ch5-metric-L}
\big\lVert \hat{\omega} \big\rVert^2_{\LL_{2,\mathrm{abs}}^\bullet, R}
= \inf_{(\omega,\hat{\omega})\in\hh_\mathrm{abs}(Z_{2,\infty},F)}
\big\lVert (\omega,\hat{\omega}) \big\rVert^2_{\hh_\mathrm{abs}(Z_{2,\infty},F),R}.
\end{equation}
For $(\omega,\hat{\omega})\in\hh_\mathrm{abs}(Z_{2,\infty},F)$,
by \eqref{eq-ch2-decomp-zm-nz},
we have the decomposition
$\omega\big|_{Y_{(-\infty,0]}} =
\omega^\mathrm{zm}+\omega^\mathrm{nz}$.
Since $\pi^*_Y\hat{\omega} = \omega^\mathrm{zm}$,
we have
\begin{align}
\label{eq2-pf-prop-ch5-metric-L}
\begin{split}
\big\lVert(\omega,\hat{\omega})\big\rVert^2_{\hh_\mathrm{abs}(Z_{2,\infty},F),R}
= \big\lVert\omega\big\rVert^2_{Z_{2,R}}
& = \big\lVert\omega \big\rVert^2_{Z_{2,0}} + \big\lVert\omega^\mathrm{nz}\big\rVert^2_{Y_{[-R,0]}}
+ \big\lVert\omega^\mathrm{zm}\big\rVert^2_{Y_{[-R,0]}}\\
& = \big\lVert\omega \big\rVert^2_{Z_{2,0}} + \big\lVert\omega^\mathrm{nz}\big\rVert^2_{Y_{[-R,0]}}
+ R \big\lVert\hat{\omega}\big\rVert^2_{\LL_{2,\mathrm{abs}}^\bullet}.
\end{split}
\end{align}
By \eqref{eq1-pf-prop-ch5-metric-L}
and \eqref{eq2-pf-prop-ch5-metric-L},
it remains to show that
for $\hat{\omega}\in\LL_{2,\bd}^\bullet$,
there exists $(\omega,\hat{\omega})\in\hh_\bd(Z_{2,\infty},F)$ such that
\begin{equation}
\label{eq3-pf-prop-ch5-metric-L}
\big\lVert\omega\big\rVert^2_{Z_{2,0}} + \big\lVert\omega^\mathrm{nz}\big\rVert^2_{Y_{[-R,0]}}
= \mathscr{O}(1)
\big\lVert\hat{\omega}\big\rVert^2_{\LL_{2,\mathrm{abs}}^\bullet}.
\end{equation}
For $\hat{\omega}\in\LL_{2,\mathrm{abs}}^\bullet$,
we choose $(\omega,\hat{\omega})\in\hh_\mathrm{abs}(Z_{2,\infty},F)$
such that $\omega$ is a generalized eigensection of $D^F_{Z_{2,\infty},\mathrm{ac}}$ associated with  $0$.
The existence of such a $\omega$ is guaranteed by \eqref{eq-ch2-decomp-hh}.
By \eqref{eq-ch2-nz-decreasing},
the last part of Proposition \ref{prop-ch2-g-eigen-Dinf},
i.e., $\hh(Y,F)\ni \phi\mapsto E(\phi,0)\in\hh(Z_{2,\infty},F)$ is injective,
and the fact that $\dim \hh(Y,F)$ is finite,
we get \eqref{eq3-pf-prop-ch5-metric-L}.
The proof of Proposition \ref{prop-ch5-metric-L} is completed.
\end{proof}


\subsection{Asymptotics of horizontal maps}
\label{ch5-3}

For ease of notation,
we denote by $\big\lVert\,\cdot\,\big\rVert_{\mathscr{H},R}$
one of the following metrics:
$\big\lVert\,\cdot\,\big\rVert_{\hh_\mathrm{rel}(Z_{1,\infty},F),R}$,
$\big\lVert\,\cdot\,\big\rVert_{\hh(Z_{12,\infty},F),R}$,
$\big\lVert\,\cdot\,\big\rVert_{\hh_\mathrm{abs}(Z_{2,\infty},F),R}$.
We denote by $\big\langle\cdot,\cdot\big\rangle_{\mathscr{H},R}$
the scalar product associated with $\big\lVert\,\cdot\,\big\rVert_{\mathscr{H},R}$.

\begin{prop}
\label{prop-ch5-map-H}
There exists $a>0$
such that for
$(\omega,\hat{\omega})\in\hhe^p_\mathrm{rel}(Z_{1,\infty},F)$,
$(\mu_1,\mu_2,\hat{\mu})\in\hhe^p(Z_{12,\infty},F)$,
$(\tau,\hat{\tau})\in\hhe^p_\mathrm{abs}(Z_{2,\infty},F)$ and
$(\sigma,\hat{\sigma})\in\hhe^{p+1}_\mathrm{rel}(Z_{1,\infty},F)$,
we have
\begin{align}
\label{eq-prop-ch5-map-H}
\begin{split}
& \big\langle \alpha_p(R)(\omega,\hat{\omega}),(\mu_1,\mu_2,\hat{\mu}) \big\rangle_{\mathscr{H},R} \\
& = \big\langle\omega,\mu_1\big\rangle_{Z_{1,R}}
+ \mathscr{O}\big(e^{-aR}\big)
\big\lVert(\omega,\hat{\omega})\big\rVert_{\mathscr{H},R}
\big\lVert(\mu_1,\mu_2,\hat{\mu})\big\rVert_{\mathscr{H},R}, \\
& \big\langle \beta_p(R)(\mu_1,\mu_2,\hat{\mu}),(\tau,\hat{\tau}) \big\rangle_{\mathscr{H},R} \\
& = \big\langle\mu_2,\tau\big\rangle_{Z_{2,R}}
+ \mathscr{O}\big(e^{-aR}\big)
\big\lVert(\mu_1,\mu_2,\hat{\mu})\big\rVert_{\mathscr{H},R}
\big\lVert(\tau,\hat{\tau})\big\rVert_{\mathscr{H},R}, \\
& \big\langle \delta_p(R)(\tau,\hat{\tau}),(\sigma,\hat{\sigma}) \big\rangle_{\mathscr{H},R} \\
& = \big\langle \hat{\tau},\cntrtu\hat{\sigma} \big\rangle_Y
+ \mathscr{O}\big(e^{-aR}\big)
\big\lVert(\tau,\hat{\tau})\big\rVert_{\mathscr{H},R}
\big\lVert(\sigma,\hat{\sigma})\big\rVert_{\mathscr{H},R}.
\end{split}
\end{align}
\end{prop}
\begin{proof}
We only prove the first identity in \eqref{eq-prop-ch5-map-H}.

We denote
\begin{equation}
\label{eq0-pf-prop-ch5-map-H}
\alpha_p(R)(\omega,\hat{\omega}) = (\eta_1,\eta_2,\hat{\eta})
\in \hhe^p(Z_{12,\infty},F).
\end{equation}
By \eqref{eq-ch5-def-alpha-beta-delta-R}
and \eqref{eq0-pf-prop-ch5-map-H},
we have
\begin{equation}
\label{eq1-pf-prop-ch5-map-H}
\alpha_p\left(\big[\mathscr{F}_{Z_{1,R}}(\omega,\hat{\omega})\big]\right) =
\big[\mathscr{F}_{Z_R}(\eta_1,\eta_2,\hat{\eta})\big] \in H^p(Z,F).
\end{equation}
By Proposition \ref{prop-ch1-mv-harmonic}
and \eqref{eq1-pf-prop-ch5-map-H},
we have
\begin{equation}
\label{eq2-pf-prop-ch5-map-H}
\big\langle \mathscr{F}_{Z_R}(\eta_1,\eta_2,\hat{\eta}),\mathscr{F}_{Z_R}(\mu_1,\mu_2,\hat{\mu}) \big\rangle_{Z_R} =
\big\langle \mathscr{F}_{Z_{1,R}}(\omega,\hat{\omega}),\mathscr{F}_{Z_R}(\mu_1,\mu_2,\hat{\mu}) \big\rangle_{Z_{1,R}}.
\end{equation}
Taking $(\mu_1,\mu_2,\hat{\mu})=(\eta_1,\eta_2,\hat{\eta})$
in \eqref{eq2-pf-prop-ch5-map-H}
and applying Cauchy-Schwarz inequality,
we get
\begin{equation}
\label{eq3-pf-prop-ch5-map-H}
\big\lVert \mathscr{F}_{Z_R}(\eta_1,\eta_2,\hat{\eta}) \big\rVert_{Z_R} \leqslant
\big\lVert \mathscr{F}_{Z_{1,R}}(\omega,\hat{\omega}) \big\rVert_{Z_{1,R}}.
\end{equation}
By Propositions
\ref{prop-ch3-approx-f-g-total},
\ref{prop-ch3-approx-harmonic-tot-inj},
\ref{prop-ch3-approx-f-g-bd},
\ref{prop-ch3-approx-harmonic-bd-inj}
and \eqref{eq3-pf-prop-ch5-map-H},
we get
\begin{equation}
\label{eq4-pf-prop-ch5-map-H}
\big\lVert (\eta_1,\eta_2,\hat{\eta}) \big\rVert_{\mathscr{H},R}
= \mathscr{O}(1)
\big\lVert (\omega,\hat{\omega}) \big\rVert_{\mathscr{H},R}.
\end{equation}
By Proposition \ref{prop-ch5-metric-H},
we have
\begin{align}
\label{eq5-pf-prop-ch5-map-H}
\begin{split}
& \big\langle \mathscr{F}_{Z_R}(\eta_1,\eta_2,\hat{\eta}),\mathscr{F}_{Z_R}(\mu_1,\mu_2,\hat{\mu}) \big\rangle_{Z_R}  \\
& = \big\langle(\eta_1,\eta_2,\hat{\eta}),(\mu_1,\mu_2,\hat{\mu})\big\rangle_{\mathscr{H},R}
+ \mathscr{O}\big(e^{-aR}\big)
\big\lVert(\eta_1,\eta_2,\hat{\eta})\big\rVert_{\mathscr{H},R}
\big\lVert(\mu_1,\mu_2,\hat{\mu})\big\rVert_{\mathscr{H},R},
\end{split}
\end{align}
By Propositions
\ref{prop-ch3-approx-f-g-total},
\ref{prop-ch3-approx-harmonic-tot-inj},
\ref{prop-ch3-approx-f-g-bd},
\ref{prop-ch3-approx-harmonic-bd-inj},
we have
\begin{align}
\label{eq6-pf-prop-ch5-map-H}
\begin{split}
& \big\langle\mathscr{F}_{Z_{1,R}}(\omega,\hat{\omega}),\mathscr{F}_{Z_R}(\mu_1,\mu_2,\hat{\mu})\big\rangle_{Z_{1,R}} \\
& = \big\langle\omega,\mu_1\big\rangle_{Z_{1,R}}
+ \mathscr{O}\big(e^{-aR}\big)
\big\lVert(\omega,\hat{\omega})\big\rVert_{\mathscr{H},R}
\big\lVert(\mu_1,\mu_2,\hat{\mu})\big\rVert_{\mathscr{H},R}.
\end{split}
\end{align}
By \eqref{eq0-pf-prop-ch5-map-H},
\eqref{eq2-pf-prop-ch5-map-H} and
\eqref{eq4-pf-prop-ch5-map-H}-\eqref{eq6-pf-prop-ch5-map-H},
we get the first identity in \eqref{eq-prop-ch5-map-H}.
The proof of Proposition \ref{prop-ch5-map-H} is completed.
\end{proof}

Now we compare the third row in \eqref{eq-ch5-filt-mv-harmonic}
with \eqref{eq-ch4-mv-LL}.

\begin{prop}
\label{prop-ch5-map-L}
We have
\begin{align}
\label{eq-prop-ch5-map-L}
\begin{split}
& \bar{\alpha}_p(R) = \frac{1}{2}\alpha_{p,\LL} + \mathscr{O}\big(R^{-1/2}\big), \hspace{10mm}
\bar{\beta}_p(R) =  \beta_{p,\LL} + \mathscr{O}\big(R^{-1/2}\big), \\
& \bar{\delta}_p(R) = R^{-1}\delta_{p,\LL} + \mathscr{O}\big(R^{-3/2}\big).
\end{split}
\end{align}
\end{prop}
\begin{proof}
We only prove the first identity in \eqref{eq-prop-ch5-map-L}.

Let $\hat{\omega}\in\LL_{1,\mathrm{rel}}^p$.
By \eqref{eq-ch2-decomp-hh},
there exists $(\omega,\hat{\omega}) \in \hhe^p_\mathrm{rel}(Z_{1,\infty},F)$
such that $\omega$ is a generalized eigensection.
We will use the same notation as in \eqref{eq0-pf-prop-ch5-map-H}.
By \eqref{eq-ch5-filt-mv-harmonic}
and \eqref{eq0-pf-prop-ch5-map-H},
we have
\begin{equation}
\label{eq2-pf-prop-ch5-map-L}
\bar{\alpha}_p(R)(\hat{\omega}) = \hat{\eta}.
\end{equation}
By Proposition \ref{prop-ch5-metric-L},
the first identity in \eqref{eq-prop-ch5-map-L} is equivalent to
\begin{equation}
\label{eq3-pf-prop-ch5-map-L}
\big\lVert \hat{\eta} - \frac{1}{2}\alpha_{p,\LL}(\hat{\omega}) \big\rVert^2_{\LL_1^p \cap \LL_2^p,R}
= \mathscr{O}\big(R^{-1}\big)
\big\lVert \hat{\omega} \big\rVert^2_Y .
\end{equation}
By \eqref{eq-ch2-decomp-hh},
there exists $(\eta_1',\eta_2',\hat{\eta}')\in\hhe^p(Z_{12,\infty},F)$
such that $\eta_1'$ and $\eta_2'$ are generalized eigensections and
\begin{equation}
\label{eq4-pf-prop-ch5-map-L}
\hat{\eta}' = \frac{1}{2}\alpha_{p,\LL}(\hat{\omega}).
\end{equation}
Since $\big\lVert\,\cdot\,\big\rVert_{\LL_1^p \cap \LL_2^p,R}$
is the quotient metric of $\big\lVert\,\cdot\,\big\rVert_{\mathscr{H},R}$,
it is sufficient to show that
\begin{equation}
\label{eq5-pf-prop-ch5-map-L}
\big\lVert (\eta_1,\eta_2,\hat{\eta}) - (\eta_1',\eta_2',\hat{\eta}') \big\rVert^2_{\mathscr{H},R}
= \mathscr{O}\big(R^{-1}\big)
\big\lVert \hat{\omega} \big\rVert^2_Y.
\end{equation}
By Riesz representation theorem,
it is sufficient to show that
for $(\mu_1,\mu_2,\hat{\mu})\in\hhe^p(Z_{12,\infty},F)$,
\begin{equation}
\label{eq6-pf-prop-ch5-map-L}
\big\langle (\eta_1,\eta_2,\hat{\eta}) - (\eta_1',\eta_2',\hat{\eta}'),(\mu_1,\mu_2,\hat{\mu}) \big\rangle_{\mathscr{H},R}
= \mathscr{O}\big(R^{-1/2}\big)
\big\lVert \hat{\omega} \big\rVert_Y
\big\lVert (\mu_1,\mu_2,\hat{\mu}) \big\rVert_{\mathscr{H},R}.
\end{equation}

Let
\begin{equation}
\label{eq11-pf-prop-ch5-map-L}
\mu_1 = \mu_1^\mathrm{pp} + \mu_1^\mathrm{ac}
\end{equation}
such that $\mu_1^\mathrm{pp}$
(resp. $\mu_1^\mathrm{ac}$)
is a $L^2$-eigensection
(resp. generalized eigensection)
of $D^F_{Z_{1,\infty}}$.
Since $\omega$ and $\mu_1^\mathrm{ac}$ are both generalized eigensections,
using \eqref{eq-ch2-nz-decreasing}
in the same way as in the proof of Proposition \ref{prop-ch5-metric-L}
we get from
\eqref{eq-ch2-def-zm-nz-1} and \eqref{eq-ch3-def-metric-hh}
that
\begin{align}
\label{eq12-pf-prop-ch5-map-L}
\begin{split}
\big\langle\omega,\mu_1^\mathrm{ac}\big\rangle_{Z_{1,R}}
& = R \big\langle \hat{\omega},\hat{\mu} \big\rangle_Y
+ \mathscr{O}(1)
\big\lVert\hat{\omega}\big\rVert_Y \big\lVert\hat{\mu}\big\rVert_Y \\
& = R \big\langle \hat{\omega},\hat{\mu} \big\rangle_Y
+ \mathscr{O}\big(R^{-1/2}\big)
\big\lVert\hat{\omega}\big\rVert_Y \big\lVert(\mu_1,\mu_2,\hat{\mu})\big\rVert_{\mathscr{H},R}.
\end{split}
\end{align}
Note that $\omega$ is a generalized eigensection
while $\mu_1^\mathrm{pp}$ is a $L^2$-eigensection,
we have
\begin{equation}
\label{eq13-pf-prop-ch5-map-L}
0 = \big\langle\omega,\mu_1^\mathrm{pp}\big\rangle_{Z_{1,\infty}}
= \big\langle\omega,\mu_1^\mathrm{pp}\big\rangle_{Z_{1,R}}
+ \big\langle\omega,\mu_1^\mathrm{pp}\big\rangle_{Y_{[R,\infty)}}.
\end{equation}
Since $\big(\mu_1^\mathrm{pp}\big)^\mathrm{zm}=0$,
by \eqref{eq-ch2-nz-decreasing}
and \eqref{eq5-pf-prop-ch3-approx-f-g-total} ,
we have
\begin{align}
\label{eq14-pf-prop-ch5-map-L}
\begin{split}
& \big\langle\omega,\mu_1^\mathrm{pp}\big\rangle_{Z_{1,R}}
= - \big\langle\omega,\mu_1^\mathrm{pp}\big\rangle_{Y_{[R,+\infty)}}
= - \big\langle\omega^\mathrm{nz},\big(\mu_1^\mathrm{pp}\big)^\mathrm{nz}\big\rangle_{Y_{[R,+\infty)}} \\
& = \mathscr{O}\big(e^{-aR}\big)
\big\lVert\omega\big\rVert_{\partial Z_{1,0}} \big\lVert\mu_1\big\rVert_{\partial Z_{1,0}}
= \mathscr{O}\big(e^{-aR}\big)
\big\lVert\hat{\omega}\big\rVert_Y \big\lVert(\mu_1,\mu_2,\hat{\mu})\big\rVert_{\mathscr{H},R}.
\end{split}
\end{align}
By \eqref{eq12-pf-prop-ch5-map-L}
and \eqref{eq14-pf-prop-ch5-map-L},
we have
\begin{equation}
\label{eq15-pf-prop-ch5-map-L}
\big\langle\omega,\mu_1\big\rangle_{Z_{1,R}} =
R \big\langle \hat{\omega},\hat{\mu} \big\rangle_Y
+ \mathscr{O}\big(R^{-1/2}\big)
\big\lVert\hat{\omega}\big\rVert_Y \big\lVert(\mu_1,\mu_2,\hat{\mu})\big\rVert_{\mathscr{H},R}.
\end{equation}
The same argument also yields
\begin{align}
\label{eq16-pf-prop-ch5-map-L}
\begin{split}
\big\langle\eta_j',\mu_j\big\rangle_{Z_{j,R}} =
R \big\langle \hat{\eta}',\hat{\mu} \big\rangle_Y
+ \mathscr{O}\big(R^{-1/2}\big)
\big\lVert\hat{\eta}'\big\rVert_Y \big\lVert(\mu_1,\mu_2,\hat{\mu})\big\rVert_{\mathscr{H},R},
\quad \text{for } j=1,2.
\end{split}
\end{align}
By the definition of $\alpha_{p,\LL}$
(see \eqref{eq-ch4-def-alpha}),
\eqref{eq4-pf-prop-ch5-map-L},
\eqref{eq15-pf-prop-ch5-map-L}
and \eqref{eq16-pf-prop-ch5-map-L},
we get
\begin{equation}
\label{eq17-pf-prop-ch5-map-L}
\big\langle\omega,\mu_1\big\rangle_{Z_{1,R}}
- \big\langle\eta_1',\mu_1\big\rangle_{Z_{1,R}}
- \big\langle\eta_2',\mu_2\big\rangle_{Z_{2,R}}
= \mathscr{O}\big(R^{-1/2}\big)
\big\lVert\hat{\omega}\big\rVert_Y \big\lVert(\mu_1,\mu_2,\hat{\mu})\big\rVert_{\mathscr{H},R}.
\end{equation}

The same argument as in
\eqref{eq2-pf-prop-ch5-metric-L} and \eqref{eq3-pf-prop-ch5-metric-L}
yields
\begin{equation}
\big\lVert(\omega,\hat{\omega})\big\rVert^2_{\mathscr{H},R}
\leqslant R \big\lVert\hat{\omega}\big\rVert^2_Y
+ \mathscr{O}\big(1\big) \big\lVert\hat{\omega}\big\rVert^2_Y.
\end{equation}
Thus, by
Proposition \ref{prop-ch5-map-H}
and \eqref{eq0-pf-prop-ch5-map-H},
we have
\begin{align}
\label{eq21-pf-prop-ch5-map-L}
\begin{split}
\big\langle (\eta_1,\eta_2,\hat{\eta}),(\mu_1,\mu_2,\hat{\mu}) \big\rangle_{\mathscr{H},R}
& = \big\langle\omega,\mu_1\big\rangle_{Z_{1,R}}
+ \mathscr{O}\big(e^{-aR}\big)
\big\lVert(\omega,\hat{\omega})\big\rVert_{\mathscr{H},R}
\big\lVert(\mu_1,\mu_2,\hat{\mu})\big\rVert_{\mathscr{H},R} \\
& = \big\langle\omega,\mu_1\big\rangle_{Z_{1,R}}
+ \mathscr{O}\big(Re^{-aR}\big)
\big\lVert\hat{\omega}\big\rVert_Y \big\lVert(\mu_1,\mu_2,\hat{\mu})\big\rVert_{\mathscr{H},R}.
\end{split}
\end{align}
The following identity follows from the definition of
$\big\langle\cdot,\cdot\big\rangle_{\mathscr{H},R}$,
\begin{equation}
\label{eq22-pf-prop-ch5-map-L}
\big\langle (\eta_1',\eta_2',\hat{\eta}'),(\mu_1,\mu_2,\hat{\mu}) \big\rangle_{\mathscr{H},R}
= \big\langle\eta_1',\mu_1\big\rangle_{Z_{1,R}} + \big\langle\eta_2',\mu_2\big\rangle_{Z_{2,R}}.
\end{equation}
By \eqref{eq17-pf-prop-ch5-map-L}-\eqref{eq22-pf-prop-ch5-map-L},
we obtain \eqref{eq6-pf-prop-ch5-map-L}.
The proof of Proposition \ref{prop-ch5-map-L} is completed.
\end{proof}

\begin{rem}
A special case of the problem addressed in this subsection
was considered by M{\"u}ller and Strohmaier \cite[Theorem 3.3]{MuStro},
where the Mayer-Vietoris exact sequence in question is given by
\begin{equation}
\cdots \rightarrow
H^p_\mathrm{rel}(Z_1,\C) \rightarrow
H^p_\mathrm{abs}(Z_1,\C) \rightarrow
H^p(Y,\C) \rightarrow
\cdots.
\end{equation}
\end{rem}


\subsection{A proof of Theorem \ref{thm-ch0-2}}
\label{ch5-4}

First we briefly recall some properties of the torsion
associated with an acyclic complex
(cf. \cite[\textsection 1a]{BGS1}).
We denote by $\mathscr{T}(V^\bullet,\partial)$ the torsion
associated with a finite acyclic complex $(V^\bullet,\partial)$ of Hermitian vector spaces.
\begin{itemize}
\item[-] Let $(V^\bullet[r],\partial)$ be the $r$-th right-shift of $(V^\bullet,\partial)$,
i.e., $V^k[r] = V^{k-r}$,
then
\begin{equation}
\label{eq-ch5-t1}
\mathscr{T}(V^\bullet[r],\partial) = \big( \mathscr{T}(V^\bullet,\partial) \big)^{(-1)^r}.
\end{equation}
\item[-] If $(V^\bullet,\partial)$ is the direct sum
of two complexes $(V^\bullet_1,\partial_1)$ and $(V^\bullet_2,\partial_2)$,
then
\begin{equation}
\label{eq-ch5-t2}
\mathscr{T}(V^\bullet,\partial) = \mathscr{T}(V^\bullet_1,\partial_1) \mathscr{T}(V^\bullet_2,\partial_2).
\end{equation}
\item[-] For a short acyclic complex
$(V^\bullet,\partial): 0 \rightarrow V^1 \rightarrow V^2 \rightarrow 0$,
let $A$ be the matrix of $\partial : V^1 \rightarrow V^2$ with respect to any orthogonal bases,
then
\begin{equation}
\label{eq-ch5-t3}
\mathscr{T}(V^\bullet,\partial) = \big| \det(A) \big|.
\end{equation}
\end{itemize}

Let $\LL_{j,\mathrm{abs}}^{\bullet,\perp} \subseteq \hh(Y,F)$
be the orthogonal complement of $\LL_{j,\mathrm{abs}}^{\bullet}$
with respect to the $L^2$-metric induced by Hodge theory.
We construct $S_j^p\in\mathrm{End}\big(\hhe^p(Y,F)\big)$ as follows,
\begin{equation}
\label{eq-ch5-def-S}
S_j^p = \mathrm{Id}_{\LL_{j,\mathrm{abs}}^p} - \mathrm{Id}_{\LL_{j,\mathrm{abs}}^{p, \perp}} .
\end{equation}
We identify $\hhe^p(Y,F)$ with $\hhe^p(Y,F)du$ via $du\wedge$.
Then $S_j^p$ also acts on $\hhe^p(Y,F)du$.
By \eqref{eq-ch2-LL-C},
\eqref{eq-ch2-def-LL-bd} and
\eqref{eq-ch2-def-LL-bd-2},
we have
\begin{equation}
\label{eq-ch5-C-S}
C_j^p =
\left(
\begin{array}{cc}
S_j^p & 0 \\
0     & - S_j^{p-1}
\end{array} \right).
\end{equation}

We equip the exact sequence \eqref{eq-ch4-mv-LL} with the metrics
$\big\lVert\,\cdot\,\big\rVert_{\LL_{1,\mathrm{rel}}^\bullet}$,
$\big\lVert\,\cdot\,\big\rVert_{\LL_{2,\mathrm{abs}}^\bullet}$ and
$\big\lVert\,\cdot\,\big\rVert_{\LL_1^\bullet \cap \LL_2^\bullet}$
in \textsection \ref{ch5-2}.
Let $T_\LL$ be its torsion.

\begin{prop}
\label{prop-ch5-TL}
The following identities hold,
\begin{align}
\label{eq-prop-ch5-TL}
\begin{split}
T_\LL = \; & \prod_{p=0}^n {\det}^\ast\Big(\frac{2-S^p_{1}\circ S^p_{2}-S^p_{2}\circ S^p_{1}}{4}\Big)^{\frac{1}{4}(-1)^p} \\
= \; & \prod_{p=0}^n {\det}^\ast\Big(\frac{2-C^p_{12}-(C^p_{12})^{-1}}{4}\Big)^{\frac{1}{4}(-1)^pp} .
\end{split}
\end{align}
\end{prop}
\begin{proof}
The exact sequence \eqref{eq-ch4-mv-LL} is an orthogonal sum of the following exact sequences,
\begin{align}
\label{eq1-pf-prop-ch5-TL}
\begin{split}
& 0 \rightarrow
\LL^p_{1,\mathrm{rel}}\cap\LL^p_{2, \mathrm{rel}} \rightarrow
\LL_1^p\cap\LL_2^p \rightarrow
\LL^p_{1,\mathrm{abs}}\cap\LL^p_{2, \mathrm{abs}} \rightarrow
0, \\
& 0 \rightarrow
\LL^p_{2,\mathrm{abs}}\cap(\LL^p_{1, \mathrm{abs}}\cap \LL^p_{2, \mathrm{abs}})^{\perp} \overset{\delta_{p,\LL}} \longrightarrow
\LL^{p+1}_{1,\mathrm{rel}} \cap(\LL^{p+1}_{1,\mathrm{rel}}\cap \LL^{p+1}_{2,\mathrm{rel}})^{\perp} \rightarrow
0.
\end{split}
\end{align}
By \eqref{eq-ch5-t1}-\eqref{eq-ch5-t3},
the torsion of the first row in \eqref{eq1-pf-prop-ch5-TL} vanishes.
We turn to study the second row in \eqref{eq1-pf-prop-ch5-TL}.
By \eqref{eq-ch2-def-LL-bd-2},
we have the following commutative diagram
whose vertical maps are bijective,
\begin{align}
\label{eq2-pf-prop-ch5-TL}
\begin{split}
\xymatrix{
0 \ar[r] &
\LL^p_{2, \mathrm{abs}}\cap(\LL^p_{1, \mathrm{abs}}\cap \LL^p_{2, \mathrm{abs}})^{\perp} \ar[r]\ar[d]^{\mathrm{Id}} &
\LL^{p+1}_{1, \mathrm{rel}} \cap(\LL^{p+1}_{1, \mathrm{rel}}\cap \LL^{p+1}_{2, \mathrm{rel}})^{\perp} \ar[r]\ar[d]^{\cntrtu}  & 0  \\
0 \ar[r] &
\LL^p_{2, \mathrm{abs}}\cap(\LL^p_{1, \mathrm{abs}}\cap \LL^p_{2, \mathrm{abs}})^{\perp} \ar@{-->}[r] &
\LL^{p, \perp}_{1, \mathrm{abs}} \cap (\LL^{p, \perp}_{1, \mathrm{abs}} \cap \LL^{p, \perp}_{2,\mathrm{abs}})^{\perp} \ar[r] & 0  \;.
}
\end{split}
\end{align}
Let
\begin{equation}
\label{eq3-pf-prop-ch5-TL}
P_p : \; \hhe^p(Y,F) \rightarrow \LL^p_{2, \mathrm{abs}}, \quad
Q_p : \; \hhe^p(Y,F) \rightarrow \LL^{p,\perp}_{1, \mathrm{abs}}
\end{equation}
be orthogonal projections.
We have
\begin{equation}
\label{eq4-pf-prop-ch5-TL}
\LL^p_{2, \mathrm{abs}}\cap(\LL^p_{1, \mathrm{abs}}\cap \LL^p_{2, \mathrm{abs}})^{\perp} = \im(P_pQ_p), \quad
\LL^{p, \perp}_{1, \mathrm{abs}} \cap (\LL^{p, \perp}_{1, \mathrm{abs}} \cap \LL^{p, \perp}_{2,\mathrm{abs}})^{\perp} = \im(P_pQ_p).
\end{equation}
Moreover,
by the definition of $\delta_{p,\LL}$
(see \eqref{eq-ch4-def-delta}),
the second row in \eqref{eq2-pf-prop-ch5-TL} is given by
\begin{equation}
\label{eq5-pf-prop-ch5-TL}
0 \rightarrow
\im(P_pQ_p) \overset{Q_p} \longrightarrow
\im(Q_pP_p) \rightarrow 0.
\end{equation}
By Lemma \ref{lab-prop-torsion-proj} in Appendix
and \eqref{eq2-pf-prop-ch5-TL}-\eqref{eq5-pf-prop-ch5-TL},
the torsion of the second row in \eqref{eq1-pf-prop-ch5-TL} is
\begin{equation}
\label{eq6-pf-prop-ch5-TL}
{\det}^\ast\big(1-P_p-Q_p+P_pQ_p+Q_pP_p\big)^\frac{1}{4}.
\end{equation}
By \eqref{eq-ch5-def-S},
we have
\begin{equation}
\label{eq7-pf-prop-ch5-TL}
P_p = \frac{1}{2}(1+S^p_2), \quad
Q_p = \frac{1}{2}(1-S^p_1).
\end{equation}
By \eqref{eq-ch5-t1}, \eqref{eq-ch5-t2},
\eqref{eq6-pf-prop-ch5-TL}
and \eqref{eq7-pf-prop-ch5-TL},
we get the first line in \eqref{eq-prop-ch5-TL}.

We denote
\begin{equation}
\widehat{I}_{p} = {\det}^\ast\big(\frac{2-S^p_1\circ S^p_2-S^p_2\circ S^p_1}{4}\big)^{\frac{1}{4}}, \quad
I_p = {\det}^\ast\big(\frac{2-C^p_{12}-(C^p_{12})^{-1}}{4}\big)^{\frac{1}{4}}.
\end{equation}
By \eqref{eq-ch5-C-S},
we have $I_p = \widehat{I}_{p}\cdot \widehat{I}_{p+1}$,
which leads to the second line in \eqref{eq-prop-ch5-TL}.
The proof of Proposition \ref{prop-ch5-TL} is completed.
\end{proof}

\begin{proof}[Proof of Theorem \ref{thm-ch0-2}]
We start by equipping the vector spaces in \eqref{eq-ch5-filt-mv-harmonic} with metrics.
All the metrics mentioned bellow are defined in \textsection \ref{ch5-2}, \ref{ch5-3}.
\begin{itemize}
\item[-] The vector spaces in the second (resp. first) row are equipped with (resp. the restrictions of)
the metrics $\big\lVert\,\cdot\,\big\rVert_{\mathscr{H},R}$.
\item[-] The vector spaces in the third row are equipped with the metrics
$\big\lVert\,\cdot\,\big\rVert_{\LL^\bullet_{1,\mathrm{rel}}}$,
$\big\lVert\,\cdot\,\big\rVert_{\LL^\bullet_{2,\mathrm{abs}}}$ and
$\big\lVert\,\cdot\,\big\rVert_{\LL^\bullet_1 \cap \LL^\bullet_2}$.
\end{itemize}

For $i=1,2,3$,
we denote by $T_{i,\bullet}$ the torsion of the $i$-th row in \eqref{eq-ch5-filt-mv-harmonic}.
For $j=1,\cdots,3n+3$,
we denote by $T_{\bullet,j}$ be the torsion of the $j$-th column in \eqref{eq-ch5-filt-mv-harmonic}.
By \cite[Theorem 1.20]{BGS1},
we have
\begin{equation}
\label{eq1-pf-thm-ch0-2}
T_{1,\bullet}\,T_{2,\bullet}^{-1}\,T_{3,\bullet} = \prod_{k=1}^{3n+3}T_{\bullet,k}^{(-1)^{k+1}}.
\end{equation}
By Proposition \ref{prop-ch5-metric-L}
and \eqref{eq-ch5-t1}-\eqref{eq-ch5-t3},
we have
\begin{align}
\begin{split}
T_{\bullet,3p+1} & = \big( 1 + \mathscr{O}\big(R^{-1}\big) \big) R^{\frac{1}{2}\dim \LL_{1,\mathrm{rel}}^p}, \\
T_{\bullet,3p+2} & = \big( 1 + \mathscr{O}\big(R^{-1}\big) \big) (2R)^{\frac{1}{2}\dim \LL_1^p \cap \LL_2^p}, \\
T_{\bullet,3p+3} & = \big( 1 + \mathscr{O}\big(R^{-1}\big) \big) R^{\frac{1}{2}\dim \LL_{2,\mathrm{abs}}^p}.
\end{split}
\end{align}
By \eqref{eq-ch5-t1}-\eqref{eq-ch5-t3},
we have
\begin{equation}
T_{1,\bullet} = 1.
\end{equation}
By Proposition \ref{prop-ch5-metric-H}
and \eqref{eq-ch5-t1}-\eqref{eq-ch5-t3},
we have
\begin{equation}
T_{2,\bullet} = \big( 1 + \mathscr{O}\big(e^{-aR}\big) \big)  \mathscr{T}_R.
\end{equation}
By Proposition \ref{prop-ch5-map-L}
and \eqref{eq-ch5-t1}-\eqref{eq-ch5-t3},
we have
\begin{equation}
T_{3,\bullet} = \big( 1 + \mathscr{O}\big(R^{-1}\big) \big)
\Big( \prod_{p=0}^n 2^{ (-1)^p \dim\im(\alpha_{p,\LL}) } \Big)
\Big( \prod_{p=0}^n R^{ (-1)^p \dim\im(\delta_{p,\LL}) } \Big)
T_\LL.
\end{equation}
It follows from the exactness of \eqref{eq-ch4-mv-LL} that
\begin{align}
\label{eq6-pf-thm-ch0-2}
\begin{split}
& \dim \LL_1^p\cap\LL_2^p = \dim \im (\alpha_{p,\LL}) + \dim \im (\beta_{p,\LL}), \\
& \sum_{p=0}^n (-1)^p \big(
\dim \LL_{1,\mathrm{rel}}^p
- \dim \LL_1^p \cap \LL_2^p
+ \dim \LL_{2,\mathrm{abs}}^p \big) = 0 .
\end{split}
\end{align}
By Lemma \ref{lem-ch4-chi-abd},
Proposition \ref{prop-ch5-TL}
and \eqref{eq1-pf-thm-ch0-2}-\eqref{eq6-pf-thm-ch0-2},
we get \eqref{eq-thm-ch0-2}.
The proof of Theorem \ref{thm-ch0-2} is completed.
\end{proof}

\subsection{Anomaly formulas and a proof of Theorem \ref{thm-ch0-3}}
\label{ch5-5}

Here we adopt the notations in Section \ref{ch1-1}.
Moreover, we assume that the Riemannian metric $g^{TX}$ is product near
$Y = \partial X$.
Let $\n^{TX}$ be the Levi-Civita connection on $TX$.  Let $R^{TX}=\left(\n^{TX}\right)^2$ be its curvature.
Let $o(TX)$ be the orientation bundle of $TX$.
The Euler form (cf. \cite[(4.9)]{BZ92}) is defined as
\begin{equation}
e(TX,\n^{TX}) = \mathrm{Pf}\Big[\frac{R^{TX}}{2\pi}\Big] \in \Omega^{\dim X}\big(X,o(TX)\big) .
\end{equation}

Let ${g^{TX}}'$ be another Riemannian metric on $TX$.
Let ${\n^{TX}}'$ be the associated Levi-Civita connection.
We assume that
$g^{TX}$ and ${g^{TX}}'$ coincide on a neighborhood of $Y = \partial X$.
For $s\in[0,1]$,
set $g^{TX}_s = (1-s)g^{TX}+s{g^{TX}}'$.
Let $\n^{TX}_s$ be the Levi-Civita connection associated with $g^{TX}_s$.
Set
\begin{align}
\begin{split}
& \widetilde{e}\big(TX,\n^{TX},{\n^{TX}}'\big) \\
& = \int_0^1
\left\{\derpar{}{b}\Big|_{b=0} \mathrm{Pf}\left[ \frac{1}{2\pi}\left(\n^{TX}_s\right)^2 +
\frac{b}{2\pi}\left( \derpar{}{s}\n^{TX}_s - \frac{1}{2}\left[\n^{TX}_s,\big(g^{TX}_s\big)^{-1}\derpar{}{s}g^{TX}_s\right] \right) \right] \right\}
ds.
\end{split}
\end{align}
We remark that
$\widetilde{e}\big(TX,\n^{TX},{\n^{TX}}'\big)$
vanishes near $Y=\partial X$.
By \cite[(4.10)]{BZ92},
we have
\begin{equation}
d \, \widetilde{e}\big(TX,\n^{TX},{\n^{TX}}'\big)
= e(TX,{\n^{TX}}') - e(TX,\n^{TX}).
\end{equation}

We define a closed one-form (see \cite[Prop. 4.6]{BZ92})
\begin{equation}
\theta(F,h^F) = \tr\big[\left(h^F\right)^{-1}\n^Fh^F\big] \in \Omega^1(X).
\end{equation}

Let $\big\lVert\,\cdot\,\big\rVert^\mathrm{RS}_{\det H^\bullet_\bd(X,F)}$
and ${\big\lVert\,\cdot\,\big\rVert^\mathrm{RS}}'_{\hspace{-4mm}\det H^\bullet_\bd(X,F)}$
be the Ray-Singer metrics in Defitition \ref{def-ch1-rs-metric}
associated with $g^{TX}$ and ${g^{TX}}'$.
The following theorem is a special case of \cite[Theorem 0.1]{BruMa06}.

\begin{thm}
\label{thm-ch5-anomaly}
We have
\begin{equation}
\log \left( \frac{\,{\big\lVert\,\cdot\,\big\rVert^\mathrm{RS}}'_{\hspace{-4mm}\det H^\bullet_\bd(X,F)}}
{\big\lVert\,\cdot\,\big\rVert^\mathrm{RS}_{\det H^\bullet_\bd(X,F)}} \right)
=
- \frac{1}{2} \int_X \theta(F,h^F) \, \widetilde{e}\big(TX,\n^{TX},{\n^{TX}}'\big).
\end{equation}
\end{thm}

\begin{proof}[Proof of Theorem \ref{thm-ch0-3}]
We only show that $\big\lVert\rho_R\big\rVert^\mathrm{RS}_{\lambda_R(F)}$ is independent of $R$.
The rest of the proof is explained in the paragraph containing \eqref{eq-thm-ch0-3-equiv}.

Let $0 < R < R'$.
Let $\phi : [-R,R] \rightarrow [-R',R']$ be a smooth function such that
\begin{align}
\begin{split}
& \phi'>0, \quad
\phi(0) = 0, \quad
\phi(\pm R) = \pm R', \\
& \phi'(u) = 1 \quad \text{for } u\in [-R,-2R/3] \,\cup\, [-R/3,R/3] \,\cup\, [2R/3,R].
\end{split}
\end{align}
Let $\varphi_0 : Z_R \rightarrow Z_{R'}$ be the diffeomorphism defined by $\phi$
in the same way as \eqref{eq-ch3-def-varphiR}.
For $j=1,2$,
set $\varphi_j = \varphi_0\big|_{Z_{j,R}} : Z_{j,R} \rightarrow Z_{j,R'}$,
which is also a diffeomorphism.

We will use the convention
$Z_{0,R}=Z_R$ and
$H^\bullet_\mathrm{bd}(Z_{0,R},F) = H^\bullet(Z_R,F)$.
Let
\begin{equation}
\varphi^{\det}_j: \det H^\bullet_\mathrm{bd}(Z_{j,R},F) \rightarrow \det H^\bullet_\mathrm{bd}(Z_{j,R'},F),
\quad j=0,1,2
\end{equation}
be the isomorphism induced by $\varphi_j$.
Since $\varphi_j$ is isometric near $\partial Z_{j,R}$,
by Theorem \ref{thm-ch5-anomaly},
we have
\begin{equation}
\log\left(\frac{\big\lVert\rho_{R'}\big\rVert^\mathrm{RS}_{\lambda_{R'}(F)}}
{\big\lVert\rho_R\big\rVert^\mathrm{RS}_{\lambda_R(F)}}\right)
= \sum_{j=0}^2 (-1)^{(j-1)(j-2)/2}
\log\left(\frac{\big\lVert\varphi_j^{\det}(\cdot)\big\rVert^\mathrm{RS}_{\det H^\bullet_\mathrm{bd}(Z_{j,R'},F)}}
{\big\lVert\,\cdot\,\big\rVert^\mathrm{RS}_{\det H^\bullet_\mathrm{bd}(Z_{j,R},F)}}\right)
= 0.
\end{equation}
The proof of Theorem \ref{thm-ch0-3} is completed.
\end{proof}


\section{Appendix}
\label{ch6}

The appendix contains several technical results.

Let $A : V\rightarrow W$ be a linear map between Hermitian vector spaces of the same dimension.
We denote by $\det(A)$ the determinant of the matrix of $A$ under any orthogonal bases,
which is well-defined up to $\mathrm{U}(1):=\big\{z\in\C\;:\;|z|=1\big\}$.

\begin{lemma}
\label{lab-prop-torsion-proj}
Let $V$ be a Hermitian vector space.
Let $H_1,H_2\subseteq V$ be vector subspaces.
Let $P_j\in\mathrm{End}(V)$ ($j=1,2$) be the orthogonal projection to $H_j$.
We have
\be\label{e.347}
 \left| {\det}(P_1|_{\mathrm{Im}(P_2P_1)}) \right| =&
\left| {\det}(P_2|_{\mathrm{Im}(P_1P_2)}) \right| \\
=& {\det}^\ast\big(\mathrm{Id} - P_1 - P_2 + P_1P_2 + P_2P_1\big)^\frac{1}{4} .
\end{aligned}\end{align}
\end{lemma}
\begin{proof}
We claim that
there exists an orthogonal decomposition $V=\bigoplus_k V_k$ such that
$\dim V_k \leqslant 2$ and $H_j=\bigoplus_k \left(V_k \cap H_j\right)$ for $j=1,2$.
Once the claim is proved, we may suppose that $\dim V\leqslant 2$.
Then the only non trivial case is that $\dim V=2$ and $\dim H_1 = \dim H_2 = 1$.
We may suppose that
\be\label{e.348}
V=\C^2 ,\hspace{5mm}
H_1 = \C \left(
\begin{array}{c}
\hspace{-1mm} 1 \hspace{-1mm} \\
\hspace{-1mm} 0 \hspace{-1mm}
\end{array} \right), \hspace{5mm}
H_2 = \C \left(
\begin{array}{c}
\hspace{-1mm} \cos\theta \hspace{-1mm} \\
\hspace{-1mm} \sin\theta \hspace{-1mm}
\end{array} \right), \hspace{5mm}
\text{with } 0\leqslant\theta\leqslant\frac{\pi}{2} .
\end{aligned}\end{align}
We have
$\left| {\det}(P_1|_{\mathrm{Im}(P_2P_1)}) \right| =
\left| {\det}(P_2|_{\mathrm{Im}(P_1P_2)}) \right| = \cos\theta$, and
\be\label{e.349}
P_1 = \left(
\begin{array}{cc}
1 & 0 \\
0 & 0
\end{array} \right) ,\hspace{5mm}
P_2 = \left(
\begin{array}{cc}
\cos^2\theta & \cos\theta\sin\theta \\
\cos\theta\sin\theta & \sin^2\theta
\end{array} \right) .
\end{aligned}\end{align}
Then \eqref{e.347} follows from a direct calculation.

Now we prove the claim.
The operator
$P_1P_2P_1\big|_{H_1}$ and $P_2P_1P_2\big|_{H_2}$
are self-adjoint.
Let
\begin{equation}
H_1 = \bigoplus_{0\leqslant\lambda\leqslant 1} H_1^\lambda, \quad
H_2 = \bigoplus_{0\leqslant\lambda\leqslant 1} H_2^\lambda
\end{equation}
be the associated spectral decomposition,
i.e.,
$P_1P_2P_1\big|_{H_1^\lambda} = \lambda\mathrm{Id}$ and
$P_2P_1P_2\big|_{H_2^\lambda} = \lambda\mathrm{Id}$.
In particular,
$H_1^1 = H_2^1 = H_1\cap H_2 $, $H_1^0 = H_1 \cap H_2^\perp $, $H_2^0 = H_2 \cap H_1^\perp$.
We get the orthogonal decomposition
\be\label{e.353}
V = (H_1+H_2)^\perp \oplus (H_1\cap H_2) \oplus (H_1\cap H_2^\perp) \oplus (H_2\cap H_1^\perp) \oplus \bigoplus_{{0<\lambda<1}}(H_1^\lambda+H_2^\lambda) ,
\end{aligned}\end{align}
which is invariant under the actions of $P_1$ and $P_2$,
as for any $y\in H_2^\lambda$,
\begin{equation}
P_1P_2P_1(P_1y) = P_1P_2P_1(P_1P_2y) = P_1(P_2P_1P_2y) = \lambda P_1y.
\end{equation}
Thus $P_1H_2^\lambda \subseteq H_1^\lambda$.
The problem is reduced to each block.
The claim is trivial for
$H_1\cap H_2$, $(H_1+H_2)^\perp$, $H_1 \cap H_2^\perp$ and $H_2 \cap H_1^\perp$.
Now we suppose that $H = H_1^\lambda+H_2^\lambda$ with $0<\lambda<1$.
Let $\big(v_k\big)_{1\leqslant k\leqslant r}$ be an orthogonal basis of $H_1^\lambda$.
Let $V_k\subseteq H$ be the vector subspace spanned by $\{v_k,P_2v_k\}$.
Then $\big(V_k\big)_{1\leqslant k\leqslant r}$ satisfies the desired properties.
The proof of Proposition \ref{lab-prop-torsion-proj} is completed.
\end{proof}


Let $(V,\langle\cdot,\cdot\rangle)$ be a Hermitian vector space of dimension $m$.
Let $(a,b)\subseteq\R$ be an open interval.
Let $C: (a,b) \rightarrow \End(V)$ be an analytic function
such that $C(\lambda)$ is unitary for any $\lambda\in (a,b)$.
The following theorem comes from  \cite[\textsection 2.6, Theorem 6.1]{Kato95}.

\begin{thm}
\label{lab-thm-holomophic-eigenvalue-c-matrix}
There exist real analytic functions
$\theta_1(\lambda), \cdots, \theta_m(\lambda)$
such that
\begin{equation}
\big\{e^{i\theta_1(\lambda)},\cdots,e^{i\theta_m(\lambda)}\big\}
= \Sp\big(C(\lambda)\big).
\end{equation}
Moreover,
there exist orthogonal projections
$P_1(\lambda),\cdots, P_m(\lambda)\in\End(V)$
depending analytically on $\lambda$ such that
\begin{align}
\label{eq-thm-app-C}
\begin{split}
& \mathrm{Id} = P_1(\lambda) + \cdots + P_m(\lambda), \quad
P_j(\lambda)P_k(\lambda) = 0, \quad \text{for } j\neq k,\\
& C(\lambda) = e^{i\theta_1(\lambda)} P_1(\lambda) + \cdots + e^{i\theta_m(\lambda)} P_m(\lambda).
\end{split}
\end{align}
\end{thm}

For $R>2$,
we consider the equation
\begin{equation}
\label{eq-app-e-C}
e^{4iR\lambda}C(\lambda)v = v,
\end{equation}
where $\lambda\in (a,b)$ and $v\in V$.
By Theorem \ref{lab-thm-holomophic-eigenvalue-c-matrix},
for $R$ and $\lambda$ fixed,
\eqref{eq-app-e-C}, viewed as an equation of $v$,
admits a non trivial solution
if and only if one of $4R\lambda+\theta_1(\lambda),\cdots,4R\lambda+\theta_m(\lambda)$ lies in $2\pi\Z$.
For $R>2$, set
\begin{align}
\begin{split}
\Lambda_R(C)  & =
\big\{ \rho > 0 :\; \det \big( e^{4iR\rho}C(\rho) - \mathrm{Id} \big) = 0 \big\}, \quad \\
\Lambda^*_R(C) & =
\big\{ \lambda > 0:\; \det \big( e^{4iR\lambda}C(0) - \mathrm{Id} \big) = 0 \big\}.
\end{split}
\end{align}

\begin{prop}
\label{lab-prop-Lambda-Lambda-star-estimation}
We fix $\kappa>0$.
There exists $a>0$ such that
for $R^{-1+\kappa}\leqslant\gamma\leqslant1$
and $f\in\mathscr{C}^1(\R)$,
\be\label{e.406}
\Big| \sum_{\rho\in\Lambda_R(C),\;|\rho|<\gamma} f(\rho)
- \sum_{\lambda\in\Lambda^*_R(C),\;|\lambda|<\gamma} f(\lambda) \Big|
\leqslant a \gamma^2 \sup_{|x|\leqslant\gamma} |f'(x)|
+ a \gamma \sup_{|x|\leqslant\gamma} |f(x)|  .
\end{aligned}\end{align}
\end{prop}
\begin{proof}
By Theorem \ref{lab-thm-holomophic-eigenvalue-c-matrix},
we may suppose that $V=\C$ and $C(\rho)=e^{i\alpha(\rho)}$,
where $\rho\mapsto\alpha(\rho)$ is an analytic function.
The rest of the proof is a direct estimate.
We leave it to the reader.
\end{proof}

Set
\begin{equation}
\theta_{C,R}(s) = - \sum_{\lambda\in\Lambda^*_R(C)}\lambda^{-2s}.
\end{equation}

\begin{prop}
\label{prop-Lambdazeta-der0}
We assume that $\Sp\big(C(0)\big) = \overline{\Sp\big(C(0)\big)}$.
We denote $r=\dim \Ker(C(0)-\Id)$.
We have
\begin{equation}
\label{eq-Lambdazeta-der0}
\theta_{C,R}{}'(0) = r \log (2R) +\log 2\cdot\dim V +
\frac{1}{2} \log {\det}^* \Big( \frac{2- C(0) - {C(0)}^{-1}}{4} \Big).
\end{equation}
\end{prop}
\begin{proof}
As a special case of the Hurwitz zeta function (cf. \cite[\textsection 7]{Weil99}), we have
\begin{equation}
\label{lab-eq-pf-Lambdazeta-der0}
- \derpar{}{s}\Big|_{s=0} \sum_{k=1}^{+\infty} \Big(\frac{2\pi k - \alpha}{4R}\Big)^{-2s}
= \left\{
\begin{array}{rl}
\log (4R)\quad & \text{for } \alpha = 0 , \\
\frac{1}{2} \log (2-2\cos\alpha) & \text{for } 0<\alpha \leqslant \pi .
\end{array} \right.
\end{equation}

Since $C(0)$ is diagonalizable,
it suffices to consider the following cases.

Case 1. $m=1$, $r=1$, $C(0)=1$,
then \eqref{eq-Lambdazeta-der0} is equivalent to \eqref{lab-eq-pf-Lambdazeta-der0} with $\alpha = 0$.

Case 2. $m=1$, $r=0$, $C(0)=-1$,
then \eqref{eq-Lambdazeta-der0} is equivalent to \eqref{lab-eq-pf-Lambdazeta-der0} with $\alpha = \pi$.

Case 3. $m=2$, $r=0$, $\Sp(C(0))=\{e^{i\alpha},e^{-i\alpha}\}$ with $\alpha\in\,(0,\pi)$,
then \eqref{eq-Lambdazeta-der0} is equivalent to \eqref{lab-eq-pf-Lambdazeta-der0} with $\alpha\in\,(0,\pi)$.

This completes the proof of Proposition \ref{prop-Lambdazeta-der0}.
\end{proof}

\bibliographystyle{plain}

\end{document}